\documentclass[a4paper,10pt]{amsart}

\usepackage{times,amsmath,amssymb,amsthm}
\usepackage{graphicx}
\usepackage{a4wide}
\usepackage[all]{xy}
\usepackage{paralist}
\usepackage{extpfeil}
\usepackage{subfigure}
\usepackage{multicol}
\usepackage{paralist}

\usepackage[T1]{fontenc}

\usepackage{tikz}
\usetikzlibrary{shapes,fit,positioning,calc,matrix}
\tikzset{
  optree/.style={scale=.5,thick,grow'=up,level distance=10mm,inner sep=1pt},
  comp/.style={draw=none,circle,fill,line width=0,inner sep=0pt},
  dot/.style={draw,circle,fill,inner sep=0pt,minimum width=3pt},
  circ/.style={draw,circle,inner sep=1pt,minimum width=4mm},
  emptycirc/.style={draw,circle,inner sep=1pt,minimum width=2mm},
  root/.style={level distance=10mm,inner sep=1pt},
  leaf/.style={draw=none,circle,fill,line width=0,inner sep=0pt},
  nodot/.style={draw,circle,inner sep=1pt},
}

\DeclareFontFamily{U}{shuffle}{} 
\DeclareFontShape{U}{shuffle}{m}{n}{<5-8>shuffle7  <8->shuffle10}{}
\DeclareSymbolFont{Shuffle}{U}{shuffle}{m}{n}
\DeclareMathSymbol\shuffle{\mathbin}{Shuffle}{"001} 
\DeclareMathSymbol\cshuffle{\mathbin}{Shuffle}{"002}

\newcommand{\cev}[1]{\reflectbox{\ensuremath{\vec{\reflectbox{\ensuremath{#1}}}}}}

\usepackage{color}
\definecolor{Chocolat}{rgb}{0.36, 0.2, 0.09}
\definecolor{BleuTresFonce}{rgb}{0.215, 0.215, 0.36}

\usepackage[colorlinks,final,hyperindex]{hyperref}
\hypersetup{citecolor=BleuTresFonce, linkcolor=Chocolat}

\newcommand{\pullback}{\mbox{\LARGE{$\lrcorner$}}}
\newcommand{\dif}{\mathrm{d}}
\newcommand{\tig}{\tilde{g}}
\newcommand{\tih}{\tilde{h}}
\newcommand{\GG}{\mathbb{G}}
\newcommand{\oGG}{\overline{\mathbb{G}}}
\newcommand{\KDO}{\calO^{\ac}}

\newcommand{\Tw}{\mathrm{Tw}}
\newcommand{\epi}{\twoheadrightarrow}

\newcommand{\mono}{\rightarrowtail}

\def\Ker{\mathrm{Ker}}
\newcommand{\NN}{\mathbb{N}}
\newcommand{\Sy}{\mathbb{S}}
\newcommand{\si}{\sigma}
\def\KK{\mathbb{K}}
\def\B{\mathrm{B}}
\newcommand{\ac}{{\scriptstyle \text{\rm !`}}}

\def\TTT{\mathcal{T}}

\def\qi{\xrightarrow{\sim}}

\def\C{\mathcal{C}}
\def\P{\mathcal{P}}
\def\oP{\overline{\mathcal{P}}}
\newcommand{\opd}[1]{\mathcal{#1}}
\newcommand{\nsOpd}{\mathcal{O}_\textrm{ns}}

\newcommand{\htensor}{\otimes_H}
\newcommand{\N}{\mathbb{N}}

\newcommand{\I}{\mathrm{I}}

\def\calC{\mathcal{C}}

\def\calE{\mathcal{E}}

\def\calH{\mathcal{H}}

\def\calO{\mathcal{O}}
\def\calP{\mathcal{P}}
\def\calQ{\mathcal{Q}}

\def\calS{\mathcal{S}}
\def\calT{\mathcal{T}}

\DeclareMathOperator{\id}{id}
\DeclareMathOperator{\End}{End}
\DeclareMathOperator{\Hom}{Hom}

\newcommand{\Y}{\vcenter{\xymatrix@M=0pt@R=4pt@C=4pt{
\ar@{-}[dr] &  &\ar@{-}[dl]  \\
 &*{} \ar@{-}[d] &  \\  & &}}}

\newcommand{\Yb}{\vcenter{\xymatrix@M=0pt@R=4pt@C=4pt{
\ar@{-}[dr] &  &\ar@{-}[dl]  \\
 &* {\scriptstyle (12)} \ar@{-}[d] &
    \\  & &}}}

\newcommand{\YYl}{\vcenter{\xymatrix@M=0pt@R=4pt@C=4pt{
\ar@{-}[dr] &  &\ar@{-}[dl]  & \\
 &*{} \ar@{-}[dr] &  &\ar@{-}[dl]\\  & &*{}\ar@{-}[d]&\\ &&& }}}

\newcommand{\YYr}{\vcenter{\xymatrix@M=0pt@R=4pt@C=4pt{
& \ar@{-}[dr] &  &\ar@{-}[dl]  \\
\ar@{-}[dr] &&*{} \ar@{-}[dl] &\\  & *{}\ar@{-}[d]&&\\ &&& }}}

\newcommand{\W}{\vcenter{\xymatrix@M=0pt@R=4pt@C=4pt{
\ar@{-}[dr] &*{} \ar@{-}[d] &\ar@{-}[dl]  \\
 &*{} \ar@{-}[d] &  \\  & &}}}

\makeatletter

\allowdisplaybreaks

\theoremstyle{plain}

\newtheorem {theorem}{Theorem}
\newtheorem {lemma}{Lemma}
\newtheorem {corollary}{Corollary}

\newtheorem {proposition}{Proposition}
\newtheorem {definition-proposition}{Definition-Proposition}
\newtheorem{theoremintro}{Theorem}
\newtheorem*{theoremintro5}{Theorem\! 5}

\theoremstyle{definition}

\newtheorem {definition}{Definition}

\theoremstyle{remark}

\newtheorem{example}{\sc Example}
\newtheorem{remark}{\sc Remark}
\newtheorem*{further}{\sc Further study}
\newtheorem*{remark*}{\sc Remark}
\newtheorem*{remarks*}{\sc Remarks}

\subjclass[2010]{Primary 18D50; Secondary 18G55}

\keywords{Homotopical algebra, Operad, Koszul duality, $E_\infty$-operad}

\thanks{ M.D. and B.V. were supported  by the ANR grants HOGT and SAT}

\title[Symmetric Homotopy Theory for operads]{Symmetric Homotopy Theory for operads}

\author{Malte Dehling}
\address{Mathematisches Institut, 
Georg-August-Universit\"at G\"ottingen,
Bunsenstra\ss e 3-5, 
37073 G\"ottingen, Germany}
\email{mdehling@uni-math.gwdg.de}

\author{Bruno Vallette}
\address{Laboratoire J.A.Dieudonn\'e,
Universit\'e de Nice Sophia-Antipolis, Parc Valrose,
06108 Nice Cedex 02, France}
\email{brunov@unice.fr}

\date{\today}

\begin{document}

\begin{abstract}
The purpose of this foundational paper is to introduce various notions and constructions in order to develop the homotopy theory for differential graded operads over any ring. The main new idea is to consider the action of the symmetric groups as part of the defining structure of an operad and not as the underlying category. We introduce a  new dual category of higher cooperads, a new higher bar-cobar adjunction with the category of operads, and a new  higher notion of homotopy operads, for which  we establish the relevant homotopy properties. For instance, the higher bar-cobar construction provides us with a  cofibrant replacement functor for operads over any ring. All these constructions are produced conceptually by applying the curved Koszul duality for colored operads. This paper is a first step toward a new Koszul duality theory for operads, where the action of the symmetric groups is properly taken into account. 
\end{abstract}

\maketitle

\tableofcontents

\section*{Introduction}

An operad is an algebraic device which encodes  operations with multiple inputs acting on categories of algebras like associative algebras, Lie algebras, commutative algebras, etc. They are used to govern the relations and the symmetries of these operations.
The two main breakthroughs of operad theory are the theory of  $E_\infty$-operads and the Koszul duality theory. On the one hand, the notion of an $E_\infty$-operad emerged from the theory of iterated loop spaces \cite{May72, BoardmanVogt73}. 
An algebra over an $E_\infty$-operad, called an $E_\infty$-algebra, is a generalization of the notion of commutative algebra where \textit{everything}  is relaxed up to homotopy, that is both the associativity relation and the symmetry of the product. 
Such a rich algebraic structure is present on the level of the singular cochain complex of a topological space \cite{Mandell06} and this creates  enough algebraic invariants to faithfully detect its  homotopy type.
On the other hand, the Koszul duality theory \cite{GinzburgKapranov94, GetzlerJones94} provides us with many tools to work out the homological algebra of operads and hence the homotopy theory of algebras over  operads. 
Applied to the operad encoding associative algebras, it gives the notion of an $A_\infty$-algebra, related to loop spaces,  due to J. Stasheff \cite{Stasheff63} and the bar construction of Eilenberg--MacLane \cite{EilenbergMacLane53}. Applied to the operad encoding Lie algebras, it gives the notion of an $L_\infty$-algebra which plays a crucial role in deformation theory \cite{Kontsevich03, Getzler09} and the Chevalley--Eilenberg complex \cite{ChevalleyEilenberg48}. 
Applied to the operad encoding commutative algebras, it gives the notion of a $C_\infty$-algebra, present in rational homotopy theory \cite{Quillen69, Sullivan77}, and the Harrison complex \cite{Harrison62}. 
So it works well in general in characteristic $0$, but sometimes fails to apply outside that case. For instance, a $C_\infty$-algebra is an $E_\infty$-algebra only in characteristic $0$, precisely because the symmetry of te product is not relaxed up to homotopy.  

\smallskip

With the increasing number of new applications of operad theory found recently in algebraic topology \cite{Mandell06, Fresse08} and in algebraic geometry \cite{Kapranov98, MoerdijkToen10, GRSA12, Lurie12} for instance, there is a need to 
have at hand suitable tools to  manage the homotopy properties for differential graded operads over any ring. The main goal of the present paper is to develop such tools. For example, we would like to reconcile  $E_\infty$-operads  with the  Koszul duality theory as follows: produce $E_\infty$-operads over any ring by  applying a new kind of  Koszul duality theory. 
If one steps back and considers the classical approach of the Koszul duality theory of \cite{GinzburgKapranov94, GetzlerJones94}, one would see that it elaborates on the ground category of representations of the symmetric groups, called $\Sy$-modules. So, in order the get appropriate homotopy properties, one needs to be given  projective $\Sy$-modules, see \cite{Fresse04}. This is, of course, not always the case, as the example of the operad encoding commutative algebras shows: its underlying $\Sy$-module is never projective in positive characteristic.  
\smallskip

There are actually two ways to apprehend the notion of an operad.  One can start from an $\Sy$-module and consider composition maps acting on it. In this case, the symmetry of the operations are given, only their compositions are modeled. 
Or one can also start from  an $\NN$-module, that is a module graded by the arity number, and consider both the composition maps \textit{and} the symmetric groups action as the defining structure. This is the point of view taken in the present paper. We work over the ground category of differential graded $\NN$-modules, which is better behaved than that the category of differential graded $\Sy$-modules which respect to homological properties. 
For instance, the projective property is easier to satisfy there. We will then develop a theory similar to that of Ginzburg--Kapranov and Getzler--Jones but in this second context. 

\smallskip

To produce such a theory, we do not work by hand but instead we rather use the conceptual approach which consists in considering the colored operad $\calO$ whose category of algebras over $\NN$-modules is the category of non-unital, or equivalently, augmented operads. It turns out that this colored operad admits a natural inhomogenous presentation. Therefore, it fits into the framework of the recent  curved Koszul duality theory for (colored) operads settled by Hirsh--Mill\`es in 
in \cite{HirshMilles12}. The first result of the paper says that it is actually a Koszul colored operad. 
\begin{theoremintro}\label{thmintro:OKoszul}
The colored operad $\calO$ is an inhomogeneous curved Koszul operad. 
\end{theoremintro}

This is the key starting point of the following theory; this Koszul property assures that all the consequent constructions will be well behaved with respect to homological properties. 
We describe all the general constructions associated to Koszul colored operads in the case of the operad $\calO$, see \cite[Chapter~$11$]{LodayVallette12} for more details. This will produce the following notions. First, 
coalgebras over the Koszul dual colored cooperad $\calO^{\ac}$ give rise to a new category of \textit{higher cooperads}. Then, there is a new  adjunction, called the \textit{higher bar-cobar adjunction},
$$\xymatrix@C=30pt{{\widetilde\B \ : \ \mathsf{nu} \ \mathsf{Op}  \ }  \ar@_{->}@<-1ex>[r]  \ar@^{}[r]|(0.39){\perp}   & \ar@_{->}@<-1ex>[l]  {\ \mathsf{conil}\ \mathsf{higher}\ \mathsf{Coop}  \ : \ \widetilde{\Omega}}}\ .   $$
between (non-unital) operads and (conilpotent) higher cooperads. As expected, the counit of this adjunction induces a  cofibrant resolution functor. 
\begin{theoremintro}\label{thmintro:BarCobarRes}
The  higher bar-cobar counit 
$$\widetilde{\Omega} \widetilde{\B}\,  \P\ \qi\  \P$$ 
is a functorial resolution of non-unital dg operads, which is cofibrant when the underlying chain complex of $\P$ is non-negatively graded and made up of 
 projective $\KK$-modules.
\end{theoremintro}

Let us compare this higher bar-cobar construction with the classical bar-cobar adjunction. 
First, the intermediate category of the higher bar-cobar resolution is made up of  higher cooperads, which is 
more involved than the category of cooperads. The conceptual reason is quite natural: the higher bar construction $\widetilde\B\, \P$ needs to resolve up to homotopy and in a functorial way the symmetric groups action of the operad $\P$. 
As a result, while the  classical bar-cobar construction was cofibrant in characteristic $0$, or when the underlying $\Sy$-module of $\P$ is projective, the higher bar-cobar construction is always cofibrant over any ring providing that the underlying $\NN$-module is projective, which is a much weaker assumption.
\smallskip

Let us recall the existence of a particularly well behaved $E_\infty$-operad whose components are defined by the bar resolutions of the symmetric groups: the Barratt--Eccles operad $\calE$, see  \cite{BergerFresse04}. This operad allows us to relate the bar-cobar construction and the higher bar-cobar construction in the following precise way: there exists a natural isomorphism 
$$\I \oplus \widetilde{\Omega} \widetilde{\B} \, \P \cong \Omega \B \left( \P\otimes_H \calE\right) \ , $$
where $\P\otimes_H \calE$ stands for the aritywise tensor product (Proposition~\ref{prop:BarCobarBE}) and where $\I$ is the identity operad. 
It  was already known that the latter construction was giving a cofibrant replacement functor. On the bright side, one can interpret this result as a new universal property satisfied by  the Barratt--Eccles operad. 
\smallskip

Applied to the operad encoding commutative algebras, the higher bar-cobar resolution produces automatically a cofibrant $E_\infty$-operad over any ring, which is also a Hopf operad resolution. 
Under the above isomorphism, this result is not that new since it gives an operad isomorphic to the bar-cobar construction of the Barratt--Eccles operad. 
Notice that the existence of a cofibrant Hopf $E_\infty$-operad plays a crucial role in B. Fresse's works \cite{Fresse08, Fresse10bis}. 
However the actual description of the associated category of algebras does not seem to be  present in the literature, for instance in loc. cit.. 
Since we believe that it could  be useful to have  an explicit description of an $E_\infty$-algebra by means of generating operations and relations, to perform 
computations of homology groups in algebraic topology along the lines of \cite{Fresse08} for instance, we include it in the present paper. 
The advantage of our  model  is that it relies on simple combinatorics of planar trees and permutations. Therefore, this provides us with a ``simple'' explicit notion of an $E_\infty$-algebra (Proposition~\ref{prop:EinftyAlg}). 

\smallskip

The Koszul property of Theorem~\ref{thmintro:OKoszul} gives a cofibrant resolution of the colored operad $\calO$. So algebras over this resolution provide us with a new notion of \textit{higher homotopy operads} where both the relations for the partial composition products \textit{and} for the symmetric group actions are relaxed coherently up to homotopy. 
This very general notion  includes many of the existing homotopy algebraic notions 
as particular cases: $A_\infty$-algebras, $A_\infty$-modules,  homotopy nonsymmetric operads, for example.
It comes naturally equipped with a more general notion of \textit{$\infty$-morphism}, which is proved to be an efficient  tool to describe the homotopy properties of differential graded operads. For instance, the following homotopy transfer theorem shows how to transfer operadic structures through homotopy equivalences of dg $\NN$-modules (and not necessarily homotopy equivalences of dg $\Sy$-modules). 
\begin{theoremintro5}[Homotopy Transfer Theorem]
Let $\{\calH(n), \dif_{\calH(n)}\}_{n\in \NN}$
be a homotopy retract of 
$\{\P(n), \allowbreak \dif_{\P(n)}\}_{n\in \NN}$ in the category of dg $\NN$-modules:
\begin{align*}
&\xymatrix{     *{ \quad \ \  \quad (\P, \dif_\P)\ } \ar@(dl,ul)[]^{h}\ \ar@<0.5ex>[r]^{p} & *{\
(\calH, \dif_{\calH})\quad \ \  \ \quad }  \ar@<0.5ex>[l]^{i} \ , }\\
&i p- \id_\P =\dif_\P  h+ h  \dif_\P, \ i\  \text{quasi-isomorphism}\ .
\end{align*}
Any (higher homotopy) operad structure on $\P$ can be transferred into a higher homotopy operad structure on $\calH$ such that the quasi-isomorphism $i$ extends to an $\infty$-quasi-isomorphism.
\end{theoremintro5}

The methods developed and used in this paper to solve the above-mentioned operadic issues also allows us to define a suitable notion of group representation up to homotopy. This is treated separately in an appendix. 

 \smallskip
 
 A few remarks are in order to conclude this introduction. We tried to provide the reader with as many figures of trees as possible to render the exposition more readable. A particular care has been applied to the treatment of the signs; on the opposite to some articles in this domain (including some written by the second author), \textit{all} the signs are made explicit.  We hope that, in this way, the reader will be able to perform actual computations with the various notions introduced here.  
To go further, one can apply the very same method to the categories of cyclic operads, modular operads and \textit{unital} operads. In this last case, this would help developing the homotopy properties of the partial compositions, symmetric groups action \textit{and} the unit at the same time. 
The colored operad which encodes unital operads still admits a natural inhomogeneous presentation, of course more involved than that of $\calO$. It is very likely to be curved Koszul too. The present paper should represent about two-thirds of the relevant computations to treat this case. 
Since we were only interested in applying our theory to augmented operads, or equivalently to non-unital operads, we only focused to this later case. 

\smallskip

We chose to developed the present theory because we had in mind the example of the operad encoding Lie algebras. Our ultimate goal is to be able to produce a suitable version of \textit{weak $L_\infty$-algebras} where both the Jacobi identity \textit{and} the skew-symmetry of the Lie bracket 
will be relaxed up to higher coherent homotopies. 
There already exists a notion of \emph{weak Lie $2$-algebras} due to  D. Roytenberg \cite{Roytenberg07, Noohi13},   which   is made up of  just a first stratum of  homotopies. If not the most general relaxed notion, it however already admits applications in symplectic geometry and in mathematical physics.
Producing a definition including an infinite sequence of coherent higher homotopies by hands is out of reach (at least for us), but the present conceptual theory is the right tool to solve this issue. 
First computations already show that we can recover Roytenberg's notion
of weak Lie $2$-algebras and that we can
produce as a next step a definition of \emph{weak Lie $3$-algebras}.
The full story is work in progress and will be the subject of another
paper.

\smallskip
Finally, the present paper, with these new higher bar and cobar constructions, is the first step toward a new higher Koszul duality theory for operads taking faithfully into account the symmetric group actions. 
The general pattern for the Koszul duality theory of algebras over a colored operad  detailed
 by J. Mill\`es in \cite{Milles12} can be applied to algebras over the operad $\calO$, that is to operads.
In this way, one should be able to simplify the cofibrant resolutions produced by the higher bar-cobar resolution of Theorem~\ref{thmintro:BarCobarRes}. 
The case of the operad encoding commutative algebras is of particular interest since it  should give an operad related the cofibrant $E_\infty$-operad given in \cite{Fresse09bis} and to the chain complex of the Fox--Neuwirth cells of the Fulton--Mac\-Pher\-son operad of \cite{GetzlerJones94}. 
This  will also be the subject of a sequel to this paper. 

\subsection*{Layout}
The paper is organised as follows. 
In the first section, we introduce the colored operad which encodes non-unital operads and we prove that its natural  inhomogeneous quadratic  presentation is curved Koszul. In Section~$2$, we introduce a new category of higher cooperads and we define a higher bar-cobar adjunction with the category of non-unital differential graded operads. The resulting higher bar-cobar adjunction is shown to provide a  cofibrant replacement functor, which, applied to the operad encoding commutative algebras gives a cofibrant $E_\infty$-operad. In the third section, we introduce a new notion of higher homotopy operad where the action of the symmetric groups is also relaxed up to homotopy. This gives rise to new tools to study the homotopy properties of operads over any ring. The appendix~A deals with the notion of group representation up to homotopy. 

\subsection*{Acknowledgements} We express our appreciation to the J.A. Dieudonn\'e department of the University of Nice Sophia Antipolis and the Newton Institute for the invitations and the excellent working conditions. We would like to thank Clemens Berger, Gabriel C. Drummond-Cole, Benoit Fresse, Joan Bellier-Mill\`es, and Henrik Strohmayer  for useful discussions.  

\subsection*{Conventions}
We work over a commutative ground ring  $\KK$, for instance the ring of integers $\mathbb{Z}$. Since all the $\KK$-modules considered in this paper are free, they are all projective. 
We denote by $s$ the suspension  operator
of degree $1$:  $(sC)_{\bullet+1}:=C_\bullet$.
We often denote a permutation $\sigma \in \Sy_n$ of the  set $\{1,\dots,n\}$ by its values
$[\sigma(1),\dots,\sigma(n)]$. A collection of vector spaces indexed by the integers is called an $\NN$-module. 

\section{The colored operad which encodes operads}\label{sec:one}
In this section, we begin by recalling the definitions of the various types of operads used throughout the text. For further
details, we refer the reader to the book \cite{LodayVallette12}. We introduce the colored operad which encodes the
category of operads \emph{with} their symmetric groups actions viewed as structure maps. We show that it admits an inhomogeneous quadratic presentation, which is Koszul in the sense of the curved Koszul duality theory of Hirsh--Mill\`es \cite{HirshMilles12}. 

\subsection{Non-unital operads and colored operads}

\begin{definition}[Non-unital operad]\label{def:Operad}
A \emph{non-unital  operad} $\P$ is a collection $\left\{  \P(n)  \right\}_{n\in \NN}$ of vector spaces  
equipped with actions of the symmetric groups $\Sy_n$ and with partial composition maps:
$$\left\{\begin{array}{clll}
   (-)^\sigma&\colon & \P(n) \longrightarrow  \P(n)\ ,
      & \text{for}\ \sigma \in \Sy_n \ , \\
    \circ_i&\colon& \P(n) \otimes
      \P(k)
      \longrightarrow \P(n+k-1) \ ,
      & \text{for}\ 1 \leq i \leq n \ .
  \end{array}
  \right.
  $$
These are required to satisfy the axioms of a right group action 
\begin{itemize}
\item[(i)] $(\mu^{\sigma})^{\tau}=\mu^{\sigma\tau}\ $ ,
\item[(ii)] $\mu^{\id_n}=\mu\ , \quad $ for all $\mu \in \P(n)$ and $\sigma, \tau\in \Sy_n\ $, 
\end{itemize} 
the parallel composition and the sequential composition axioms 
\begin{itemize}
\item[(iii)] $(\lambda \circ_j \mu) \circ_i \nu  {=} (\lambda \circ_i \nu) \circ_{j+k-1} \mu, \quad$ when $i<j$\ ,
\item[(iv)] $\lambda \circ_i (\mu \circ_j \nu)
          = (\lambda \circ_i \mu) \circ_{j+i-1} \nu, \quad $ for all $\lambda \in \P(m)$,  $\mu \in \P(n)$, $\nu \in \P(k)$\ , 
\end{itemize} 
and the compatibility axioms between them
\begin{itemize}
\item[(v)]$\mu \circ_i \nu^\sigma
          = (\mu \circ_i \nu)^{ \sigma'}, \quad$ for all $\mu \in \P(n)$, $\nu \in \P(k)$ and $\sigma\in \Sy_k$\ ,

\noindent          
where $\sigma'$ is the permutation of $\Sy_{n+k-1}$          which acts like $\sigma$ on the block 
$\{i, \ldots, i+k-1\}$ and identically outside, 
          
\item[(vi)] 
$\mu^\sigma \circ_i \nu
          = (\mu \circ_{\sigma(i)} \nu)^{\sigma''}, \quad$ for all $\mu \in \P(n)$, $\nu \in \P(k)$ and $\sigma\in \Sy_n$\ , 

\noindent          
where $\sigma''$ is the permutation of $\Sy_{n+k-1}$ which acts like $\sigma$ on 
$\{ 1, \ldots, i-1\}\cup \{i+k, \ldots, n+k-1\}$          with values in 
$\{ 1, \ldots, \sigma(i)-1\}\cup \{\sigma(i)+k, \ldots, n+k-1\}$ and identically on the block 
$\{i, \ldots, i+k-1\}$ with values in 
$\{\sigma(i), \ldots, \sigma(i)+k-1\}$\ .
\end{itemize} 
\end{definition}

Since this definition is set-theoretical, it applies to any symmetric (closed) 
monoidal category, for instance to topological spaces or  
differential graded vector spaces. In the present paper, we will only work with differential graded non-unital operads. 
A morphism of dg non-unital operads is a morphism of dg $\NN$-modules which respects the structures maps; the associated category is denoted by $\mathsf{nu\ Op}$. \\

Let us consider the identity operad  $\I$ made up of only the identity operation: $\I(1)=\KK \id$ and $\I(n)=0$, for $n\neq 1$.
The usual notion of an operad  is defined with a unit $u : \I \to \P(1)$ for the partial composition products: 
$\id \circ_1\,  \mu=\mu$ and $\mu\circ _i \id = \mu$, where $\id:=u(1) \in \P(1)$. However, many operads come equipped with an augmentation, that is a morphism of operads $\varepsilon : \P \to \I$. Such an operad splits into $\P= \I \oplus \oP$, where the augmentation ideal $\oP:=\ker \varepsilon$ is a non-unital operad. In the other way round, one can freely add  a unit to a non-unital operad $\P$: $\P^+:=\I \oplus \P$. These two functors provide us with an equivalence of categories between non-unital operads and augmented operads:
$$ \mathsf{nu\ Op} \cong \mathsf{aug\ Op}\ .$$

\begin{example}
  Given a vector space $V$, we consider the collection of all  multilinear maps from $V$ to $V$: 
  \[ \End_V(n) := \Hom(V^{\otimes n},V) . \]
  The permutation of the inputs, respecting the Koszul sign rule in
  the graded case, and the classical composition of functions endow
  it with a non-unital operad structure.
\end{example}

\begin{definition}[Algebra over an operad]
  A  \emph{$\P$-algebra} structure on a vector space $V$ is a morphism of operads 
  $\P \to \End_V$. 
\end{definition}

An operad can act on a single vector space. To model the operations acting on various vector spaces at the same time, we need to consider the following colored version of operads. 

\begin{definition}[Colored  operad] 
  Let $X$ be a set; its elements are referred to as colors. An 
  \emph{$X$-colored  operad} $\opd{O}$, or \emph{colored  operad} for short,  consists of an underlying 
  collection of  vector spaces 
  \[ \{ \opd{O}(x_0; x_1,\ldots,x_n)\ |\ x_0,\ldots,x_n \in X \}, \]
  equipped with symmetric group actions, partial composition maps and a unit map
  $$\left\{\begin{array}{cll}
    (-)^\sigma & \colon & \opd{O}(x_0; x_1,\ldots,x_n)
      \longrightarrow \opd{O}(x_0; x_{\sigma(1)},\ldots,x_{\sigma(n)})\ ,
      \quad \text{for}\ \sigma \in \Sy_n \ , \\
    \circ_i & \colon & \calO(x_0; x_1,\ldots,x_n) \otimes \calO(x_i; y_1,\ldots, y_k)
      \longrightarrow \calO(x_0; x_1,\ldots,x_{i-1},y_1,\ldots,y_k,x_{i+1},\ldots,x_n) \ , \\
      && \hfill
        \text{for}\ 1 \leq i \leq n\quad \text{and for}\quad x_0, x_1, \ldots, x_n, y_1,\ldots, y_k\in X\ , \ \\
    u_x & \colon & \KK \longrightarrow \calO(x;x)\ ,
      \quad \text{for any }x\in X\ .
  \end{array}
  \right.
  $$
They  should satisfy  the aforementioned axioms (i)-(vi) color-wise.
 The data of the unit maps  is equivalent to unit elements $\id_x\in 
 \calO(x;x)$. They are  required to satisfy
  $$\id_x \circ_1\, \mu =\mu \quad \text{and}\quad 
 \mu \circ_i \id_{x_i}  =\mu,\quad \text{for}\ \mu\in \calO(x; x_1, \ldots, x_n)\ \text{and for}\ 1\leq i\leq n\ .$$  
\end{definition}

\begin{example}
  Given a  vector space $V_x$, for each color $x \in X$, we consider the collection of all  multilinear maps from the $V_x$ to themselves: 
  \[ \End_V(x_0; {x_1,\ldots,x_n})
     := \Hom(V_{x_1}\otimes\cdots\otimes V_{x_n},V_{x_0}) . \]
  The permutation of the inputs, the classical composition of 
  functions and the identity functions endow it with a colored
  operad structure.
\end{example}

\begin{definition}[Algebra over a colored operad]
  An \emph{$\calO$-algebra} structure on a family of vector spaces
  $\{V_x\}_{x\in X}$ is a morphism of colored operads
  $\calO \to \End_V$.
\end{definition}

One can extend the aforementioned definitions to the differential
graded setting. In the next section, we will introduce a particular
colored operad whose algebras are non-unital operads.

\subsection{Non-unital operads as algebras over a colored  operad}
\label{subsec:Operads=OAlgebras}

\begin{definition}[Colored operad $\calO$]\label{def:QLCPresentation}
We define the $\NN$-colored operad $\mathcal{O}=\TTT(E)/(R)$ by generators
\begin{align*}
  &
  E(n;n) := \left\{
   \begin{aligned}\begin{tikzpicture}[optree]
    \node{}
      child { node[dot,label=right:$\sigma$]{}
        child { edge from parent node[right,near end]{\tiny$n$} }
        edge from parent node[right,near start]{\tiny$n$} } ;
   \end{tikzpicture}\end{aligned} ,
   \sigma \in \Sy_n\setminus\{\id_n\}
  \right\}
  , \text{for} \ n\ge 1,\  \text{and}& \\
&  E(n+k-1; n,k) := \left\{
   \begin{aligned}\begin{tikzpicture}[optree]
    \node{}
      child { node[circ]{$i$}
        child { node[label=above:$1$]{} edge from parent node[left,near end]{\tiny$n$} }
        child { node[label=above:$2$]{} edge from parent node[right,near end]{\tiny$k$} }
        edge from parent node[right,near start]{\tiny$n+k-1$} } ;
   \end{tikzpicture}\end{aligned} , 
   1 \leq i \leq n;
          \begin{aligned}
      \begin{tikzpicture}[optree]
    \node{}
      child { node[circ]{$j$}
        child { node[label=above:$2$]{} edge from parent node[left,near end]{\tiny$n$} }
        child { node[label=above:$1$]{} edge from parent node[right,near end]{\tiny$k$} }
        edge from parent node[right,near start]{\tiny$n+k-1$} } ;
   \end{tikzpicture}\end{aligned}, 
   1 \leq j \leq k
  \right\} , \text{for} \ n,k\ge 0,&
\end{align*}
with regular $\Sy_2$-action,  that is the action of the transposition $[21]$ sends 
$\begin{aligned}\begin{tikzpicture}[optree]
    \node{}
      child { node[circ]{$i$}
        child { node[label=above:$1$]{} edge from parent node[left,near end]{\tiny$n$} }
        child { node[label=above:$2$]{} edge from parent node[right,near end]{\tiny$k$} }
        edge from parent node[right,near start]{\tiny$n+k-1$} } ;
   \end{tikzpicture}\end{aligned}$
    to 
$ \begin{aligned}
      \begin{tikzpicture}[optree]
    \node{}
      child { node[circ]{$i$}
        child { node[label=above:$2$]{} edge from parent node[left,near end]{\tiny$k$} }
        child { node[label=above:$1$]{} edge from parent node[right,near end]{\tiny$n$} }
        edge from parent node[right,near start]{\tiny$n+k-1$} } ;
   \end{tikzpicture}\end{aligned}$, 
   and relations
\begin{multicols}{2}
\begin{enumerate}[(i)]
\item \label{Orel1a}
$ \begin{aligned}
  \begin{tikzpicture}[optree]
    \node{}
      child { node[dot,label=right:$\tau$]{}
        child { node[dot,label=right:$\sigma$]{}
          child { edge from parent node[right,near end]{\tiny$n$} } } } ;
  \end{tikzpicture}\end{aligned}
  \stackrel{(\tau\neq\sigma^{-1})}{=}
  \begin{aligned}\begin{tikzpicture}[optree]
    \node{}
      child { node[dot,label=right:$\sigma\tau$]{}
        child { edge from parent node[right,near end]{\tiny$n$} } } ;
  \end{tikzpicture}\end{aligned}
  , $
\setcounter{enumi}{2}
\item \label{Orel3a}
$ \begin{aligned}\begin{tikzpicture}[optree,scale=1.2]
    \node{}
      child { node[circ]{$i$}
        child { node[circ]{$j$}
          child { node[label=above:$1$]{}
            edge from parent node[left,near end]{\tiny$m$} }
          child { node[label=above:$2$]{}
            edge from parent node[right,near end]{\tiny$n$} } }
        child { node[label=above:$3$]{}
          edge from parent node[right,near end]{\tiny$k$} } } ;
  \end{tikzpicture}\end{aligned}
  \stackrel{(i<j)}{=}
  \begin{aligned}\begin{tikzpicture}[optree,scale=1.2]
    \node{}
      child { node[circ,ellipse]{$j+k-1$}
        child { node[circ]{$i$}
          child { node[label=above:$1$]{}
            edge from parent node[left,near end]{\tiny$m$} }
          child { node[label=above:$3$]{}
            edge from parent node[right,near end]{\tiny$k$} } }
        child { node[label=above:$2$]{}
          edge from parent node[right,near end]{\tiny$n$} } } ;
  \end{tikzpicture}\end{aligned}
  , $
\setcounter{enumi}{4}
\item \label{Orel2a}
$ \begin{aligned}\begin{tikzpicture}[optree]
    \node[leaf]{}
      child { node[circ]{$i$}
        child { node[leaf,label=above:$1$]{}
          edge from parent node[left,near end]{\tiny$n$} }
        child { node[dot,label=right:$\sigma$]{}
          child { node[leaf,label=above:$2$]{}
            edge from parent node[right,near end]{\tiny$k$} } } } ;
  \end{tikzpicture}\end{aligned}
  =
  \begin{aligned}\begin{tikzpicture}[optree]
    \node[leaf]{}
      child { node[dot,label=right:$\sigma'$]{}
        child { node[circ]{$i$}
          child { node[leaf,label=above:$1$]{}
            edge from parent node[left,near end]{\tiny$n$} }
          child { node[leaf,label=above:$2$]{}
            edge from parent node[right,near end]{\tiny$k$} } } } ;
  \end{tikzpicture}\end{aligned}
  , $
\setcounter{enumi}{1}
\item \label{Orel1b}
$ \begin{aligned}\begin{tikzpicture}[optree]
    \node{}
      child { node[dot,label=right:$\sigma^{-1}$]{}
        child { node[dot,label=right:$\sigma$]{}
          child { edge from parent node[right,near end]{\tiny$n$} } } } ;
  \end{tikzpicture}\end{aligned}
  =
  \begin{aligned}\begin{tikzpicture}[optree]
    \node{}
      child { node[comp,label=right:$\id_n$]{}
        child { edge from parent node[right,near end]{\tiny$n$} } } ;
  \end{tikzpicture}\end{aligned}
  , $
\setcounter{enumi}{3}
\item \label{Orel3b}
$ \begin{aligned}\begin{tikzpicture}[optree,scale=1.2]
    \node{}
      child { node[circ]{$i$}
        child { node[label=above:$1$]{}
          edge from parent node[left,near end]{\tiny$m$} }
        child { node[circ]{$j$}
          child { node[label=above:$2$]{}
            edge from parent node[left,near end]{\tiny$n$} }
          child { node[label=above:$3$]{}
            edge from parent node[right,near end]{\tiny$k$} } } } ;
  \end{tikzpicture}\end{aligned}
  =
  \begin{aligned}\begin{tikzpicture}[optree,scale=1.2]
    \node{}
      child { node[circ,ellipse]{$j+i-1$}
        child { node[circ]{$i$}
          child { node[label=above:$1$]{}
            edge from parent node[left,near end]{\tiny$m$} }
          child { node[label=above:$2$]{}
            edge from parent node[right,near end]{\tiny$n$} } }
        child { node[label=above:$3$]{}
          edge from parent node[right,near end]{\tiny$k$} }
      } ;
  \end{tikzpicture}\end{aligned}
  , $
\setcounter{enumi}{5}
\item \label{Orel2b}
$ \begin{aligned}\begin{tikzpicture}[optree]
    \node{}
      child { node[circ]{$i$}
        child { node[dot,label=left:$\sigma$]{}
          child { node[label=above:$1$]{}
            edge from parent node[left,near end]{\tiny$n$} } }
        child { node[label=above:$2$]{}
          edge from parent node[right,near end]{\tiny$k$} } } ;
  \end{tikzpicture}\end{aligned}
  =
  \begin{aligned}\begin{tikzpicture}[optree]
    \node{}
      child { node[dot,label=right:$\sigma''$]{}
        child { node[circ,ellipse,inner sep=0pt]{$\sigma(i)$}
          child { node[label=above:$1$]{}
            edge from parent node[left,near end]{\tiny$n$} }
          child { node[label=above:$2$]{}
            edge from parent node[right,near end]{\tiny$k$} } } } ;
  \end{tikzpicture}\end{aligned}
  . $
\end{enumerate}
\end{multicols}
\end{definition}

\begin{lemma}
  The algebras over the colored operad $\mathcal{O}$ are the non-unital
   operads.
\end{lemma}

\begin{proof}Straightforward from Definition~\ref{def:Operad}.
\end{proof}

We consider the following two suboperads of $\calO$. The operad
$\calO_{ns}$ is generated by the binary generators $E(n+k-1; n, k)$
with relations \eqref{Orel3a} and \eqref{Orel3b}. Algebras over the operad $\calO_{ns}$ are   
non-unital nonsymmetric operads. It was shown 
by P. Van der Laan in \cite[Proposition~$4.2$]{VanderLaan03} that 
one  basis   is made up of planar rooted
trees with vertices indexed bijectively by $1, \ldots, k$. The partial compositions then correspond  to the substitution of trees.

The colored operad $\KK[\Sy]$ is generated by the unary generators
$E(n;n)$ with relations \eqref{Orel1a} and \eqref{Orel1b}. This is the
trivial presentation of the group algebras $\KK[\Sy_n]$. So the
colored operad $\KK[\Sy]$ is a colored operad concentrated in arity
1 equal to the  colored group algebra of symmetric groups
$\{\KK[\Sy_n]\}_{n\in \NN}$. 

The full colored operad $\calO$ is made up of the two suboperads $\calO_{ns}$ and $\KK[\Sy]$ by mean of distributive
law, see \cite[\S $8.6$]{LodayVallette12}, as follows. 

\begin{proposition}\label{prop:DistLawO}
The rewriting rule
defined by the relations (v) and (vi), read from left to right,
induces a distributive law
$\calO_{ns} \circ \KK[\Sy] \to \KK[\Sy] \circ \calO_{ns}$\ .
Therefore, the underlying colored $\Sy$-module of $\mathcal O$ is isomorphic to $\KK[\Sy] \circ \calO_{ns}$.
\end{proposition}

\begin{proof}
One has  to check that the surjection $\KK[\Sy] \circ \calO_{ns} \epi \calO$ is injective. It is well known that 
the free non-unital operad on an $\Sy$-module $M$ is given by the module of nontrivial trees with vertices labeled by the elements of $M$ and with the leaves labeled by a permutation of $\Sy_n$, see \cite[Section~$5.6$]{LodayVallette12}. Together with \cite[Proposition~$4.2$]{VanderLaan03}, this proves that the colored operad $\mathcal O$ is isomorphic to $\KK[\Sy] \circ \calO_{ns}$.
\end{proof}

\begin{remarks*}\leavevmode
\begin{itemize}
  \item[$\diamond$] This proposition proves that the colored operad $\calO$ defined here by generators and relations
    is the reduced version, i.e. without anything in arity $0$, of the colored operad $S$ of Berger--Moerdijk 
    \cite[\S $1.5.6$]{BergerMoerdijk07} defined by labeled trees and encoding operads.
    
  \item[$\diamond$]  The colored operad $\mathcal O$ is made up of regular representations of the symmetric groups, however  it is
    not a regular operad. This means that the colored operad $\calO$ is not obtained from a nonsymmetric operad by tensoring it with the regular representation of the symmetric goups since 
    the relation (iii) contains a nontrivial permutation of the inputs. 

  \item[$\diamond$]  This proposition is actually equivalent to the commutativity of the following diagram of forgetful functors 
 $$\xymatrix@R=20pt@C=40pt@M=5pt{
{\mathsf{nu}\ \mathsf{Op}} \ar[d]\ar[r]& {\mathsf{nu}\ \mathsf{ns}\ \mathsf{Op}\ar[d] } \\
\Sy\textsf{-}\mathsf{Mod}\ar[r]& \NN\textsf{-}\mathsf{Mod}\ , }$$
where $\mathsf{ns}\ \mathsf{Op}$ stands for the category of ns operads. The induced commutative diagram of left adjoint functors shows that the free $\calO$-algebra on an $\NN$-module $\P$ is isomorphic to 
$$\calO(\P)\cong \overline{\calT}(\P\otimes_H \KK[\Sy])\ ,$$
where $\overline{\calT}$ stands for the non-unital free operad functor of an $\Sy$-module, see \cite[\S~$5.5$]{LodayVallette12}
and where
 $\otimes_H$ is the aritywise tensor product. 
\end{itemize}
\end{remarks*}

\subsection{Trees, bases, orders and signs}\label{subsec:TOS}
We begin by considering planar trees, without any restriction on the valence of the vertices, which can be equal to $0$; we call them \emph{operadic trees}. 

\begin{proposition}\label{prop:OperadicTreesBasis}
  The set of  planar trees  labeled with one permutation  on the number of leaves and one permutation on the number of vertices 
  forms a $\KK$-linear basis of the colored operad $\calO$.
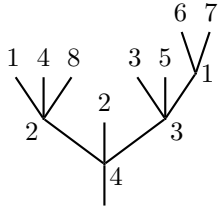
  \begin{align*}
    \begin{aligned}\begin{tikzpicture}[optree,
        level distance=12mm,
        level 2/.style={sibling distance=16mm},
        level 3/.style={sibling distance=8mm}]
    \node{}
      child { node[comp,label=-3:$4$] {} 
        child {node[comp,label=183:$2$] {} 
          child {  node[label=above:$1$]{} }
          child { node[label=above:$4$]{} }
          child { node[label=above:$8$]{} } }
          child { node[label=above:$2$]{} }
        child { node[comp,label=-3:$3$] {} 
          child { node[label=above:$3$]{} }
          child { node[label=above:$5$]{} }
          child { node[comp,label=right:$1$] {} 
            child { node[label=above:$6$]{} }
            child { node[label=above:$7$]{} } } } };
    \end{tikzpicture}\end{aligned}
  \end{align*}
\end{proposition}

\begin{proof} It is a direct corollary of Proposition~\ref{prop:DistLawO}, since planar trees form a basis of the colored operad $\mathcal{O}_{ns}$; the permutation labeling the leaves comes from the suboperad $\KK[\Sy]$.
\end{proof}

In the rest of the paper, we will consider the following  total order on the set of vertices of operadic trees. 

\begin{description}
\item[left-recursive] Starting at the root, recursively order child
vertices from left to right.
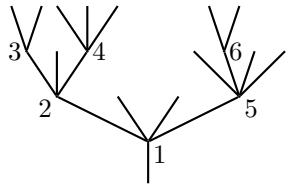
\begin{align*}
  \begin{tikzpicture}[optree,
      level distance=12mm,
      level 2/.style={sibling distance=16mm},
      level 3/.style={sibling distance=8mm}]
    \node{}
      child { node[comp,label=-3:$1$] {}
        child { node[comp,label=183:$2$] {}
          child { node[comp,label=182:$3$] {}
            child
            child }
          child
          child { node[comp,label=right:$4$] {}
            child
            child
            child } }
        child
        child
        child { node[comp,label=-3:$5$] {}
          child
          child { node[comp,label=2:$6$] {}
            child
            child }
          child
          child } } ;
  \end{tikzpicture}
\end{align*}
\end{description}

We now consider \emph{composite trees}, which are  planar trees with only unary and binary vertices labeled respectively by elements of the symmetric groups and by positive integers. 
 Composite trees which are left combs with nondecreasing binary vertices from top to bottom, are called \emph{standard composite trees}. They provide us with another basis for the colored operad $\mathcal O$ as follows. \\
 
  The key point lies  in the correspondence between the binary generators of the colored operad $\calO$ and trees with two vertices:
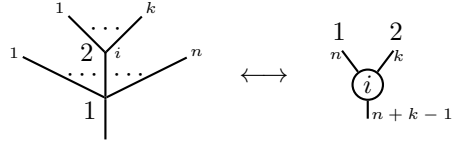
\begin{align*}  
\begin{aligned}\begin{tikzpicture}[optree,
      level distance=12mm,
      level 2/.style={sibling distance=12mm},
      level 3/.style={sibling distance=12mm}]
    \node{}
      child { node[comp,label={[label distance=0.5mm]185:$1$}]{}       
        child { node{\tiny $1$} }
        child { edge from parent[draw=none] node{$\,\cdots$} }
        child { node[comp,label=right:\tiny$\,i$]{} node[comp,label=182:$2 \ $]{} 
          child { node{\tiny $1$} }
          child { edge from parent[draw=none] node{$\,\cdots$} }
          child { node{\tiny $k$} } }
        child { edge from parent[draw=none] node{$\,\cdots$} }
        child { node{\tiny $n$} } } ;
  \end{tikzpicture}\end{aligned}
  \quad\longleftrightarrow\quad
  \begin{aligned}\begin{tikzpicture}[optree]
    \node{}
      child { node[circ]{$i$}
        child { node[label=above:$1$]{} edge from parent node[left,near end]{\tiny$n$} }
        child { node[label=above:$2$]{} edge from parent node[right,near end]{\tiny$k$} }
        edge from parent node[right,near start]{\tiny$n+k-1$}
      } ;
  \end{tikzpicture}\end{aligned}
  .
\end{align*}
The labels of the vertices of any operadic tree given by the left-recursive order coincide with the labels of the
leaves of the associated standard composite tree: 
if we traverse the 
operadic tree in a left-recursive manner, the inner edges found in 
this order are encoded as binary vertices in the composite tree.
In the other way round,  any standard composite tree corresponds to the canonical way of creating an operadic
tree by grafting the corollas in the left-recursive way, i.e. from bottom to top and from left to right. In this way, we get a bijection between standard composite trees and operadic trees. 
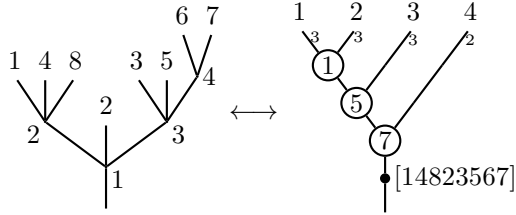
\begin{align*}
  \begin{aligned}\begin{tikzpicture}[optree,
      level distance=12mm,
      level 2/.style={sibling distance=16mm},
      level 3/.style={sibling distance=8mm}]
    \node{}
      child { node[comp,label=-3:$1$]{}
        child { node[comp,label=183:$2$]{}
          child { node[label=above:$1$]{} }
          child { node[label=above:$4$]{} }
          child { node[label=above:$8$]{} } }
          child { node[label=above:$2$]{} }
        child { node[comp,label=-3:$3$]{}
          child { node[label=above:$3$]{} }
          child { node[label=above:$5$]{} }
          child { node[comp,label=-2:$4$]{}
            child { node[label=above:$6$]{} }
            child { node[label=above:$7$]{} } } } };
  \end{tikzpicture}\end{aligned}
\longleftrightarrow
  \begin{aligned}\begin{tikzpicture}[optree]
    \node{}
      child { node[dot,label={right:$[14823567]$}]{}
        child { node[circ]{$7$}
          child { node[circ]{$5$}
            child { node[circ]{$1$}
              child { node[label=above:$1$]{} edge from parent node[right,near end]{\tiny$3$} }
              child { node[label=above:$2$]{} edge from parent node[right,near end]{\tiny$3$} }
            }
            child {
              child { edge from parent[draw=none] }
              child { node[label=above:$3$]{} edge from parent node[right,near end]{\tiny$3$} }
            }
          }
          child {
            child { edge from parent[draw=none] }
            child {
              child { edge from parent[draw=none] }
              child { node[label=above:$4$]{} edge from parent node[right,near end]{\tiny$2$} }
            }
          }
        }
      };
  \end{tikzpicture}\end{aligned} 
\end{align*}

\begin{proposition}\label{prop:CompositeTreesBasis}
A $\KK$-linear basis for the colored operad $\calO$ is given by the left combs with 
nondecreasing binary vertices, seen top to bottom,  at most one 
unary vertex labeled by a permutation at the bottom and a  labeling of the leaves
with a permutation  on the number of them.
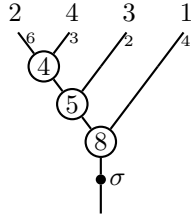
\begin{align*}
  \begin{aligned}\begin{tikzpicture}[optree]
    \node{}
      child { node[dot,label=right:$\sigma$]{}
        child { node[circ]{$8$}
          child { node[circ]{$5$}
            child { node[circ]{$4$}
              child { node[label=above:$2$]{} edge from parent node[right,near end]{\tiny$6$} }
              child {node[label=above:$4$]{}  edge from parent node[right,near end]{\tiny$3$} }
            }            
            child {
              child { edge from parent[draw=none] }
              child { node[label=above:$3$]{} edge from parent node[right,near end]{\tiny$2$} }
            }
          }
          child {
            child { edge from parent[draw=none] }
            child {
              child { edge from parent[draw=none] }
              child { node[label=above:$1$]{} edge from parent node[right,near end]{\tiny$4$} } 
            }
          }
        }
      };
  \end{tikzpicture}\end{aligned} 
\end{align*}
\end{proposition}

\begin{proof}
The proof is given by the above mentioned bijection and by Proposition~\ref{prop:OperadicTreesBasis}. 
 The permutation labeling the root of the composite tree corresponds to the permutation labeling the leaves of the operadic tree.
\end{proof}

On composite trees, we will consider the following  total order on the set of vertices. 

\begin{description}
  \item[left-levelwise] Starting at the root, order child vertices left 
  to right before proceeding to the next level/generation.
  
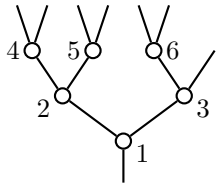
\begin{align*}
    \begin{tikzpicture}[optree,
        level distance=12mm,
        level 2/.style={sibling distance=32mm},
        level 3/.style={sibling distance=16mm},
        level 4/.style={sibling distance=8mm}]
      \node{}
        child { node[emptycirc,label=-3:$1$]{}
          child { node[emptycirc,label=183:$2$]{}
            child { node[emptycirc,label=left:$4$]{}
              child
              child }
            child { node[emptycirc,label=left:$5$]{}
              child
              child } }
          child { node[emptycirc,label=-3:$3$]{}
            child { node[emptycirc,label=right:$6$]{}
              child
              child }
            child } } ;
    \end{tikzpicture}
  \end{align*}
\end{description}

We follow the Koszul sign rule for trees, see \cite[Chapter~$6$]{LodayVallette12}. This means that we  
equip  composite trees  with the left-levelwise ordering on the vertices, i.e. we 
have vertices $\nu_1,\ldots,\nu_n$. We then identify the space of 
labelings of a given tree by elements of graded vector spaces $V_i$ at 
$\nu_i$ with the tensor product $V_1\otimes\cdots\otimes V_n$.
A relabeling of the vertices corresponds to a permutation of the 
factors $V_i$ and the isomorphism honors the classical Koszul sign 
rule: $x\otimes y \mapsto (-1)^{xy} y\otimes x$.

\subsection{Quadratic analogue}\label{subsec:KDofqO}
We consider the projection $q\colon \TTT(E) \epi \TTT(E)^{(2)}$  
onto the weight $2$ part of the free operad. The quadratic analogue 
of the colored operad $\calO$ is the quadratic colored operad 
$$q\calO := \TTT(E)/(qR)\ .$$ In the case of the colored 
operad $\calO$, the space of relations $qR$ differs from $R$ only in 
its unary part: the right-hand sides of the relations
\eqref{Orel1a}, \eqref{Orel1b} project to $0$, i.e.
\begin{align*}
& \begin{aligned}
  \begin{tikzpicture}[optree]
    \node{}
      child { node[dot,label=right:$\tau$]{}
        child { node[dot,label=right:$\sigma$]{}
          child { edge from parent node[right,near end]{\tiny$n$} } } } ;
  \end{tikzpicture}\end{aligned}
  = 0 .
\end{align*}

The first step in the curved Koszul duality   theory is to 
show that the homogenous quadratic colored operad $q\calO$ is Koszul. To this extend, we consider the nilpotent quadratic  unital
colored algebra $\KK\{\Sy\}$ generated by all the elements of $\Sy$ except for the identities and with the maximal set of relations, i.e. any product is equal to zero. Its underlying module is the same as the colored algebra $\KK[\Sy]$.

\begin{proposition}\label{prop:DistLawqO}\leavevmode 
\begin{enumerate}
\item The rewriting rule
defined by the relations (v) and (vi), read from left to right,
induces a distributive law
$\calO_{ns} \circ \KK\{\Sy\} \to \KK\{\Sy\} \circ \calO_{ns}$\ .
Therefore, the underlying colored $\Sy$-module of $q\calO$ is isomorphic to $\KK\{\Sy\} \circ \calO_{ns}$.

\item The homogenous quadratic colored operad $q\calO$ is Koszul. 
\end{enumerate}
\end{proposition}

\begin{proof}
These two points are actually proved at once by the Diamond Lemma for distributive laws
\cite[Theorem~$8.6.5$]{LodayVallette12}. Both quadratic colored operads $\KK\{\Sy\}$ and $\calO_{ns}$ are Koszul: the
first one is nilpotent, with maximal set of relations, and the second one was proved to be Koszul in \cite[Theorem~$4.3$]{VanderLaan03}.
Strictly speaking, Van der Laan's Koszul duality theory for colored operads was proved over a field of characteristic zero. One can extend it over a commutative ring using the methods of Fresse \cite{Fresse04}. This applies to the colored operad $q\calO$ since it is made up of free finitely generated $\KK$-modules and free $\KK[\Sy_n]$-modules. 

 It remains to
prove that the natural map $(\KK\{\Sy\}\circ \calO_{ns})^{(3)} \to q\calO$ is injective in weight 3. Since the
left-hand term is concentrated in weight $1$ (except for its unit), it is enough to prove that
$\overline{\KK\{\Sy\}}\circ \calO_{ns}^{(2)}$ imbeds into $q\calO$. We use the relations of $q\calO$, read from left
to right, as rewriting rules. In this way, one can always pull down the symmetric group elements. So if a tree
presents at least two such elements, it vanishes in the colored operad $q\calO$. Let us now restrict to the summand
generated by trees with one element coming from the symmetric group; it is a set-theoretical summand, that is linearly
spanned by trees subject to the relations $(iii)$, $(iv)$, $(v)$, and $(vi)$. For our purpose, it is now enough to
prove that in each equivalence class of trees with two binary vertices and one unary vertex, there exists a unique
standard composite tree with the unary vertex at the bottom. This is shown by checking that the ambiguities are all
confluent; these ambiguities are given by the left-hand terms of relations $(iii)$ and $(iv)$ with a symmetric group
element labeling one leaf. Let us give one such verification below; the other ones follow the same pattern.
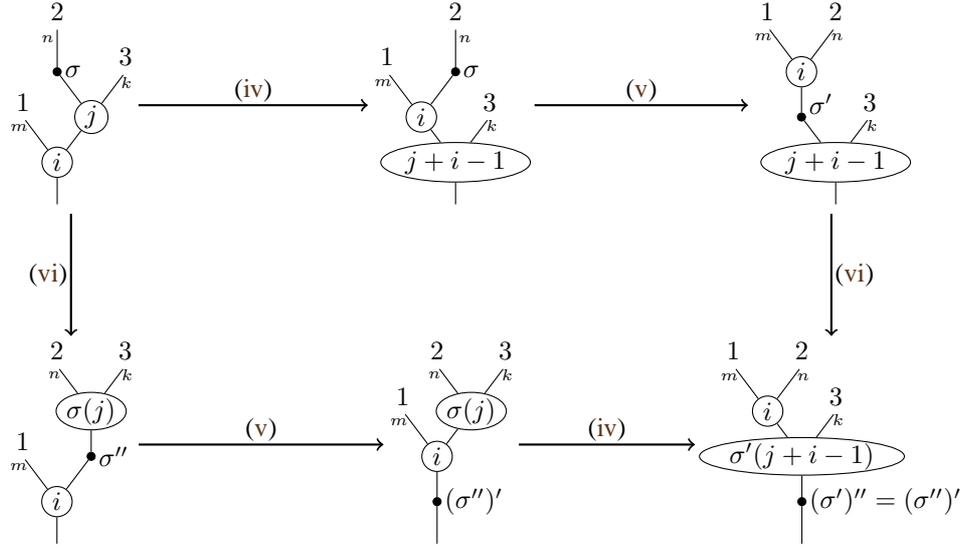
\begin{center}
  \begin{tikzpicture}[optree]
    \matrix (m) [column sep={5cm,between origins}, row sep={4.5cm,between origins}] {
      \node[name=m-1-1] {\tikz
      \node {}
        child[scale=0.6] { node[circ]{$i$}
          child { node[label=above:$1$]{}
            edge from parent node[left,near end]{\tiny$m$} }
          child { node[circ]{$j$}
            child { node[dot,label=right:$\sigma$]{}
              child { node[label=above:$2$]{}
                edge from parent node[left,near end]{\tiny$n$} } }
            child { node[label=above:$3$]{}
              edge from parent node[right,near end]{\tiny$k$} } } } ;
      } ;
      &
      \node[name=m-1-2] {\tikz
      \node {}
        child[scale=0.6] { node[circ,ellipse]{$j+i-1$}
          child { node[circ]{$i$}
            child { node[label=above:$1$]{}
              edge from parent node[left,near end]{\tiny$m$} }
            child { node[dot,label=right:$\sigma$]{}
              child { node[label=above:$2$]{}
                edge from parent node[right,near end]{\tiny$n$} } } }
          child { node[label=above:$3$]{}
            edge from parent node[right,near end]{\tiny$k$} } } ;
      } ;
      &
      \node[name=m-1-3] {\tikz
      \node {}
        child[scale=0.6] { node[circ,ellipse]{$j+i-1$}
          child { node[dot,label=3:$\sigma'$]{}
            child { node[circ]{$i$}
              child { node[label=above:$1$]{}
                edge from parent node[left,near end]{\tiny$m$} }
              child { node[label=above:$2$]{}
                edge from parent node[right,near end]{\tiny$n$} } } }
          child { node[label=above:$3$]{}
            edge from parent node[right,near end]{\tiny$k$} } } ;
      } ;
      \\
      \node[name=m-2-1] {\tikz
      \node {}
        child[scale=0.6] { node[circ]{$i$}
          child { node[label=above:$1$]{}
            edge from parent node[left,near end]{\tiny$m$} }
          child { node[dot,label=right:$\sigma''$]{}
            child { node[circ,ellipse,inner sep=0pt]{$\sigma(j)$}
              child { node[label=above:$2$]{}
                edge from parent node[left,near end]{\tiny$n$} }
              child { node[label=above:$3$]{}
                edge from parent node[right,near end]{\tiny$k$} } } } } ;
      } ;
      &
      \node[name=m-2-2] {\tikz
      \node {}
        child[scale=0.6] { node[dot,label=right:$(\sigma'')'$]{}
          child { node[circ]{$i$}
            child { node[label=above:$1$]{}
              edge from parent node[left,near end]{\tiny$m$} }
            child { node[circ,ellipse,inner sep=0pt]{$\sigma(j)$}
              child { node[label=above:$2$]{}
                edge from parent node[left,near end]{\tiny$n$} }
              child { node[label=above:$3$]{}
                edge from parent node[right,near end]{\tiny$k$} } } } } ;
      } ;
      &
      \node[name=m-2-3] {\tikz
      \node {}
        child[scale=0.6] { node[dot,label=right:${(\sigma')''=(\sigma'')'}$]{}
          child { node[circ,ellipse,inner sep=0pt]{$\sigma'(j+i-1)$}
            child { node[circ]{$i$}
              child { node[label=above:$1$]{}
                edge from parent node[left,near end]{\tiny$m$} }
              child { node[label=above:$2$]{}
                edge from parent node[right,near end]{\tiny$n$} } }
            child { node[label=above:$3$]{}
              edge from parent node[right,near end]{\tiny$k$} } } } ;
      } ;
      \\
    };
    \draw[->] (m-1-1) edge node[above]{\eqref{Orel3b}} (m-1-2)
              (m-1-2) edge node[above]{\eqref{Orel2a}} (m-1-3)
              (m-1-3) edge node[right]{\eqref{Orel2b}} (m-2-3)
              (m-1-1) edge node[left] {\eqref{Orel2b}} (m-2-1)
              (m-2-1) edge node[above]{\eqref{Orel2a}} (m-2-2)
              (m-2-2) edge node[above]{\eqref{Orel3b}} (m-2-3) ;
  \end{tikzpicture}
\end{center}

\end{proof}

The second step in the curved Koszul duality theory is to calculate the Koszul dual operad and cooperad of the
quadratic analogue $q\calO$.

\begin{proposition}\label{prop:KoszulDualOp}
The Koszul dual operad $q\calO^!$ of the quadratic analogue $q\calO$ is generated by 
\begin{align*}
  &
  s^{-1}E(n;n)^* = \left\{
   \begin{aligned}\begin{tikzpicture}[optree]
    \node{}
      child { node[dot,label=right:$s^{-1}\sigma^*$]{}
        child { edge from parent node[right,near end]{\tiny$n$} }
        edge from parent node[right,near start]{\tiny$n$} } ;
   \end{tikzpicture}\end{aligned} ,
   \sigma \in \Sy_n\setminus\{\id_n\}
  \right\}
  ,&\\
 & E(n+k-1; n,k)^*\otimes \mathrm{sgn}_{\Sy_2} = \left\{
   \begin{aligned}\begin{tikzpicture}[optree]
    \node{}
      child { node[circ]{$i^*$}
        child { node[label=above:$1$]{} edge from parent node[left,near end]{\tiny$n$} }
        child {node[label=above:$2$]{}  edge from parent node[right,near end]{\tiny$k$} }
        edge from parent node[right,near start]{\tiny$n+k-1$} } ;
   \end{tikzpicture}\end{aligned} ,
      1 \leq i \leq n; 
   \begin{aligned}\begin{tikzpicture}[optree]
    \node{}
      child { node[circ]{$j^*$}
        child { node[label=above:$2$]{} edge from parent node[left,near end]{\tiny$n$} }
        child {node[label=above:$1$]{}  edge from parent node[right,near end]{\tiny$k$} }
        edge from parent node[right,near start]{\tiny$n+k-1$} } ;
   \end{tikzpicture}\end{aligned} ,
   1 \leq j \leq k
  \right\} ,&
\end{align*}
with the following quadratic relations\\
\begin{subequations}
  \begin{minipage}{.5\linewidth}
  \begin{align}\label{Align:Relation1a}
    &
    \begin{aligned}\begin{tikzpicture}[optree,scale=1.2]
      \node{}
        child { node[circ]{$i^*$}
          child { node[circ]{$j^*$}
            child { node[above,label=$1$]{} edge from parent node[left,near end]{\tiny$m$} }
            child { node[above,label=$2$]{} edge from parent node[right,near end]{\tiny$n$} } }
          child { node[above,label=$3$]{} edge from parent node[right,near end]{\tiny$k$} } } ;
    \end{tikzpicture}\end{aligned}
    \stackrel{(i<j)}{=} 
    \begin{aligned}\begin{tikzpicture}[optree,scale=1.2]
      \node{}
        child { node[circ,ellipse]{$(j+k-1)^*$}
          child { node[circ]{$i^*$}
            child { node[above,label=$1$]{} edge from parent node[left,near end]{\tiny$m$} }
            child { node[above,label=$3$]{} edge from parent node[right,near end]{\tiny$k$} } }
          child { node[above,label=$2$]{} edge from parent node[right,near end]{\tiny$n$} } } ;
    \end{tikzpicture}\end{aligned} ,
  \end{align}
  \end{minipage}
  \begin{minipage}{.5\linewidth}
  \begin{align}\label{Align:Relation1b}
    &
    \begin{aligned}\begin{tikzpicture}[optree,scale=1.2]
      \node{}
        child { node[circ]{$i^*$}
          child { node[above,label=$1$]{} edge from parent node[left,near end]{\tiny$m$} }
          child { node[circ]{$j^*$}
            child { node[above,label=$2$]{} edge from parent node[left,near end]{\tiny$n$} }
            child { node[above,label=$3$]{} edge from parent node[right,near end]{\tiny$k$} } } } ;
    \end{tikzpicture}\end{aligned}
    =
    \begin{aligned}\begin{tikzpicture}[optree,scale=1.2]
      \node{}
        child { node[circ,ellipse]{$(j+i-1)^*$}
          child { node[circ]{$i^*$}
            child { node[above,label=$1$]{} edge from parent node[left,near end]{\tiny$m$} }
            child { node[above,label=$2$]{} edge from parent node[right,near end]{\tiny$n$} } }
          child { node[above,label=$3$]{} edge from parent node[right,near end]{\tiny$k$} } } ;
    \end{tikzpicture}\end{aligned} ,
  \end{align}
  \end{minipage}
\end{subequations}
\begin{align}
  &
  \begin{aligned}\begin{tikzpicture}[optree]
    \node{}
      child { node[dot,label=right:$\tau^*$]{}
        child { node[circ]{$i^*$}
          child { node[above,label=$1$]{} edge from parent node[left,near end]{\tiny$n$} }
          child { node[above,label=$2$]{} edge from parent node[right,near end]{\tiny$k$} } } } ;
  \end{tikzpicture}\end{aligned}
  =
  \begin{cases}
    \begin{aligned}\begin{tikzpicture}[optree]
      \node{}
        child { node[circ]{$i^*$}
          child { node[above,label=$1$]{} edge from parent node[left,near end]{\tiny$n$} }
          child { node[dot,label=right:$\sigma^*$]{}
            child { node[above,label=$2$]{} edge from parent node[right,near end]{\tiny$k$} } } } ;
    \end{tikzpicture}\end{aligned}
    , & \textrm{if } \tau = \sigma' \\
        \begin{aligned}\begin{tikzpicture}[optree,scale=1.2]
      \node{}
        child { node[circ,ellipse]{$\sigma^{-1}(i)^*$}
          child { node[dot,label=left:$\sigma^*$]{}
            child { node[above,label=$1$]{} edge from parent node[right,near end]{\tiny$n$} } }
          child { node[above,label=$2$]{} edge from parent node[right,near end]{\tiny$k$} } } ;
    \end{tikzpicture}\end{aligned}
    , & \textrm{if } \tau = \sigma'' \\
    \ \quad 0, & \textrm{otherwise}.
  \end{cases}
\end{align}
\end{proposition}

\begin{proof} The method recalled  in \cite[Chapter~7]{LodayVallette12} to compute the Koszul dual  of an operad still holds mutatis mutandis for colored operads, see also \cite{VanderLaan03}. The only new point here is to  manage the non-binary parts. Notice that the operad $q\calO$ is  color-wise finitely generated.

The Koszul dual cooperad $q\calO^{\ac}$ is the quadratic (colored) cooperad with suspended presentation 
$$q\calO^{\ac}:=\calC(sE, s^2qR)\ . $$ 
The piecewise and colorwise linear dual $\left(q\calO^{\ac} \right)^*$ is a quadratic colored operad. 
By the same arguments as for \cite[Proposition~$7.2.1$]{LodayVallette12}, its operadic suspension 
$$q\mathcal{O}^! :=  (q\mathcal{O}^{\ac})^* \htensor \mathcal{S}^{-1}$$
is still a quadratic operad, 
where $\calS^{-1}$, in the colored setting, is defined by the endomorphism operad of the space 
${\{ \KK\,  s^{-1}_n   \}_{n\in \NN}}$. Its quadratic presentation 
is given by 
\begin{align*}
  q\mathcal{O}^! = \TTT\left(s^{-1}E(1)^* \oplus E(2)^*\otimes  \mathrm{sgn}_{\Sy_2}\right) / (qR^\perp)\ ,
\end{align*}
where the orthogonal space is obtained by the scalar product $\langle -; - \rangle$ defined on the arity $3$ part in \cite[Section~$7.6.3$]{LodayVallette12}. Using this, one proves the two first relations. 
By the same  arguments as for \cite[Theorem~$7.6.2$]{LodayVallette12}, the extension of the scalar product $\langle -; - \rangle$ to $\TTT\left(s^{-1}E(1)^* \oplus E(2)^*\otimes  \mathrm{sgn}_{\Sy_2}\right)(2)^{(2)}\otimes \TTT(E)(2)^{(2)}$ given by 
\begin{eqnarray*}
\langle \alpha^* \circ_1 \sigma^*, \mu \circ_{1} \tau \rangle:=  +\alpha^*(\mu)\sigma^*(\tau)\ \ \ \\
\langle \alpha^* \circ_2 \sigma^*, \mu \circ_{2} \tau \rangle:=  +\alpha^*(\mu)\sigma^*(\tau)\ \ \ \\
\langle \sigma^* \circ_1 \alpha^*, \tau \circ_1 \mu  \rangle:=  -\sigma^*(\tau)\alpha^*(\mu)\ , \\
\end{eqnarray*}
for $\alpha^*\in E(2)^*$, $\mu \in E(2)$, $\sigma^*\in E(1)^*$ and $\tau \in E(1)$,
allows one to compute the Koszul dual operad. This gives the last 3 relations and concludes the proof. 
\end{proof}

\begin{lemma}\label{lem:qO!}
The underlying colored $\Sy$-module of the operad $q\calO^!$ is isomorphic to 
$$q\calO^!\cong \calO_{ns} \circ T(s^{-1}\bar\Sy^*)\ ,  $$
where $T(s^{-1}\bar\Sy^*)$ is the free colored algebra generated by
$\left\{s^{-1}\sigma^*\ | \ \sigma \in \Sy_n\setminus\{\id_n\}, n\in \NN\right\}$.
\end{lemma}

\begin{proof}
This is a straightforward consequence of the Diamond Lemma for distributive laws \cite[Theorem~$8.6.5$]{LodayVallette12} and Proposition~\ref{prop:DistLawqO} since the colored operad $\calO_{ns}$ is Koszul auto-dual \cite[Theorem~$4.3$]{VanderLaan03}, that is $\calO_{ns}^!\cong \calO_{ns}$, and since the Koszul dual of the quadratic nilpotent algebra 
$\KK\{\Sy\}$
is the free algebra
$T(s^{-1}\bar\Sy^*)$.
\end{proof}

\begin{proposition}\label{pro:CompositeTreeBasisII}
A $\KK[\Sy]$-linear basis for the colored operad $q\calO^!$ is given by 
the left combs with nondecreasing  binary vertices,
seen top to bottom, and strings of symmetric groups elements labeling
the leaves.

\begin{align*}
  \begin{aligned}\begin{tikzpicture}[optree]
    \node{}
      child { node[circ]{$4$}
        child { node[circ]{$2$}
          child { node[circ]{$2$}
            child {
              child[level distance=5mm] { node[dot]{}
                child {
                  child {
                  edge from parent node[right,near end]{\tiny$3$} } } } }
            child { node[dot]{}
              child[level distance=5mm] { node[dot]{}
                child { node[dot]{}
                  child {
                  edge from parent node[right,near end]{\tiny$2$} } } } } }
          child {
            child { edge from parent[draw=none] }
            child {
              child[level distance=5mm] { node[dot]{}
                child {
                  child {
                  edge from parent node[right,near end]{\tiny$2$} } } } } } }
        child {
          child { edge from parent[draw=none] }
          child {
            child { edge from parent[draw=none] }
            child {
              child[level distance=5mm] { node[dot]{}
                child {
                  child {
                  edge from parent node[right,near end]{\tiny$3$} } } } } } } } ;
  \end{tikzpicture}\end{aligned}
\end{align*}
\end{proposition}
\noindent
(Since the context is obvious, we do not use the linear dual star notation anymore. 
To lighten the figure, we often omit to represent the labeled permutations.)

\begin{proof}
This is a direct corollary of Lemma~\ref{lem:qO!} since standard composite trees form a basis of the colored operad $\calO_{ns}$. 
\end{proof}

On this basis, the partial composition of the colored operad $q\calO^!$ is given by the grafting of trees as usual. 
Then one has to  rewrite the result in terms of 
standard composite trees
keeping track of the order on generators and sticking to the Koszul 
sign rule as in the examples below. In the sequel, we order the vertices of composite trees with leaves labeled by strings of permutations 
by, first, ordering all the binary vertices leftlevel-wise and then ordering the strings of permutations leftlevel-wise.

\begin{example}[Partial composition in $q\calO^!$]
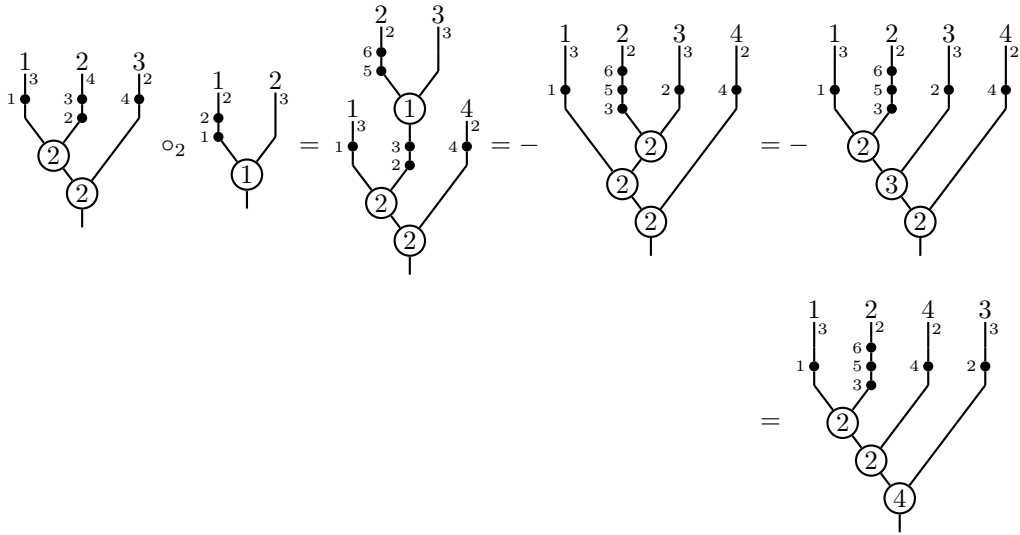
\label{Exam:PartialCompoCompoTrees}
\label{SqO!:EXcomp}
\begin{align*}
  \begin{aligned}\begin{tikzpicture}[optree]
    \node{}
      child { node[circ]{$2$}
        child { node[circ]{$2$}
          child {
            child[level distance=5mm] { node[dot,label=left:\tiny$1$]{}
              child[level distance=10mm] { node{$1$}
              edge from parent node[right,near end]{\tiny$3$} } } }
          child { node[dot,label=left:\tiny$2$]{}
            child[level distance=5mm] { node[dot,label=left:\tiny$3$]{}
              child[level distance=10mm] { node{$2$}
              edge from parent node[right,near end]{\tiny$4$} } } } }
        child {
          child { edge from parent[draw=none] }
          child {
            child[level distance=5mm] { node[dot,label=left:\tiny$4$]{}
              child[level distance=10mm] { node{$3$}
              edge from parent node[right,near end]{\tiny$2$} } } } } } ;
  \end{tikzpicture}\end{aligned}
  \circ_2
  \begin{aligned}\begin{tikzpicture}[optree]
    \node{}
      child { node[circ]{$1$}
        child { node[dot,label=left:\tiny$1$]{}
          child[level distance=5mm] { node[dot,label=left:\tiny$2$]{}
            child[level distance=10mm] { node{$1$}
            edge from parent node[right,near end]{\tiny$2$} } } }
        child {
          child[level distance=5mm] {
            child[level distance=10mm] { node{$2$}
            edge from parent node[right,near end]{\tiny$3$} } } } } ;
  \end{tikzpicture}\end{aligned}
  =
  \begin{aligned}\begin{tikzpicture}[optree]
    \node{}
      child { node[circ]{$2$}
        child { node[circ]{$2$}
          child {
            child[level distance=5mm] { node[dot,label=left:\tiny$1$]{}
              child[level distance=10mm] { node{$1$}
              edge from parent node[right,near end]{\tiny$3$} } } }
          child { node[dot,label=left:\tiny$2$]{}
            child[level distance=5mm] { node[dot,label=left:\tiny$3$]{}
              child[level distance=10mm] { node[circ]{$1$}
                child { node[dot,label=left:\tiny$5$]{}
                  child[level distance=5mm] { node[dot,label=left:\tiny$6$]{}
                    child[level distance=10mm] { node{$2$}
                    edge from parent node[right,near end]{\tiny$2$} } } }
                child {
                  child[level distance=5mm] {
                    child[level distance=10mm] { node{$3$}
                    edge from parent node[right,near end]{\tiny$3$} } } } } } } }
        child {
          child { edge from parent[draw=none] }
          child {
            child[level distance=5mm] { node[dot,label=left:\tiny$4$]{}
              child[level distance=10mm] { node{$4$}
              edge from parent node[right,near end]{\tiny$2$} } } } } } ;
  \end{tikzpicture}\end{aligned}
   =-
  \begin{aligned}\begin{tikzpicture}[optree]
    \node{}
      child { node[circ]{$2$}
        child { node[circ]{$2$}
          child {
            child {
              child[level distance=5mm] { node[dot,label=left:\tiny$1$]{}
                child {
                  child[level distance=10mm] { node{$1$}
                  edge from parent node[right,near end]{\tiny$3$} } } } }
            child { edge from parent[draw=none] } }
          child { node[circ]{$2$}
            child { node[dot,label=left:\tiny$3$]{}
              child[level distance=5mm] { node[dot,label=left:\tiny$5$]{}
                child { node[dot,label=left:\tiny$6$]{}
                  child[level distance=10mm] { node{$2$}
                  edge from parent node[right,near end]{\tiny$2$} } } } }
            child {
              child[level distance=5mm] { node[dot,label=left:\tiny$2$]{}
                child {
                  child[level distance=10mm] { node{$3$}
                  edge from parent node[right,near end]{\tiny$3$} } } } } } }
        child {
          child { edge from parent[draw=none] }
          child {
            child { edge from parent[draw=none] }
            child {
              child[level distance=5mm] { node[dot,label=left:\tiny$4$]{}
                child {
                  child[level distance=10mm] { node{$4$}
                  edge from parent node[right,near end]{\tiny$2$} } } } } } } } ;
  \end{tikzpicture}\end{aligned}
  &=-
  \begin{aligned}\begin{tikzpicture}[optree]
    \node{}
      child { node[circ]{$2$}
        child { node[circ]{$3$}
          child { node[circ]{$2$}
            child {
              child[level distance=5mm] { node[dot,label=left:\tiny$1$]{}
                child {
                  child[level distance=10mm] { node{$1$}
                  edge from parent node[right,near end]{\tiny$3$} } } } }
            child { node[dot,label=left:\tiny$3$]{}
              child[level distance=5mm] { node[dot,label=left:\tiny$5$]{}
                child { node[dot,label=left:\tiny$6$]{}
                  child[level distance=10mm] { node{$2$}
                  edge from parent node[right,near end]{\tiny$2$} } } } } }
          child {
            child { edge from parent[draw=none] }
            child {
              child[level distance=5mm] { node[dot,label=left:\tiny$2$]{}
                child {
                  child[level distance=10mm] { node{$3$}
                  edge from parent node[right,near end]{\tiny$3$} } } } } } }
        child {
          child { edge from parent[draw=none] }
          child {
            child { edge from parent[draw=none] }
            child {
              child[level distance=5mm] { node[dot,label=left:\tiny$4$]{}
                child {
                  child[level distance=10mm] { node{$4$}
                  edge from parent node[right,near end]{\tiny$2$} } } } } } } } ;
  \end{tikzpicture}\end{aligned}
  \\
  & =
  \begin{aligned}\begin{tikzpicture}[optree]
    \node{}
      child { node[circ]{$4$}
        child { node[circ]{$2$}
          child { node[circ]{$2$}
            child {
              child[level distance=5mm] { node[dot,label=left:\tiny$1$]{}
                child {
                  child[level distance=10mm] { node{$1$}
                  edge from parent node[right,near end]{\tiny$3$} } } } }
            child { node[dot,label=left:\tiny$3$]{}
              child[level distance=5mm] { node[dot,label=left:\tiny$5$]{}
                child { node[dot,label=left:\tiny$6$]{}
                  child[level distance=10mm] { node{$2$}
                  edge from parent node[right,near end]{\tiny$2$} } } } } }
          child {
            child { edge from parent[draw=none] }
            child {
              child[level distance=5mm] { node[dot,label=left:\tiny$4$]{}
                child {
                  child[level distance=10mm] { node{$4$}
                  edge from parent node[right,near end]{\tiny$2$} } } } } } }
        child {
          child { edge from parent[draw=none] }
          child {
            child { edge from parent[draw=none] }
            child {
              child[level distance=5mm] { node[dot,label=left:\tiny$2$]{}
                child {
                  child[level distance=10mm] { node{$3$}
                  edge from parent node[right,near end]{\tiny$3$} } } } } } } } ;
  \end{tikzpicture}\end{aligned}
\end{align*}
(Assuming the permutation at $2$ is $[1 3 4 2]$ and at $3$ it is $[4 1 2 3]$.) In this example, we  reorder the
labeling permutations bottom to top
in leaf order. Following the Koszul sign rule, the sign is that of the permutation
$[1 3 5 6 4 2]$, i.e. $+1$.
\end{example}

\begin{proposition}\label{prop:OperadicTreesBasisII}
The set of planar trees  with vertices labeled by (possibly empty) strings of permutations on the number of leaves of the vertex
forms another $\KK[\Sy]$-linear basis of the colored operad $q\calO^!$.
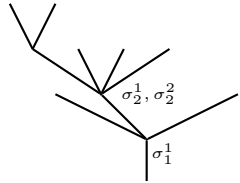
\begin{align*}
  \begin{aligned}\begin{tikzpicture}[optree,
      level distance=12mm,
      level 2/.style={sibling distance=12mm},
      level 3/.style={sibling distance=12mm}]
    \node{}
      child { node[comp,label={-3:{\tiny $\sigma^1_1$}}]{}
        child
        child { node[comp,label={[label distance=2mm]1:{\tiny $\sigma^1_2,\sigma^2_2$}}]{}
          child {
            child
            child }
          child
          child
          child }
        child { edge from parent[draw=none] }
        child { edge from parent[draw=none] }
        child } ;
  \end{tikzpicture}\end{aligned}
\end{align*}
\end{proposition}

\begin{proof}
The proof is a consequence of  Proposition~\ref{prop:OperadicTreesBasis} for the "tree" part and
Proposition~\ref{pro:CompositeTreeBasisII} for the "permutations" part: strings of symmetric group elements on top of
composite trees correspond to strings of symmetric group elements labeling the vertices of operadic trees as follows
\begin{align*}
  \begin{aligned}\begin{tikzpicture}[optree]
    \node{}
    child[level distance=7mm] { node[dot,label=right:$\sigma^l$]{}
      child { node[dot,label={[xshift=.5mm,yshift=.5mm]right:$\vdots$}]{}
        child { node[dot,label=right:$\sigma^1$]{}
          child { edge from parent node[left,near end]{\tiny$n$} }
        }
      }
    } ;
  \end{tikzpicture}\end{aligned}
  \quad\longrightarrow\quad
  (-1)^{\frac{l(l-1)}{2}}
  \begin{aligned}\begin{tikzpicture}[optree,
      level distance=12mm,
      level 2/.style={sibling distance=12mm},
      level 3/.style={sibling distance=6mm}]
    \node{}
      child { node[comp,label={[label distance=1mm]1:{\tiny$\sigma^1,\ldots,\sigma^l$}}]{}
        child { node{\tiny $1$} }
        child { edge from parent[draw=none] node{$\,\cdots$} }
        child { node{\tiny $n$} }
      } ;
  \end{tikzpicture}\end{aligned}
\end{align*}
where the sign comes from the reordering of $s$ degree $-1$ terms.
\end{proof}

The vertices of the operadic trees are ordered by the left-recursive order. Hence we denote such a labeled tree by 
$t(\bar\sigma_1, \ldots, \bar \sigma_n)$, where $t$ stands for the underlying planar tree and where each 
$\bar \sigma_j=(\sigma^1_j, \ldots, \sigma_j^{i_j})$ is a (possibly empty) string of permutations.  
The above example is  $t((\sigma_1^1), (\sigma_2^1, \sigma_2^2), \emptyset)$.
\\
 
The bijection, mentioned in Section~\ref{subsec:TOS}, between standard composite trees and operadic trees induces  a bijection between the two above bases of the colored operad $\calO^!$:
 we add to the leaves of the standard composite tree the strings of 
symmetric group elements found at the corresponding vertices of the 
operadic tree. Notice that the change of order for the strings permutations induces a sign. 
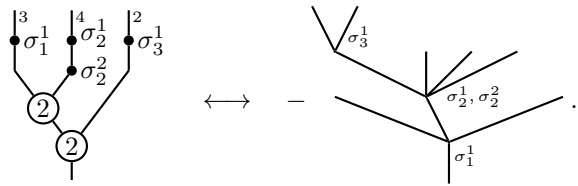
\begin{align*}
  \begin{aligned}\begin{tikzpicture}[optree]
    \node{}
      child { node[circ]{$2$}
        child { node[circ]{$2$}
          child {
            child[level distance=8mm] { node[dot,label=right:$\sigma_1^1$]{}
              child { edge from parent node[right,near end]{\tiny$3$} } } }
          child { node[dot,label={[yshift=0mm]right:$\sigma_2^2$}]{}
            child[level distance=8mm] { node[dot,label={[yshift=1mm]right:$\sigma_2^1$}]{}
              child { edge from parent node[right,near end]{\tiny$4$} } } } }
        child {
          child { edge from parent[draw=none] }
          child {
            child[level distance=8mm] { node[dot,label=right:$\sigma_3^1$]{}
              child { edge from parent node[right,near end]{\tiny$2$} } } } } } ;
  \end{tikzpicture}\end{aligned}
  \quad\longleftrightarrow\quad -
\begin{aligned}\begin{tikzpicture}[optree,
      level distance=12mm,
      level 2/.style={sibling distance=12mm},
      level 3/.style={sibling distance=12mm}]
    \node{}
      child { node[comp,label={-3:{\tiny$\sigma^1_1$}}]{}
        child
        child { edge from parent[draw=none] }
        child { node[comp,label={[label distance=2mm]1:{\tiny $\sigma^1_2,\sigma^2_2$}}]{}
          child { node[comp,label={[label distance=1mm]3:{\tiny $\sigma_3^1$}}]{}
            child
            child }
        child { edge from parent[draw=none] }
          child
          child
          child }
        child { edge from parent[draw=none] }
        child { edge from parent[draw=none] }
        child } ;
  \end{tikzpicture}\end{aligned}  \ .
\end{align*}

Under this bijection, the partial composition of the colored operad $q\calO^!$ can also be described on the operadic trees basis. 
Since the leaves of composite trees correspond to the vertices of operadic trees, and since the binary generators correspond to the internal edges, the partial composition amounts to the substitution of operadic trees. 

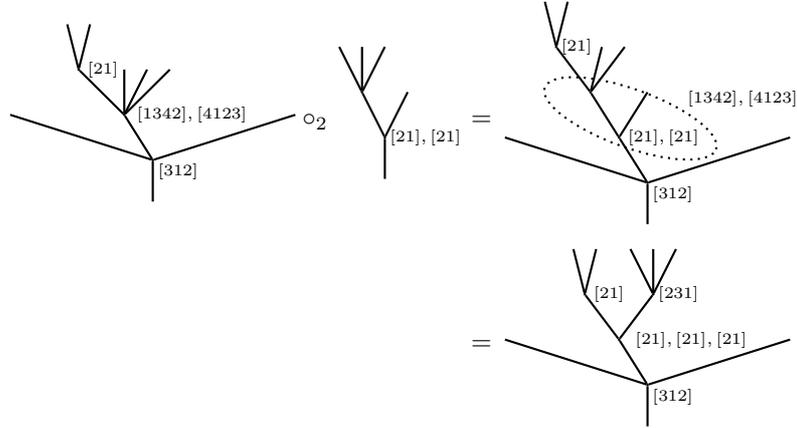
\begin{example}\label{ex:PartiCompoOpTree}
\begin{align*}
  \begin{aligned}\begin{tikzpicture}[optree,
      level distance=12mm,
      level 2/.style={sibling distance=15mm},
      level 3/.style={sibling distance=6mm}]
    \tiny
    \node{}
      child { node[comp,label={-3:$[3 1 2]$}]{}
        child
        child { edge from parent[draw=none] }        
        child { node[comp,label={[label distance=1mm]right:$[1 3 4 2], [4 1 2 3]$}]{}
          child { node[comp,label={[label distance=0.5mm]2:$[2 1]$}]{}
            child
            child }
          child { edge from parent[draw=none] }
          child
          child
          child }
        child { edge from parent[draw=none] }
        child { edge from parent[draw=none] }
        child } ;
  \end{tikzpicture}\end{aligned}
  \circ_2
  \begin{aligned}\begin{tikzpicture}[optree,
      level distance=12mm,
      level 2/.style={sibling distance=12mm},
      level 3/.style={sibling distance=6mm}]
    \tiny
    \node{}
      child { node[comp,label={right:$[2 1], [2 1]$}]{}
        child { node[comp]{}
          child
          child
          child }
        child } ;
  \end{tikzpicture}\end{aligned}
  &=
  \begin{aligned}\begin{tikzpicture}[optree,
      level distance=12mm,
      level 2/.style={sibling distance=15mm},
      level 3/.style={sibling distance=15mm},
      level 4/.style={sibling distance=6mm}]
    \tiny
    \node{}
      child { node[comp,label={-3:$[3 1 2]$}]{}
        child
        child { edge from parent[draw=none] }
        child { node[comp] (a) {}
          child { node[comp] (b) {}
            child { node[comp,label={right:$[2 1]$}]{}
              child
              child }
            child { edge from parent[draw=none] }
            child
            child }
          child }
        child { edge from parent[draw=none] }
        child { edge from parent[draw=none] }
        child } ;
    \node[rectangle,right=.5mm of a] (c) {$[2 1], [2 1]$};
    \node[draw,dotted,ellipse,rotate=-20,inner sep=-5mm,xshift=-2mm,xscale=0.85,yshift=-0.5mm,fit=(a) (b) (c),label={25:$[1 3 4 2], [4 1 2 3]$}] {};
  \end{tikzpicture}\end{aligned}\\
& = 
  \begin{aligned}\begin{tikzpicture}[optree,
      level distance=12mm,
      level 2/.style={sibling distance=15mm},
      level 3/.style={sibling distance=18mm},
      level 4/.style={sibling distance=6mm}]
    \tiny
    \node{}
      child { node[comp,label={-3:$[3 1 2]$}]{}
        child
        child { edge from parent[draw=none] }        
        child { node[comp,label={[label distance=1.5mm]right:$[2 1], [2 1], [2 1]$}] {}
          child { node[comp,label={[label distance=0.5mm]right:$[2 1]$}]{}
            child
            child }
          child { node[comp,label={right:$[2 3 1]$}] {}
            child
            child 
            child } }
        child { edge from parent[draw=none] }
        child { edge from parent[draw=none] }
        child } ;
  \end{tikzpicture}\end{aligned}
\end{align*}
This example depicts  the same partial composition as in Example~\ref{Exam:PartialCompoCompoTrees} but in terms of operadic trees this time. 
\end{example}
The main advantage here of operadic trees over composite trees lies in the fact that in the colored operad $q\calO^!$ the
 former ones have no relations, on the opposite to the latter ones: they are related by
 Relations~\eqref{Align:Relation1a} and \eqref{Align:Relation1b}.

\subsection{Koszul dual cooperad of the quadratic analogue}\label{subsec:KDCoopQ}
The Koszul dual cooperad $q\calO^\ac$ is the linear dual of
the operad $q\calO^!\htensor \calS$, as explained above in the proof of Proposition~\ref{prop:KoszulDualOp}.
This latter operad admits the same presentation as the operad $q\calO^!$ except that all the generators are now in degree $-1$ and that all the relations hold up to a minus sign. So the degree of the underlying planar tree $t$  of an operadic tree depicting a basis element of the cooperad $q\calO^\ac$ is equal to  the number of its internal edges; we denote it by $|t|$. 

Dualizing the aforementioned partial composition of the operad $q\calO^!$, we get the infinitesimal decomposition map $\Delta_{(1)}$ of the cooperad  $q\calO^\ac$ as follows. Let $t(\bar\sigma_1, \ldots, \bar \sigma_n)\in q\calO^\ac$ be an operadic tree basis element. Let $s\subset t$ be any subtree, which can be just a corolla. It is completely characterized by the associated subset $\{ s_1, \ldots, s_k\}\subset \{1, \ldots, n\}$ of vertices of $t$. The infinitesimal coproduct of $q\calO^{\ac}$ extracts any such subtree $s$ of $t$ and produces the tree $t/s$ obtained by contracting $s$ inside $t$, i.e. replacing the subset $s$ by a corolla of same arity. 

The strings of permutations $\{ \bar\sigma_{s_1}, \ldots, \bar\sigma_{s_k}\}$ labeling the vertices  of $s$ split into
two parts, in all possible ways, even empty. For this, we use the notation
$$\bar\sigma=(\cev{\sigma},\vec{\sigma})\ , $$
where the left-hand parts $\cev{\sigma}$ remain in $s$. (It can be empty only when the subtree $s$ is not a corolla). The right-hand parts $\vec{\sigma}$ leave $s$ one by one in all
possible orders, change its shape according to Relation~$(2)$ and label the new corolla in the contracted tree
$t/s$.  We denote the new string of permutations by $\shuffle(\vec{\sigma}_{s_1}, \ldots, \vec{\sigma}_{s_k})$ because
all the shuffles of the $k$ strings appear in the formula. (We should use here notation "prime" and "second" for these
induced permutations.) We denote by $\tilde{s}$, the new subtree $s$.  The other vertices of $t$ remain labeled by the
same strings of permutations. The overall sign is denoted by $\varepsilon_{s,t, \shuffle}$.

\begin{proposition}\label{prop:InfDecomp}
The   infinitesimal decomposition map of the cooperad $q\calO^\ac$ is equal to 
  \begin{align*}
    \Delta_{(1)}\big(t(\bar\sigma_1, \ldots, \bar \sigma_n)\big) =
      \sum_{s\subseteq t,  \shuffle} 
      \varepsilon_{s,t, \shuffle}\, 
      \ t/s(\bar{\sigma}_1,\ldots,
      \shuffle(\vec{\sigma}_{s_1}, \ldots, \vec{\sigma}_{s_k}),
      \ldots, \bar{\sigma}_n) 
      \otimes
      \tilde{s}(\cev{\sigma}_{s_1},\ldots,\cev{\sigma}_{s_k}) \ .
  \end{align*}
\end{proposition}

\begin{proof}
This formula is the linear dual of the partial composition map of the operad $q\calO^!$.
\end{proof}

Let us explain this procedure on an example and give the computation of the associated signs at the same time. 
Recall that the signs all come from the left-recursive order on composite trees. 

\begin{example}[Infinitesimal decomposition map of $q\calO^{\ac}$]\label{Exam:InfDecCoop}
We consider the operadic trees 
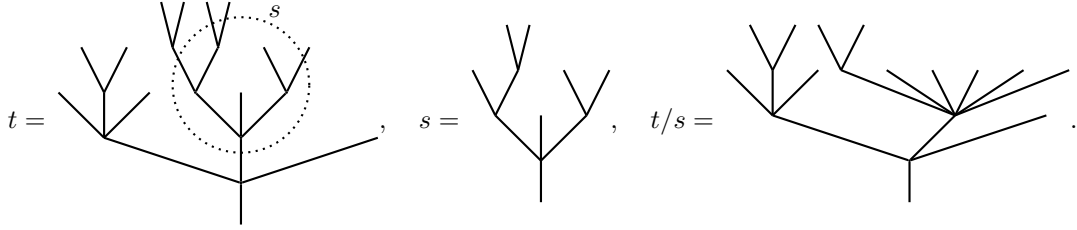
\begin{align*}
  t &=
  \begin{aligned}\begin{tikzpicture}[optree,
      level distance=12mm,
      level 2/.style={sibling distance=36mm},
      level 3/.style={sibling distance=12mm},
      level 4/.style={sibling distance=12mm},
      level 5/.style={sibling distance=6mm}]
    \node{}
      child { node[comp]{}
        child {
          child
          child {
            child
            child
          }
          child
        }
        child { node[comp] (a) {}
          child { node[comp] (b) {}
            child {
              child
              child
            }
            child { node[comp] (c) {}
              child
              child
            }
          }
          child
          child { node[comp] (d) {}
            child
            child
          }
        }
        child
      };
    \node[draw,dotted,circle,inner sep=-6mm,yshift=1mm,fit=(a) (b) (c) (d),label={70:$s$}] {};
  \end{tikzpicture}\end{aligned} ,&
  s &=
  \begin{aligned}\begin{tikzpicture}[optree,
      level distance=12mm,
      level 2/.style={sibling distance=12mm},
      level 3/.style={sibling distance=12mm},
      level 4/.style={sibling distance=6mm}]
    \tiny
    \node{}
      child {
        child {
          child
          child {
            child
            child
          }
        }
        child
        child {
          child
          child
        }
      } ;
  \end{tikzpicture}\end{aligned} ,&
  t/s &=
  \begin{aligned}\begin{tikzpicture}[optree,
      level distance=12mm,
      level 2/.style={sibling distance=24mm},
      level 3/.style={sibling distance=12mm},
      level 4/.style={sibling distance=12mm},
      level 5/.style={sibling distance=6mm}]
    \node{}
      child { node[comp]{}
        child {
          child
          child {
            child
            child
          }
          child
        }
        child { edge from parent[draw=none] }
        child {
          child {
            child
            child
          }
          child
          child
          child
          child
          child
        }
        child
      };
  \end{tikzpicture}\end{aligned} .
\end{align*}
We ommited the symmetric group elements in the above pictures. See the composite tree pictures below.

\begin{asparaenum}[\it{Step} 1.]
\item We associate to the operadic tree $t(\bar\sigma_1, \ldots, \bar \sigma_n)$ its standard composite tree. The
  reversal of order in the strings of permutations induces the sign 
  $$(-1)^{\frac{|\bar\sigma_1|(|\bar\sigma_1|-1)}{2}+\cdots +\frac{|\bar\sigma_n|(|\bar\sigma_n|-1)}{2} } \ .$$

\item  We isolate the subtree $s$ inside $t$ as follows. The label $s_1$ is the smallest index of a vertex of $s$
  viewed inside $t$. We relabel the vertices of $t$ by, first, keeping the labels $1, \ldots, s_1-1$, then, labeling
  only the vertices of $s$ by $s_1, s_1+1, \ldots, s_1+k-1$, and, finally, by labeling the rest of the vertices of
  $t$. This implies a permutation of vertices, denoted by $\nu$, and an induced permutation of strings of symmetric
  group elements, which all have degree $1$. The induced sign is the signature of $\nu$ times the sign produces by the
  Koszul sign rule. We denote this sign by $\mathrm{sgn}(\nu(|\bar\sigma_1|, \ldots, |\bar \sigma_n|))$.

  On the level of the equivalent standard composite tree, we rewrite it as a left-comb but isolating the subtree $s$.
  The signature of the permutation $\nu$ comes from the use of the (signed) Relation~\eqref{Align:Relation1a}.

\item  We extract the subtree $s$ out of the tree $t$. On the level of the left-comb composite tree, this amounts to
  using Relations~\eqref{Align:Relation1b} and \eqref{Align:Relation1a} to write the tree $t$ with a first left-comb
  giving $t/s$ and an attached parallel left-comb giving the subtree $s$. Then, we pull the left-comb representing $s$ above the one representing $t/s$ with its strings of permutations. The sign for the underlying trees is $(-1)^{|s|(s_1-1)}$; it should be multiplied with the sign coming from the change of order of the strings of permutations. 

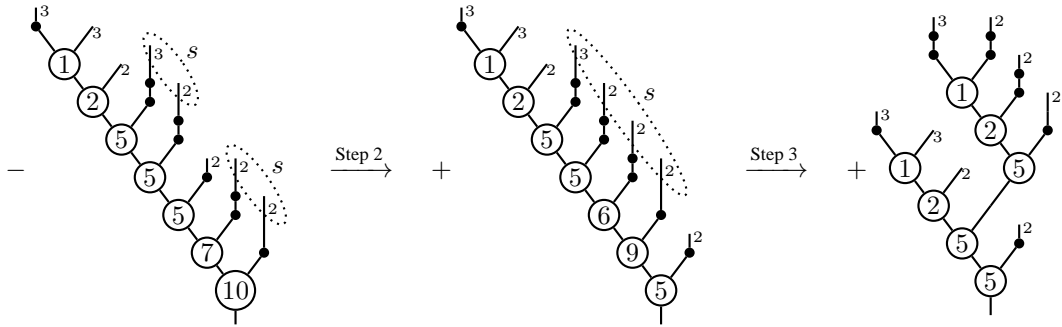
\begin{align*}
-  \begin{aligned}\begin{tikzpicture}[optree]
    \node{}
      child { node[circ]{$10$}
        child { node[circ]{$7$}
          child { node[circ]{$5$}
            child { node[circ]{$5$}
              child { node[circ]{$5$}
                child { node[circ]{$2$}
                  child { node[circ]{$1$}
                    child { node[dot]{}
                      child[level distance=5mm] { edge from parent node[right,near end]{\tiny$3$} }
                    }
                    child { edge from parent node[right,near end]{\tiny$3$} }
                  }
                  child { edge from parent node[right,near end]{\tiny$2$} }
                }
                child { node[dot]{}
                  child[level distance=5mm] { node[dot] {}
                    child[level distance=10mm] { edge from parent node[right,near end] (a) {\tiny$3$} node (a2) {} }
                  }
                }
              }
              child { node[dot]{}
                child[level distance=5mm] { node[dot]{}
                  child[level distance=10mm] { edge from parent node[right,near end] (b) {\tiny$2$} node (b2) {} }
                }
              }
            }
            child { node[dot]{}
              child[level distance=5mm] { edge from parent node[right,near end]{\tiny$2$} }
            }
          }
          child { node[dot]{}
            child[level distance=5mm] { node[dot]{}
              child[level distance=10mm] { edge from parent node[right,near end] (c) {\tiny$2$} }
            }
          }
        }
        child { node[dot]{}
          child[level distance=15mm] { edge from parent node[right,near end] (d) {\tiny$2$} }
        }
      } ;
    \node[draw,dotted,ellipse,rotate=35,inner sep=-3.5mm,yshift=.5mm,xscale=.5,yscale=1.05,fit=(a) (b),label={0:$s$}] {};
    \node[draw,dotted,ellipse,rotate=35,inner sep=-3.5mm,yshift=.5mm,xscale=.5,fit=(c) (d),label={0:$s$}] {};
  \end{tikzpicture}\end{aligned}
  && \xrightarrow{\text{Step 2}} &&
 + \begin{aligned}\begin{tikzpicture}[optree]
    \node{}
      child { node[circ]{$5$}
        child { node[circ]{$9$}
          child { node[circ]{$6$}
            child { node[circ]{$5$}
              child { node[circ]{$5$}
                child { node[circ]{$2$}
                  child { node[circ]{$1$}
                    child { node[dot]{}
                      child[level distance=5mm] { edge from parent node[right,near end]{\tiny$3$} }
                    }
                    child { edge from parent node[right,near end]{\tiny$3$} }
                  }
                  child { edge from parent node[right,near end]{\tiny$2$} }
                }
                child { node[dot]{}
                  child[level distance=5mm] { node[dot] {}
                    child[level distance=10mm] { edge from parent node[right,near end] (a) {\tiny$3$} node (a2) {} }
                  }
                }
              }
              child { node[dot]{}
                child[level distance=5mm] { node[dot]{}
                  child[level distance=10mm] { edge from parent node[right,near end] (b) {\tiny$2$} node (b2) {} }
                }
              }
            }
            child { node[dot]{}
              child[level distance=5mm] { node[dot]{}
                child[level distance=10mm] { edge from parent node[right,near end] (c) {\tiny$2$} }
              }
            }
          }
          child { node[dot]{}
            child[level distance=15mm] { edge from parent node[right,near end] (d) {\tiny$2$} node (s1) {} }
          }
        }
        child { node[dot]{}
          child[level distance=5mm] { edge from parent node[right,near end]{\tiny$2$} }
        }
      } ;
    \node[draw,dotted,ellipse,rotate=35,inner sep=-8mm,xscale=.25,yscale=.95,fit=(a) (b) (c) (d),label={0:$s$}] {};
  \end{tikzpicture}\end{aligned}
  && \xrightarrow{\text{Step 3}} &&
+  \begin{aligned}\begin{tikzpicture}[optree]
    \node{}
      child { node[circ]{$5$}
        child { node[circ]{$5$}
          child { node[circ]{$2$}
            child { node[circ]{$1$}
              child { node[dot]{}
                child[level distance=5mm] { edge from parent node[right,near end]{\tiny$3$} }
              }
              child { edge from parent node[right,near end]{\tiny$3$} }
            }
            child { edge from parent node[right,near end]{\tiny$2$} }
          }
          child {
            child { edge from parent[draw=none] }
            child { node[circ]{$5$}
              child { node[circ]{$2$}
                child { node[circ]{$1$}
                  child { node[dot]{}
                    child[level distance=5mm] { node[dot] (s0) {}
                      child { edge from parent node[right,near end] (a) {\tiny$3$} node (a2) {} }
                    }
                  }
                  child { node[dot]{}
                    child[level distance=5mm] { node[dot]{}
                      child { edge from parent node[right,near end] (b) {\tiny$2$} node (b2) {} }
                    }
                  }
                }
                child { node[dot]{}
                  child[level distance=5mm] { node[dot]{}
                    child { edge from parent node[right,near end] (c) {\tiny$2$} }
                  }
                }
              }
              child { node[dot]{}
                child[level distance=5mm] {
                  child { edge from parent node[right,near end] (d) {\tiny$2$} node (s1) {} }
                }
              }
            }
          }
        }
        child { node[dot]{}
          child[level distance=5mm] { edge from parent node[right,near end]{\tiny$2$} }
        }
      } ;
  \end{tikzpicture}\end{aligned}
\end{align*}

\item  For each choice of sub-strings labeling the vertices of $s$ and for each shuffle of them, we pull them down and
  we change them into $\sigma'$ or $\sigma''$ at each step. And we push down these elements inside the bottom tree representing $t/s$. The sign is two-fold: each use of (signed) Relation~$(2)$ creates a
  sign and the permutations of the order of the symmetric group elements produce a sign.

\item  By pulling down the symmetric group elements using Relation~$(2)$, the labels of the binary vertices
  of the standard left-comb giving $s$ might be changed (second case of Relation~$(2)$). So we might have to rewrite
  this left-comb into a standard form, i.e. with nondecreasing binary vertices from top to bottom, using
  Relation~\eqref{Align:Relation1a}; this induces a sign similar to the sign in Step 1.

  In our example, assume that $\rho_1 = [321]$, $\rho_2 = \rho_3 = [21]$. In the pictures we then use the
  $(1,0,2,0)$-shuffle $[213]$ to pull down the symmetric group elements. This induces permutations (in this order)
  $\tilde\rho_2 = [132456]$, $\tilde\rho_1 = [456312]$ and $\tilde\rho_3 = [123465]$.

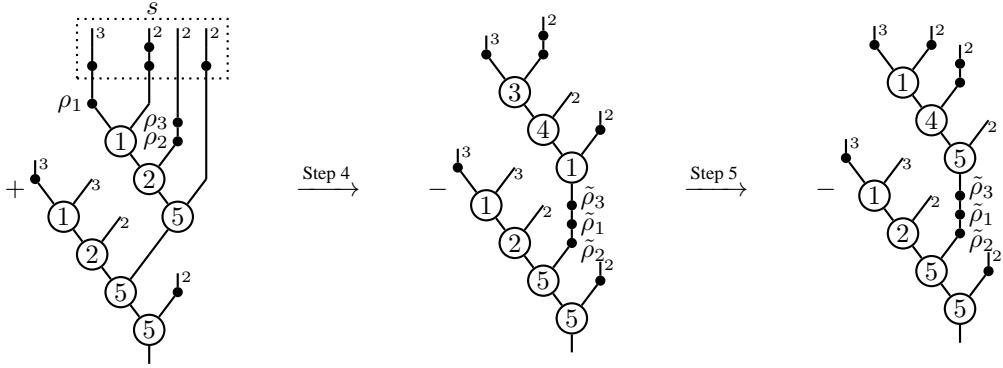
\begin{align*}
+  \begin{aligned}\begin{tikzpicture}[optree]
    \node{}
      child { node[circ]{$5$}
        child { node[circ]{$5$}
          child { node[circ]{$2$}
            child { node[circ]{$1$}
              child { node[dot]{}
                child[level distance=5mm] { edge from parent node[right,near end]{\tiny$3$} }
              }
              child { edge from parent node[right,near end]{\tiny$3$} }
            }
            child { edge from parent node[right,near end]{\tiny$2$} }
          } 
          child {
            child { edge from parent[draw=none] }
            child { node[circ]{$5$}
              child { node[circ]{$2$}
                child { node[circ]{$1$}
                  child { node[dot,label=left:$\rho_1$]{}
                    child[level distance=10mm] { node[dot] (s0) {}
                      child[level distance=5mm] {
                        child { edge from parent node[right,near end] (a) {\tiny$3$} node (a2) {} }
                      }
                    }
                  }
                  child {
                    child[level distance=10mm] { node[dot]{}
                      child[level distance=5mm] { node[dot]{}
                        child { edge from parent node[right,near end] (b) {\tiny$2$} node (b2) {} }
                      }
                    }
                  }
                }
                child { node[dot,label=left:$\rho_2$]{}
                  child[level distance=5mm] { node[dot,label=left:$\rho_3$]{}
                    child[level distance=20mm] {
                      child[level distance=5mm] { edge from parent node[right,near end] (c) {\tiny$2$} }
                    }
                  }
                }
              }
              child {
                child[level distance=30mm] { node[dot]{}
                  child[level distance=5mm] {
                    child { edge from parent node[right,near end] (d) {\tiny$2$} node (s1) {} }
                  }
                }
              }
            }
          }
        }
        child { node[dot]{}
          child[level distance=5mm] { edge from parent node[right,near end]{\tiny$2$} }
        }
      } ;
    \node[draw,dotted,rectangle,fit=(s0) (s1),xshift=.5mm,yshift=.5mm,xscale=.6,yscale=.75,label={90:$s$}] {};
  \end{tikzpicture}\end{aligned}
  && \xrightarrow{\text{Step 4}} &&
-  \begin{aligned}\begin{tikzpicture}[optree]
    \node{}
      child { node[circ]{$5$}
        child { node[circ]{$5$}
          child { node[circ]{$2$}
            child { node[circ]{$1$}
              child { node[dot]{}
                child[level distance=5mm] { edge from parent node[right,near end]{\tiny$3$} }
              }
              child { edge from parent node[right,near end]{\tiny$3$} }
            }
            child { edge from parent node[right,near end]{\tiny$2$} }
          }
          child { node[dot,label={[yshift=-1mm]right:$\tilde\rho_2$}]{}
            child[level distance=5mm] { node[dot,label=right:$\tilde\rho_1$]{}
              child { node[dot,label={[yshift=1mm]right:$\tilde\rho_3$}]{}
                child[level distance=10mm] { node[circ]{$1$}
                  child { node[circ]{$4$}
                    child { node[circ]{$3$}
                      child { node[dot]{}
                        child[level distance=5mm] { edge from parent node[right,near end] {\tiny$3$} }
                      }
                      child { node[dot]{}
                        child[level distance=5mm] { node[dot]{}
                          child { edge from parent node[right,near end] {\tiny$2$} {} }
                        }
                      }
                    }
                    child { edge from parent node[right,near end] {\tiny$2$} }
                  }
                  child { node[dot]{}
                    child[level distance=5mm] { edge from parent node[right,near end] {\tiny$2$} }
                  }
                }
              }
            }
          }
        }
        child { node[dot]{}
          child[level distance=5mm] { edge from parent node[right,near end]{\tiny$2$} }
        }
      } ;
  \end{tikzpicture}\end{aligned}
  && \xrightarrow{\text{Step 5}} &&
-  \begin{aligned}\begin{tikzpicture}[optree]
    \node{}
      child { node[circ]{$5$}
        child { node[circ]{$5$}
          child { node[circ]{$2$}
            child { node[circ]{$1$}
              child { node[dot]{}
                child[level distance=5mm] { edge from parent node[right,near end]{\tiny$3$} }
              }
              child { edge from parent node[right,near end]{\tiny$3$} }
            }
            child { edge from parent node[right,near end]{\tiny$2$} }
          }
          child { node[dot,label={[yshift=-1mm]right:$\tilde\rho_2$}]{}
            child[level distance=5mm] { node[dot,label=right:$\tilde\rho_1$]{}
              child { node[dot,label={[yshift=1mm]right:$\tilde\rho_3$}]{}
                child[level distance=10mm] { node[circ]{$5$}
                  child { node[circ]{$4$}
                    child { node[circ]{$1$}
                      child { node[dot]{}
                        child[level distance=5mm] { edge from parent node[right,near end] {\tiny$3$} }
                      }
                      child { node[dot]{}
                        child[level distance=5mm] { edge from parent node[right,near end] {\tiny$2$} {} }
                      }
                    }
                    child { node[dot]{}
                      child[level distance=5mm] { node[dot]{}
                        child { edge from parent node[right,near end] {\tiny$2$} }
                      }
                    }
                  }
                  child { edge from parent node[right,near end] {\tiny$2$} }
                }
              }
            }
          }
        }
        child { node[dot]{}
          child[level distance=5mm] { edge from parent node[right,near end]{\tiny$2$} }
        }
      } ;
  \end{tikzpicture}\end{aligned}
\end{align*}

\item  We write the operadic tree   $\tilde{s}(\cev{\sigma}_{s_1},\ldots,\cev{\sigma}_{s_k})$ corresponding to the top standard composite tree 
 on the right-hand side of the 
 operadic tree 
   $t/s(\bar{\sigma}_1,\ldots, \shuffle(\vec{\sigma}_{s_1}, \ldots, \vec{\sigma}_{s_k}), \ldots, \bar{\sigma}_n)$
 corresponding to the bottom
 standard composite tree. There is a sign coming from the bijection  operadic trees and  standard composite trees with strings of permutations. 
 \begin{align*}
  && \xrightarrow{\text{Step 6}} && 
-\   t/s\big((\bullet), \emptyset, \emptyset, (\tilde\rho_3, \tilde\rho_1, \tilde\rho_2), (\bullet)\big) \otimes  
  \tilde{s}\big((\bullet), (\bullet), (\bullet, \bullet), \emptyset \big)\ .
\end{align*}
\end{asparaenum}

\end{example}

\subsection{Curved Koszul dual cooperad}

To apply the curved Koszul duality theory to the colored operad $\calO$, one  has to check first the following conditions on the presentation of this operad. 

\begin{lemma}\label{lem:Conditions12}
The following conditions hold
\begin{enumerate}
  \item[{\rm (I)}] Minimality of generators: $R \cap \{I \oplus E\} = \{0\}$\ ,
  \item[{\rm (II)}] Maximality of relations: 
    $(R) \cap \{I \oplus E \oplus \TTT(E)^{(2)}\} = R$\ .
\end{enumerate}
\end{lemma}

\begin{proof}
The first condition is obviously satisfied. 
Regarding the second condition, since the operad $\mathcal O$ is the free $\KK$-linear version of a set-theoretical, any element of 
$(R) \cap \{I \oplus E \oplus \TTT(E)^{(2)}\}$ is a linear combinaison of elements of the form $x_1-x_2$, where 
$x_1, x_2$ are set-theoretical elements of  $\in I \oplus E \oplus \TTT(E)^{(2)}$ which give the same element in $\mathcal{O}$. 
If they are of arity $1$, then one is made up of two elements $\sigma$ and $\tau$ of a symmetric group $\Sy_n$. So, the other one is either the identity if $\tau=\sigma^{-1}$ or made up by the element $\sigma \tau$. In both cases, $x_1-x_2$ lives in $R$ by Relations~(\ref{Orel1a})-(\ref{Orel1b}). 
If they are of arity $2$, they are both made up of one binary generator of $E$ and one element of the symmetric group. The three possible cases are covered by Relations~(\ref{Orel2a})-(\ref{Orel2b}).
If they are of arity $3$,  they are both made up of two binary generators of $E$ and Relations~(\ref{Orel3a})-(\ref{Orel3b}) express all the possible ways to write such elements. Finally, we have shown that 
$(R) \cap \{I \oplus E \oplus \TTT(E)^{(2)}\}\subset R$, which concludes the proof. 
\end{proof}

Under Conditions (I) and (II), one can endow  the quadratic Koszul dual cooperad $q\calO^\ac$ with a curved cooperad structure
$$\calO^\ac := (q\calO^\ac, \dif_{\calO^\ac},\theta)\ , $$
defining the  \emph{Koszul 
dual curved cooperad}. Such a structure amonts to a \emph{curvature} $\theta : q\calO^\ac \to I$, which is a degree $-2$ map satisfying $\theta \circ \dif_{\calO^\ac}=0$ and controlling the default for the degree $-1$ coderivation $\dif_{\calO^\ac}$ to square to zero: 
$$(\dif_{\calO^\ac})^2=(\id \otimes \theta - \theta\otimes \id)\circ \Delta_{(1)}\ .$$

Condition (I) ensures that there exists a map $\varphi : qR \to E \oplus I$ such that the space of relations $R$ is the graph of $\varphi$, i.e. $R=\{ X - \varphi(X), X \in qR\}$. Let us denote the two summands of this map by $\varphi=\varphi_1+\varphi_0$, where $\varphi_1 : qR\to E$ and $\varphi_0 : qR\to I$ are equal to 
$$\varphi_1 \ : \ \begin{aligned}
  \begin{tikzpicture}[optree]
    \node{}
      child { node[dot,label=right:$\tau$]{}
        child { node[dot,label=right:$\sigma$]{}
child } } ;
  \end{tikzpicture}\end{aligned}
\   \mapsto\ 
  \begin{aligned}\begin{tikzpicture}[optree]
    \node{}
      child { node[dot,label=right:$\sigma\tau$]{}
        child  } ;
  \end{tikzpicture}\end{aligned}
  , \mathrm{for}\ \sigma \tau\neq \id,  \ \mathrm{and} \ 
\varphi_0 \ : \   \begin{aligned}\begin{tikzpicture}[optree]
    \node{}
      child { node[dot,label=right:$\sigma^{-1}$]{}
        child { node[dot,label=right:$\sigma$]{}
         child } } ;
  \end{tikzpicture}\end{aligned}
\ \mapsto \ 
  \begin{aligned}\begin{tikzpicture}[optree]
    \node{}
      child { node[comp,label=right:$\id$]{}
        child } ;
  \end{tikzpicture}\end{aligned}\ .$$

There exists a
 unique coderivation $q\calO^{\ac}\to \TTT^c(sE)$ which extends  the map
\[ q\calO^\ac \xtwoheadrightarrow{} s^2qR \xrightarrow{s^{-1}\varphi_1} sE \ .\]
Conditions (II) implies that its image lies in 
the subcooperad
$q\calO^\ac \subset \TTT^c(sE)$, see \cite[Lemma~$4.1.1$]{HirshMilles12}. This defines the coderivation $\dif_{\calO^\ac} : q\calO^\ac \to q\calO^\ac$.

\begin{proposition}\label{pro:Coder}
The coderivation $\dif_{\calO^\ac} : q\calO^\ac \to q\calO^\ac$ is equal to 
$$\dif_{\calO^{\ac}}\big( t(\bar\sigma_1, \ldots, \bar\sigma_k)\big)=
\sum_{j=1}^k \sum_{l=1}^{i_j-1}
\pm\   t\big(\bar \sigma_1, \ldots, \bar \sigma_{j-1},
(\sigma_j^1, \ldots, \underbrace{\sigma_j^{l}\sigma_j^{l+1}}_{\neq \id}, \ldots, \sigma_j^{i_j})
\bar \sigma_{j+1}, \ldots, \bar \sigma_{k}\big)\ , 
$$
where the sign is equal to 
$$(-1)^{|t|+|\bar \sigma_1|+\cdots+|\bar \sigma_{j-1}|+l}\ . $$
\end{proposition}

\begin{proof}
This a direct consequence of \cite[Proposition~$6.3.7$]{LodayVallette12}, i.e. one applies the map $\varphi_1$
everywhere one can. The sign is given by the Koszul sign rule with the left-recursive order on composite trees.
\end{proof}

\begin{example} Assume $\rho\sigma\neq \id, \sigma\tau  \neq \id$.
  \begin{align*}
&    \dif_{\calO^\ac}\left(
      \begin{aligned}\begin{tikzpicture}[optree]
        \node{}
          child { node[emptycirc]{}
            child { node[dot,label=left:$\tau$]{}
              child[level distance=5mm] { node[dot,label=left:$\sigma$]{}
                child { node[dot,label=left:$\rho$]{} child }
              }
            }
            child {
              child[level distance=5mm] { node[dot,label=right:$\chi$]{}
                child { child {} }
              }
            }
          } ;
      \end{tikzpicture}\end{aligned}
    \right)
     = -
    \begin{aligned}\begin{tikzpicture}[optree]
      \node{}
        child { node[emptycirc]{}
          child { node[dot,label=left:$\sigma\tau$]{}
            child[level distance=5mm] { node[dot,label=left:$\rho$]{} child }
          }
          child {
            child[level distance=5mm] { node[dot,label=right:$\chi$]{} child }
          }
        } ;
    \end{tikzpicture}\end{aligned}
    +
    \begin{aligned}\begin{tikzpicture}[optree]
      \node{}
        child { node[emptycirc]{}
          child { node[dot,label=left:$\tau$]{}
            child[level distance=5mm] { node[dot,label=left:$\rho\sigma$]{} child }
          }
          child {
            child[level distance=5mm] { node[dot,label=right:$\chi$]{} child }
          }
        } ;
    \end{tikzpicture}\end{aligned}    \ , \textrm{i.e.} \\
    &
        \dif_{\calO^\ac}\big(   
        t((\rho,\sigma,\tau), \chi)        \big) =
        t((\rho,\sigma\tau), \chi)
                -t((\rho\sigma,\tau), \chi)\ .
  \end{align*}
\end{example}

The curvature $\theta$ is the map defined by the following composite 
\[ q\calO^\ac \xtwoheadrightarrow{} s^2qR \xrightarrow{s^{-2}\otimes \varphi_0} I \ . \]

\begin{proposition}\label{pro:Curvature}
The curvature $\theta : q\calO^\ac \to I$ is equal to   zero, except for 
\begin{align*}
  \theta \left(\ 
    \begin{aligned}\begin{tikzpicture}[optree]
      \node{}
        child[level distance=5mm] { node[dot,label=right:$\sigma^{-1}$]{}
          child { node[dot,label=right:$\sigma$]{} child }
        } ;
    \end{tikzpicture}\end{aligned}
  \right)
  & =
  \begin{aligned}\begin{tikzpicture}[optree]
    \node{}
      child[level distance=15mm] ;
  \end{tikzpicture}\end{aligned}\  .
\end{align*}
\end{proposition}

\begin{proof}
Straightforward from the definition.
\end{proof}

\begin{theorem}\label{thm:OKoszul}
The colored operad $\calO$, with its presentation given in Definition~\ref{def:QLCPresentation}, is an inhomogeneous Koszul operad. The  cobar construction of the curved Koszul dual cooperad $\calO^{\ac}$ is a cofibrant resolution of the operad $\calO$: 
$$\Omega \calO^{\ac} \stackrel{\sim}{\longrightarrow} \calO\ . $$
\end{theorem}

\begin{proof}
Conditions~(I) and (II) of \cite{HirshMilles12} are satisfied by Lemma~\ref{lem:Conditions12}.
The homogeneous quadratic colored operad $q\calO$ is Koszul by the distributive law method, see Proposition~\ref{prop:DistLawqO}. Hence the colored operad $\calO$ is Koszul. 
The second statement is a consequence of the first one by \cite[Theorem~$4.3.1$]{HirshMilles12}.
Again, the proofs of \cite{HirshMilles12} are only given over a field of characteristic $0$. 
One can extend them over a commutative ring along the lines of \cite{Fresse04}, provided that the 
components of the colored operad $\calO$ are finitely generated projective $\KK$-modules and projective $\KK[\Sy_n]$-modules. This is the case here since the components of the colored operad $\calO$ are 
free finitely generated  $\KK$-modules and free $\KK[\Sy_n]$-modules.
\end{proof}

\section{Homotopy theory of symmetric operads}\label{sec:HoforOpinfty}

We now apply this new Koszul model  to the study of the homotopy properties of dg operads with the required properties for  the symmetric group action over any ring. 
This gives rise to a new category of cooperads, dual to the category of nu operads,  to a
new bar construction, and to a new cobar construction, which properly take  into account the symmetric group actions. 
These notions turn out to be enriched generalisations of  the classical notions. Again, this is produced automatically and conceptually by the Koszul duality theory applied to the colored operad $\calO$. 

In the end, the iteration of these new bar and cobar constructions provides us with a functorial cofibrant resolution for augmented dg operads over any ring. We relate it to the classical bar-cobar construction of the Barratt--Eccels operad. These results, applied to the operad encoding commutative algebras, give rise to a cofibrant $E_\infty$-operad and to an explicit notion of $E_\infty$-algebra by means of generators and relations. 

\subsection{Higher cooperads}\label{subsec:KDOcoalg}

\begin{definition}[Higher cooperad]
A \emph{higher cooperad} is a  coalgebra over the curved colored cooperad 
$\calO^\ac$. 
\end{definition}

Let us now make this definition explicit. 

\begin{proposition}
A higher cooperad is  a graded $\NN$-module 
$\{
\calC(n)
\}_{n\in \NN}$ equipped with  symmetric groups coactions and partial decomposition maps:
$$\left\{\begin{array}{clll}
   \delta_\sigma&\colon & \calC(n) \longrightarrow  \calC(n)\ ,
      & \text{for}\ \sigma \in \Sy_n\backslash \{\id_n\} \ , \\
    \Delta_i&\colon& \calC(n+k-1) 
      \longrightarrow \calC(n)\otimes \calC(k) \ ,
      & \text{for}\ 1 \leq i \leq n \ ,
  \end{array}
  \right.
  $$
both of degree $|\delta_\sigma|=|\Delta_i|=-1$.   
  These are required to satisfy 
the signed parallel and the signed sequential  decomposition axioms 
\begin{itemize}
\item[\emph{(1a)}] $(\Delta_j\otimes \id)\Delta_i=-(23)(\Delta_i \otimes \id)\Delta_{j+k-1}\ ,  \quad$ when $i<j$\ ,
\item[\emph{(1b)}] $(\id \otimes \Delta_j)\Delta_i
          =-(\Delta_i\otimes \id)\Delta_{j+i-1}\ ,$
\end{itemize} 
and the compatibility axiom 
\begin{itemize}
\item[\emph{(2)}]\begin{align*}
  &
  \Delta_i \delta_\tau
  =
  \begin{cases}
-(\id\otimes \delta_\sigma)\Delta_i    \ , & \textrm{if } \tau = \sigma' \\
-(\delta_\sigma \otimes \id) \Delta_{\sigma^{-1}(i)}   \  , & \textrm{if } \tau = \sigma'' \\
    \ \quad 0\ , & \textrm{otherwise}.
  \end{cases}
\end{align*}
\end{itemize} 
It is also endowed with 
a degree $-1$ map of $\NN$-modules $\dif_\calC : \calC \to \calC$ satisfying the coderivation property
\begin{align}
\begin{cases}\label{eqn:Relder}
& \Delta_i \, \dif_\calC= -(\dif_\calC\otimes \id + \id \otimes \dif_\calC)\, \Delta_i\ ,  \\
& \delta_\sigma \, \dif_\calC = - \dif_\calC \, \delta_\sigma +
\displaystyle\sum_{\tau, \omega \in \Sy_n\backslash \{\id\}\atop \tau\omega=\sigma} \delta_\tau \, \delta_\omega \ , 
\end{cases}
\end{align}
and the curved differential property
\begin{align}\label{eqn:Reld2}
{\dif_\calC}^2=-\sum_{\sigma \in \Sy_n\backslash \{\id_n\}} \delta_\sigma \, \delta_{\sigma^{-1}}\ ,  \quad \text{on} \ \ \calC(n)\ .
\end{align}
\end{proposition}

\begin{proof}
By definition, a coalgebra over the curved colored cooperad 
$\calO^\ac=(q\calO^{\ac}, \dif_{\calO^\ac},\theta)$ is a triple 
$\left( 
\calC, \Delta_\calC, \dif_\calC
\right)$ 
made up of a $q\calO^{\ac}$-coalgebra and a degree $-1$ coderivation satisfying 
$${\dif_\calC}^2=(-\theta \circ \id_\calC)\Delta_\calC \ .$$

The results of Section~\ref{subsec:KDofqO} shows that a $q\calO^{\ac}$-coalgebra is a graded $\NN$-module 
$\{
\calC(n)
\}_{n\in \NN}$
equipped with a structure map 
$$\Delta_\calC\ : \ \calC \to  q\calO^{\ac}(\C)=\prod_{n} (q\calO^{\ac}(n)\otimes \C^{\otimes n})^{\Sy_n} \ $$ 
compatible with the decomposition map $\Delta_{q\calO^{\ac}}$ of the cooperad $q\calO^{\ac}$. 
Therefore the data of this structure map $\Delta_\calC$ is completely determined by its image on the space of cogenerators $sE$ of $q\calO^{\ac}$. 
The part made up of  the unary generators $s\sigma$ in $sE(\calC)$ produces the maps $\delta_\sigma$ and 
the part made up of on the binary generators corresponding to the $\circ_i$-products in $sE(\calC)$ produces the maps $\Delta_i$. 
The corelations $s^2qR$ of $q\calO^{\ac}$ give the relations (1a), (1b) and (2).  (Notice that there is no relation among the coactions  of the symmetric groups $\Sy_n$.) 

The coderivation property of the map $\dif_\calC$ amounts to the commutativity of the following diagram
$$ 
\xymatrix@C=40pt{
\calC     \ar[r]^(0.45){\Delta_\calC}  \ar[d]^{\dif_\calC} &    q\calO^{\ac}(\C)  \ \ \ar[d]^{\dif_{\calO^{\ac}}\circ \id+\id \circ'\dif_\calC}\\
\calC    \ar[r]^(0.45){\Delta_\calC} &    q\calO^{\ac}(\C)\ .}
$$
Since the structure map $\Delta_\calC$  of the higher cooperad is compatible with the decomposition map $\Delta_{q\calO^{\ac}}$ of the cooperad $q\calO^{\ac}$, it is enough to check the commutativity of the above diagram 
on the part of the image which lands on the space of cogenerators $sE$ of $q\calO^{\ac}$. The part which gives the binary generator corresponding to the $\circ_i$-product in $sE(\calC)$ produces the first coderivation property of Relation~(\ref{eqn:Relder}). And 
the part which gives the unary generators $s\sigma$ in $sE(\calC)$ produces the second coderivation property of Relation~(\ref{eqn:Relder}). Notice that in the first case, the coderivation $\dif_{\calO^{\ac}}$ plays no role
and that, in the second case, it is responsible for the term 
$\sum_{\tau, \omega \in \Sy_n\backslash \{\id\}\atop \tau\omega=\sigma} \delta_\tau \, \delta_\omega $.

Finally, the curvature relation 
${\dif_\calC}^2=(-\theta \circ \id_\calC)\Delta_\calC$ gives here Relation~(\ref{eqn:Reld2}). 
\end{proof}

\begin{example} For any graded $\NN$-module $\P$, we consider the graded $\NN$-module
$$q\calO^{\ac}(\calP)(n)=\bigoplus_{  \text{\tiny Operadic trees}Ê\atop \text{\tiny with}\ \scriptstyle  { n}  \  \text{\tiny leaves}}
\begin{aligned}\begin{tikzpicture}[optree,
      level distance=12mm,
      level 2/.style={sibling distance=12mm},
      level 3/.style={sibling distance=12mm}]
    \node{}
      child { node[comp,label={-3:{\tiny$\sigma^1_1$}}]{}
      node[comp,label={[label distance=1mm]183:{\Small $\P(3)$}}]{}
        child 
        child { edge from parent[draw=none] }
        child { edge from parent[draw=none] }
        child { node[comp,label={[label distance=2mm]1:{\tiny $\sigma^1_2,\sigma^2_2$}}]{}
        node[comp,label={[label distance=0.5mm]183:{\Small $\P(4)$}}]{}
          child { node[comp,label={[label distance=1mm]3:{\tiny $\sigma_3^1$}}]{}
          node[comp,label={[label distance=1mm]left:{\Small $\P(2)$}}]{}
            child
            child }
        child { edge from parent[draw=none] }
          child
          child
          child }
        child { edge from parent[draw=none] }
        child { edge from parent[draw=none] }
        child } ;
  \end{tikzpicture}\end{aligned}\ .
 $$
If an operadic tree can be obtained by grafting a subtree with $k$ leaves on a subtree with $n$ leaves at the $i^\text{th}$ leaf, then its  image under the map  $\Delta_i : q\calO^{\ac}(\P)(n+k-1) 
\to q\calO^{\ac}(\P)(n) \otimes q\calO^{\ac}(\P)(k)$ is the pair of these two subtrees; otherwise, it is $0$. 
The image of an operadic tree under   $\delta_\tau$ is the sum of several contributions. For each vertex $j$, one considers the last permutation $\sigma=\sigma_j^{i_j}$ of the string and pulls it out of the operadic tree step by step changing it into $\sigma'$ or $\sigma''$. (This process is similar to Step 4 of Section~\ref{subsec:KDCoopQ}). If the final permutation is equal to $\tau$, the remaining operadic tree contributes to $\delta_\tau$. 
This defines a  $q\calO^{\ac}$-coalgebra, which is cofree among conilpotent $q\calO^{\ac}$-coalgebras. 
\end{example}

\begin{remark}
As usual, the underlying suspension $s\calC$ of a $q\calO^{\ac}$-coalgebra $\calC$ gives a notion with less signs, see \cite[\S $11.1.3$]{LodayVallette12} for instance. In the present case, one obtains a cooperad without counit equipped with compatible free coactions of the symmetric groups, i.e. Axioms (1a)-(1b)-(2) without sign. 
\end{remark}

By definition, a coalgebra over a cooperad is a coalgebra over the associated comonad, which is  
$$\C \to q\calO^{\ac}(\C)=\prod_{n} (q\calO^{\ac}(n)\otimes \C^{\otimes n})^{\Sy_n}$$
here. Since the cooperad $q\calO^{\ac}$ is weight graded, so is the induced comonad, that is 
$\C \to \prod_r {q\calO^{\ac}}^{(r)}(\C)$.
One can introduce the following \emph{coradical filtration} to detect  which part of a coalgebra is decomposed in each weight component. We consider 
$$F_1:=\lbrace
c\in\C \ | \ \Delta_i(c)=0\ \text{and}\ \delta_{\sigma}(c)=0, \ \text{for any}\ i \ \text{and any}\ \sigma
\rbrace \ . $$
More generally, the space $F_r$ is made up of elements $c\in C$ which vanish under the composite of  $r$ maps among $\Delta_i$ or $\delta_\sigma$. 
 
\begin{definition}[Conilpotent higher cooperad]
A \emph{conilpotent higher cooperad} is a higher cooperad whose coradical filtration is exhaustive. 
\end{definition}

Equivalently, this means that its structure map factors through the direct sum of the weight components, that is 
$\C \to \bigoplus_r {q\calO^{\ac}}^{(r)}(\C)$. From now, on the notation ${q\calO^{\ac}}(-)$ will always mean $\bigoplus_r {q\calO^{\ac}}^{(r)}(-)$

This subcategory of higher cooperads provides us with the correct target (resp. source) category for the forthcoming new bar (resp. cobar) construction. 

\subsection{Higher bar construction}\label{subset:HighBarConst}
Recall from the general theory \cite[Chapter 11]{LodayVallette12}  and \cite[\S $5.2$]{HirshMilles12} that any curved  twisting morphism 
$\kappa : \calO^{\ac}  \to \calO$ gives rise to a pair of adjoint functors 
$$\xymatrix@C=30pt{{\B_\kappa \ : \ \calO\text{-}\mathsf{alg}  \ }  \ar@_{->}@<-1ex>[r]  \ar@^{}[r]|(0.41){\perp}   & \ar@_{->}@<-1ex>[l]  {\ \mathsf{conil}\ \calO^{\ac}\text{-}\mathsf{coalg}  \ : \ \Omega_\kappa}}\ .   $$
We consider here the canonical Koszul morphism $\kappa : \calO^{\ac}  \to \calO$, which desuspends the generators 
$$\begin{cases}
\begin{array}{clc}
\begin{aligned}\begin{tikzpicture}[optree,
      level distance=12mm,
      level 2/.style={sibling distance=12mm},
      level 3/.style={sibling distance=6mm}]
    \node{}
      child { node[comp,label={[label distance=1mm]1:{\tiny$\sigma$}}]{}
        child { node{\tiny $1$} }
        child { edge from parent[draw=none] node{$\,\cdots$} }
        child { node{\tiny $n$} }
      } ;
  \end{tikzpicture}\end{aligned}
   & \mapsto & 
\begin{aligned}\begin{tikzpicture}[optree]
    \node{}
      child { node[dot,label=right:$\sigma$]{}
        child { edge from parent node[right,near end]{\tiny$n$} }
        edge from parent node[right,near start]{\tiny$n$} } ;
   \end{tikzpicture}\end{aligned}  \ , \\
  \begin{aligned}\begin{tikzpicture}[optree,
      level distance=12mm,
      level 2/.style={sibling distance=12mm},
      level 3/.style={sibling distance=12mm}]
    \node{}
      child { 
        child { node{\tiny $1$} }
        child { edge from parent[draw=none] node{$\,\cdots$} }
        child { node[comp,label=right:$\,i$]{} 
          child { node{\tiny $1$} }
          child { edge from parent[draw=none] node{$\,\cdots$} }
          child { node{\tiny $k$} } }
        child { edge from parent[draw=none] node{$\,\cdots$} }
        child { node{\tiny $n$} } } ;
  \end{tikzpicture}\end{aligned} 
& \mapsto & 
 \begin{aligned}\begin{tikzpicture}[optree]
    \node{}
      child { node[circ]{$i$}
        child {  edge from parent node[left,near end]{\tiny$n$} }
        child {  edge from parent node[right,near end]{\tiny$k$} }
        edge from parent node[right,near start]{\tiny$n+k-1$} } ;
   \end{tikzpicture}\end{aligned} \ , 
\end{array}
\end{cases} $$
and which is equal to zero otherwise. This gives rise to the following higher bar-cobar 
pair of adjoint functors: 
$$\xymatrix@C=30pt{{\widetilde\B \ : \ \mathsf{nu} \ \mathsf{Op}  \ }  \ar@_{->}@<-1ex>[r]  \ar@^{}[r]|(0.39){\perp}   & \ar@_{->}@<-1ex>[l]  {\ \mathsf{conil}\ \mathsf{higher}\ \mathsf{Coop}  \ : \ \widetilde{\Omega}}}\ .   $$

\begin{proposition}
The higher bar construction 
$$\widetilde{\B} \P:=\B_\kappa \P=\left(  
q\KDO(\P), \dif_1+\dif_2
\right)$$ of a non-unital operad $\P$ is the 
$\calO^{\ac}$-coalgebra defined on the underlying cofree $q\calO^{\ac}$-coalgebra $q\calO^{\ac}(\P)$ by the curved codifferential $\dif_1+\dif_2$, where $\dif_1$ is the coderivation
$$\dif_1 \ : \ q\KDO(\P) \xrightarrow{\dif_{\KDO} \circ \id_\P + \id_{q\KDO} \circ' \dif_\P} q\KDO(\P) $$
and where $d_2$ is the unique coderivation extending the action of $\kappa$
$$\dif_2 \ : \ q\KDO(\P)  \xrightarrow{\Delta_{(1)}\circ \id_\P} (q\KDO \circ_{(1)} q\KDO)(\P) 
\xrightarrow{\left(\id_{q\KDO} \otimes \kappa\right)\circ \id_\P}
(q\KDO \circ_{(1)} \calO)(\P)  \mono q\KDO \circ \calO(\P) 
\xrightarrow{\id_{q\KDO} \circ \gamma_\P}
q\KDO(\P)
 \ .$$
\end{proposition}

\begin{remark}\leavevmode
Since the category of non-unital operads is equivalent to the category of augmented operads, one can define the higher bar construction of an augmented operad by $\widetilde{\B}\,  \overline\P:=
\left(q\calO^{\ac}(\overline\P), \dif_1+\dif_2
\right)$. 
\end{remark}

The higher bar construction $\widetilde\B  \P$ of a dg nu operad $\P$ is given by the sum of operadic trees with vertices labeled by the elements of $\P$ and by strings of permutations, that we denote by 
$$t(\bar\sigma_1, \ldots, \bar\sigma_k; \mu_1, \ldots, \mu_k):=
\begin{aligned}\begin{tikzpicture}[optree,
      level distance=12mm,
      level 2/.style={sibling distance=12mm},
      level 3/.style={sibling distance=12mm}]
    \node{}
      child { node[comp,label={-3:{\tiny$\sigma^1_1$}}]{}
      node[comp,label={[label distance=1mm]183:$\mu_1$}]{}
        child 
        child { edge from parent[draw=none] }
        child { node[comp,label={[label distance=2mm]1:{\tiny $\sigma^1_2,\sigma^2_2$}}]{}
        node[comp,label={[label distance=1mm]180:$\mu_2$}]{}
          child { node[comp,label={[label distance=1mm]3:{\tiny $\sigma_3^1$}}]{}
          node[comp,label={[label distance=1mm]left:$\mu_3$}]{}
            child
            child }
        child { edge from parent[draw=none] }
          child
          child
          child }
        child { edge from parent[draw=none] }
        child { edge from parent[draw=none] }
        child } ;
  \end{tikzpicture}\end{aligned}
 \ .$$
 The coderivation $\dif_1$ is equal to the sum 
\begin{align*}
&\sum_{j=1}^k \sum_{l=1}^{i_j-1}
(-1)^{|t|+|\bar \sigma_1|+\cdots+|\bar \sigma_{j-1}|+l}\   t\big(\bar \sigma_1, \ldots, \bar \sigma_{j-1},
(\sigma_j^1, \ldots, \underbrace{\sigma_j^{l}\sigma_j^{l+1}}_{\neq \id}, \ldots, \sigma_j^{i_j})
\bar \sigma_{j+1}, \ldots, \bar \sigma_{k}; \mu_1, \ldots, \mu_k\big)
\\
+& \sum_{j=1}^k (-1)^{|t(\sigma)|+|\mu_1|+\cdots+|\mu_{j-1}|}\ 
t(\bar\sigma_1, \ldots, \bar\sigma_k; \mu_1, \ldots, d_\P(\mu_j), \ldots , \mu_k)\ .
\end{align*}
The coderivation $\dif_2$ is equal to 
\begin{align*}
&\sum_{j=1}^k 
(-1)^{|t|+|\bar \sigma_1|+\cdots+|\bar \sigma_{j-1}|}\   
t\big(\bar \sigma_1, \ldots, 
(\sigma_j^2, \ldots,  \sigma_j^{i_j}), 
\ldots, \bar \sigma_{k}; 
\mu_1, \ldots, 
\mu_k^{\sigma_j^1}
, \ldots, \mu_k\big)
\\
+& \sum_{\text{internal edges}\atop j - e - l}
\pm \ 
t/e(\bar\sigma_1, \ldots, 
\shuffle(\bar \sigma_j, \bar \sigma_l), \ldots, \bar \sigma_{l-1}, \bar \sigma_{l+1},\ldots
\bar\sigma_k; \mu_1, \ldots, \mu_j\circ_m\mu_l, \ldots ,\mu_{l-1}, \mu_{l+1}, \ldots, \mu_k)\ , 
\end{align*}
where the internal edge $e$ joins the vertex $j$ to the vertex $l$ at the $i^{\text{th}}$-place, where $m=\sigma_j^1 \cdots \sigma_j^{i_j} (i)$ and where the sign $\pm$ is given by the same procedure as in 
Section~\ref{subsec:KDCoopQ}.

\subsection{Higher cobar construction}\label{subset:HighCoBarConst}

\begin{proposition}
The higher cobar construction 
$$\widetilde{\Omega} \C:=\Omega_\kappa \C=\left(  
\calO(\C), \dif_1-\dif_2
\right)$$ of a higher cooperad  $\left( 
\calC, \Delta_\calC, \dif_\calC
\right)$ 
 is the 
dg non-unital operad defined on   $\calO(\C)$ by the differential $\dif_1-\dif_2$, where $\dif_1$ is the underlying derivation
$$\dif_1 \ : \ \calO(\C) \xrightarrow{ \id_{\calO} \circ' \dif_\C} \calO(\C) $$
and where $d_2$ is the unique derivation extending the action of $\kappa$
$$\dif_2 \ : \ \calO(\C)  
\xrightarrow{\id_\calO\circ'\Delta_\C} 
(\calO \circ_{(1)} q\KDO)(\C) 
\xrightarrow{\left(\id_{\calO} \otimes \kappa\right)\circ \id_\C}
(\calO \circ_{(1)} \calO)(\C)  
\mono 
(\calO \circ \calO)(\C) 
\xrightarrow{\gamma_\calO \circ \id_\C}
\calO(\C)
 \ .$$
\end{proposition}

The cobar construction $\widetilde{\Omega}  \C $ of a higher cooperad $\calC$ is given by the sum of operadic trees with the vertices labeled by the elements of $\C$ and with the leaves globally labeled 
by a permutation, that we denote by 
$$t(\nu_1, \ldots, \nu_k; \sigma):=
  \begin{aligned}\begin{tikzpicture}[optree,
      level distance=12mm,
      level 2/.style={sibling distance=16mm},
      level 3/.style={sibling distance=8mm}]
    \node{}
      child { node[comp,label=183:$\nu_1$]{}
        child { node[comp,label=183:$\nu_2$]{}
          child { node[label=above:$1$]{} }
          child { node[label=above:$4$]{} }
          child { node[label=above:$8$]{} } }
          child { node[label=above:$2$]{} }
        child { node[comp,label=-3:$\nu_3$]{}
          child { node[label=above:$3$]{} }
          child { node[label=above:$5$]{} }
          child { node[comp,label=-3:$\nu_4$]{}
            child { node[label=above:$6$]{} }
            child { node[label=above:$7$]{} } } } };
  \end{tikzpicture}\end{aligned} \ .$$  
  The derivation $\dif_1$ is equal to 
  $$\sum_{j=1}^k  (-1)^{|\nu_1|+\cdots+ |\nu_{j-1}|}\ t(\nu_1, \ldots, \dif_\calC (\nu_j), \ldots, \nu_k; \sigma)      \ .$$
The derivation $\dif_2$ is equal to 
\begin{align*}
&\sum_{j=1}^k  \sum_{i=1}^{n_j}(-1)^{|\nu_1|+\cdots+ |\nu_{j-1}|}\ t\sqcup e^i_j(\nu_1, \ldots, \Delta'_i (\nu_j), \Delta''_i (\nu_j),  \ldots, \nu_k; \sigma) \\
+&\sum_{j=1}^k \, \sum_{\tau \in \Sy_{n_j}\backslash\{\id_{n_j}\}} (-1)^{|\nu_1|+\cdots+ |\nu_{j-1}|}\ 
\tilde t(\nu_1, \ldots,  \delta_\tau(\nu_j),  \ldots, \nu_k; \tilde \tau \sigma) \ ,
\end{align*}
where $n_j$ is the arity of vertex $j$, where $t\sqcup e^i_j$ represents the tree $t$ with an extra edge at vertex $j$ and where $\Delta_i(\nu_j)=\Delta'_i(\nu_j)\otimes \Delta''_i(\nu_j)$.

\subsection{Higher bar-cobar adjunction}

Let $\calC$ be a higher cooperad and let $\P$ be a dg nu operad. We consider the mapping space 
$$\Hom(\calC, \P):=\prod_{n\in \NN} \Hom\big(\calC(n), \P(n)\big)\ ,  $$
endowed with the following two degree $-1$ operations:
\begin{align*}
f\ast g := \sum_i  \left( \calC \xrightarrow{\Delta_i} \calC \otimes \calC \xrightarrow{f \otimes g} \P \otimes \P 
\xrightarrow{\circ_i} \P \right)\ ,
\end{align*}
and 
\begin{align*}
\Delta (f) := \sum_{\sigma \, \in\, \Sy_n \backslash \{Ê\id_n \}}  \left( 
\calC(n) \xrightarrow{\delta_\sigma} \calC(n) \xrightarrow{f(n)} \P(n)  
\xrightarrow{(-)^\sigma} \P(n) 
\right)\ .
\end{align*}

\begin{proposition}\label{prop:ConvRel}
These operations satisfy the following relations: 
\begin{eqnarray}\label{eqn:RelPreLie}
(-1)^{|\gamma|} (\alpha \ast \beta) \ast \gamma +\alpha \ast (\beta \ast \gamma)
+(-1)^{|\beta||\gamma|}
\big(
(-1)^{|\beta|} (\alpha \ast \gamma) \ast \beta +\alpha \ast (\gamma \ast \beta)
\big)=0
\ ,
\end{eqnarray}
\begin{eqnarray}\label{eqn:RelDer}
\Delta(\alpha\ast \beta)+(-1)^{|\beta|}(\Delta\alpha)\ast\beta + \alpha \ast (\Delta\beta)=0\ .
\end{eqnarray}

\end{proposition}

\begin{proof}
These relations are checked in a straightforward way using the axioms defining the notions of an nu operad and a higher cooperad. 
The ``shifted'' pre-Lie relation~(\ref{eqn:RelPreLie}) is proved using Axioms (iii) and (iv) in the definition of an nu operad and Axioms (1a) and (1b) in the definition of a higher cooperad. 
The ``shifted'' derivation relation~(\ref{eqn:RelDer}) is proved using Axioms (v) and (vi) in the definition of an nu operad and Axiom (2)  in the definition of a higher cooperad. 

\end{proof}

\begin{remark}
In the classical case, the convolution algebra made up of maps from a cooperad to an operad forms a pre-Lie algebra, see \cite[\S 6.4.2]{LodayVallette12}. 
Because of the slightly different sign convention chosen here, the $\ast$-product satisfies the relation of a ``shifted'' pre-Lie algebra. Since $\calC$ is a higher cooperad, this enrichment gives rise to the operator $\Delta$ which a ``shifted'' derivation with respect to the $\ast$-product.  
\end{remark}

We call any morphism of $\NN$-modules $\alpha : \calC \to \P$ (that is of degree $0$) satisfying the Maurer--Cartan equation 
\begin{align}\label{align:MCeqKappa}
\partial \alpha + \alpha \ast\alpha +\Delta \alpha =0\ 
\end{align}
a \emph{twisting morphism with respect to $\kappa$}, or simply \emph{twisting morphism} since the context is obvious.
The associated set of solutions is denoted by $\Tw(\calC, \P)$.

\begin{proposition}\label{prop:BarCobarTwKappa}
The higher bar and cobar constructions $\widetilde{\B}$ and $\widetilde{\Omega}$ form a pair of adjoint functors 
$$\xymatrix@C=30pt{{\widetilde\B \ : \ \mathsf{nu} \ \mathsf{Op}  \ }  \ar@_{->}@<-1ex>[r]  \ar@^{}[r]|(0.41){\perp}   & \ar@_{->}@<-1ex>[l]  {\ \mathsf{conil}\ \mathsf{higher}\ \mathsf{Coop}  \ : \ \widetilde{\Omega}}}   $$ represented by the set of twisting morphisms:
$$\Hom_{\mathsf{dg\ nu\ Op}}(\widetilde{\Omega} \calC, \P) \cong \Tw(\calC, \P) \cong 
\Hom_{\mathsf{conil}\ \mathsf{higher}\ \mathsf{Coop}}(\calC, \widetilde{\B} \P)\ . $$
\end{proposition}

\begin{proof}
The arguments and computations are similar to the ones of \cite[\S $11.1$]{LodayVallette12}.
\end{proof}

\subsection{Higher bar-cobar resolution} In this section, we show that the composition of the higher bar construction with the higher cobar construction provides us with a resolution for augmented operads with nice homotopy properties.

\begin{proposition}\label{prop:CobarCofibrant}
Let  $\calC$ be a non-negatively graded conilpotent higher cooperad $\calC$ made up of projective $\KK$-modules.
The augmentation $\I\oplus \widetilde{\Omega}\, \calC$ of its cobar construction is a cofibrant dg operad. 
\end{proposition}

\begin{proof}
Let us first mention that we are working with the semi-model category structure defined on dg operads in \cite{Spitzweck01}, where weak equivalences are arity-wise quasi-isomorphisms and where fibrations are arity-wise epimorphisms. 
Let $f$ and $b$ be two morphisms of dg operads
$$\xymatrix@R=30pt@C=30pt{  
& \P \ar@{->>}[d]_(0.46)\sim^(0.46)f\\
\I\oplus \widetilde{\Omega}\, \calC \ar[r]^(0.6)b 
\ar@{..>}[ur]^a&\ \calQ  \ ,} $$
where $f$ is an acyclic fibration. And let us look for a morphism $a$ of dg operads which factors $b$ through $f$. 
The augmentation $\I\oplus  \widetilde{\Omega} \, \calC$ of a dg nu operad gives an augmented operad. 
And the data of a morphism of dg operads $\I\oplus \widetilde{\Omega} \, \calC \allowbreak \to \calQ$ is equivalent to the data of a morphism $\widetilde{\Omega} \, \calC \to \calQ$ of dg nu operads. 
So, according to Proposition~\ref{prop:BarCobarTwKappa}, the data of a morphism of dg operads  $b : \I\oplus \widetilde{\Omega} \, \calC \allowbreak \to \calQ$ is equivalent to the data of a twisting morphism with respect to $\kappa$, that is a degree $0$ 
map $\beta : \calC \to \calQ$ satisfying the Maurer--Cartan equation~(\ref{align:MCeqKappa}).

We first set $a(\id):=\id_\P$ and the rest of the morphism $a$ is completely determined by its restriction $\alpha:=a_{|\calC}$ on $\calC$, which has to satisfy the Maurer--Cartan equation~(\ref{align:MCeqKappa}). We construct $\alpha$ by induction on the homological degree; we use the notation $\alpha_d$ for the restriction of $\alpha$ on $\calC_d$. 
The Maurer--Cartan equation~(\ref{align:MCeqKappa}) applied to an element $c\in\calC_d$ of degree $d$ reads 
$$\dif_\P (\alpha (c)) =
\alpha(\dif_\calC(c)) - (\alpha\ast \alpha) (c) - \Delta(\alpha)(c)
\ $$
where $|\dif_\calC(c)|=d-1$, $|\Delta_i(c)|=d-1$ and $|\delta_\sigma(c)|=d-1$. 
Therefore the right-hand term amounts to applying the map $\alpha$ to elements of degree at most $d-1$. 
Since $\calC$ is non-negatively graded, we define $\alpha_d$ to be trivial on $\calC_d$, for $d<0$. 
We suppose now that $\alpha$ is known up to degree $d-1$ and that is satisfies 
$$\dif_\P \alpha_e =
\alpha_{e-1}\dif_\calC - \sum_{k+l=e-1\atop k,l \leq e-1}\alpha_k\ast \alpha_l - \Delta(\alpha_{e-1})
\,  $$
for $e\leq d-1$.
Let us now construct $\alpha_d$; 
we consider the following diagram 
$$\xymatrix@R=30pt@C=90pt{  
 & \P_d \ar@{->>}[d]^{(f_d, \dif_\P)} & \\
\calC_d \ar[r]_(0.38){(\beta_d, 
\alpha\dif_\calC -\alpha\ast \alpha - \Delta(\alpha)
)} 
\ar@{..>}[ur]^{\alpha_d}& 
 \calQ_d \times_{Z\calQ_{d-1}}Z\P_{d-1} 
\ar[r]\ar[d]\ar@{}[rd] | (0.3)\pullback
 &  Z\P_{d-1}\ar[d]^{f_{d-1}} \ \\
&Q_d\ar[r]^{\dif_\calQ} &Z\calQ_{d-1}\ ,} $$
where $Z\P_{d-1}$ and $Z\calQ_{d-1}$ stand for the set of cycles in $\P_{d-1}$ and $\calQ_{d-1}$. 
 As explained above, the map $\alpha\dif_\calC -\alpha\ast \alpha - \Delta(\alpha)$ has  already been defined on $\calC_d$. 
We still have to prove that the  image of 
$(\beta_d, 
\alpha\dif_\calC -\alpha\ast \alpha - \Delta(\alpha)
)$
 lands in 
 $ \calQ_d \times_{Z\calQ_{d-1}}Z\P_{d-1}$.
 We first compute the composite with $\dif_\P$ of the right-hand term. 
By the induction hypothesis, since $\dif_\P$ is a derivation of the operad $\P$ and by the definition of $\calC$ being a higher cooperad, we have 
\begin{eqnarray*}
&&\dif_\P\big(
\alpha_{d-1}\dif_\calC -\sum_{k+l=d-1\atop k, l \leq d-1}\alpha_{k}\ast \alpha_{l} - \Delta(\alpha_{d-1})
\big)\\
&=&
\dif_\P\alpha_{d-1}\dif_\calC
-
\sum_i\sum_{k+l=d-1\atop k, l \leq d-1}
\dif_\P \circ_i (\alpha_k\otimes \alpha_l)\Delta_i
-
\sum_{\sigma \, \in\, \bar\Sy_n } \dif_\P(-)^\sigma\alpha_{d-1}\delta_\sigma\\
&=&
\dif_\P\alpha_{d-1}\dif_\calC
-
\sum_i\sum_{k+l=d-1\atop k, l \leq d-1}
 \circ_i (\dif_\P\alpha_k\otimes \alpha_l)\Delta_i
-
\sum_i\sum_{k+l=d-1\atop k, l \leq d-1}
 \circ_i (\alpha_k\otimes \dif_\P\alpha_l)\Delta_i\\
&&-
\sum_{\sigma \, \in\, \bar\Sy_n} (-)^\sigma\dif_\P\alpha_{d-1}\delta_\sigma\\
&=&
\dif_\P\alpha_{d-1}\dif_\calC
-
\sum_i\sum_{k+l=d-1\atop k, l \leq d-1}
 \circ_i \big(
(\alpha_{k-1}\dif_\calC-
\sum_{j+m=k-1\atop j,m \leq k-1}
\alpha_{j} \ast \alpha_{m}
-\Delta(\alpha_{k-1}))
 \otimes \alpha_l\big)\Delta_i\\
 &&
-
\sum_i\sum_{k+l=d-1\atop k, l \leq d-1}
 \circ_i \big(\alpha_k\otimes 
 (\alpha_{l-1}\dif_\calC-
 \sum_{j+m=l-1\atop j,m \leq l-1}
\alpha_{j} \ast \alpha_{m}
 -\Delta(\alpha_{l-1}))
 \big)\Delta_i\\
&&-
\sum_{\sigma \, \in\, \bar\Sy_n } (-)^\sigma
(\alpha_{d-2}\dif_\calC-
\sum_{k+l=d-2\atop k,l \leq d-2}
\alpha_{k} \ast \alpha_{l}
-\Delta(\alpha_{d-2}))
\delta_\sigma\\
&=&
\dif_\P\alpha_{d-1}\dif_\calC
+
\sum_{k+l=d-2\atop k, l \leq d-1} (\alpha_{k}\ast \alpha_{l}) \dif_\calC
+ 
\Delta(\alpha_{d-2})\dif_\calC
+ \sum_{j+k+l=d-2\atop j, k, l \leq d-2} 
\big(
(\alpha_j \ast \alpha_k) \ast \alpha_l + 
\alpha_j \ast (\alpha_k \ast \alpha_l)
\big)\\
&&+ 
\sum_i\sum_{k+l=d-2\atop k, l \leq d-2} 
\big(\Delta(\alpha_k) \ast \alpha_l +\alpha_k \ast \Delta(\alpha_l) + \Delta(\alpha_k\ast\alpha_l)
\big)+
\Delta^2(\alpha_{d-2})\\
&&-
\sum_{\sigma \, \in\, \bar\Sy_n} 
\sum_{\tau, \omega \, \in\, \bar\Sy_n  \atop \tau\omega=\sigma }(-)^\omega(-)^\tau\alpha_{d-2}\delta_\tau\delta_\omega
\ .
\end{eqnarray*} 
Since $\alpha$ has degree $0$, the relation~(\ref{eqn:RelPreLie}) of Proposition~\ref{prop:ConvRel} implies 
$$(\alpha\ast \alpha)\ast \alpha+\alpha\ast (\alpha\ast \alpha)=0 $$
and the relation~(\ref{eqn:RelDer}) of Proposition~\ref{prop:ConvRel} implies
$$\Delta(\alpha\ast \alpha)+(\Delta\alpha)\ast\alpha+\alpha\ast(\Delta \alpha)=0$$
The only remaining terms are the first three ones and the last two  ones, whose sum is equal to 
\begin{eqnarray*}
&&\left(\dif_\P\alpha_{d-1}+
\sum_{k+l=d-2\atop k, l \leq d-1} (\alpha_{k}\ast \alpha_{l})+
\Delta(\alpha_{d-2})
\right)\dif_\calC+\Delta^2(\alpha_{d-2})
-
\sum_{\sigma,\tau, \omega \, \in\, \bar\Sy_n\atop \tau\omega=\sigma } 
(-)^\omega(-)^\tau\alpha_{d-2}\delta_\tau\delta_\omega
\\
&=&\alpha_{d-2}\dif_\calC\dif_\calC+\Delta^2(\alpha_{d-2})
-
\sum_{\sigma,\tau, \omega \, \in\, \bar\Sy_n \atop \tau\omega=\sigma } 
(-)^\omega(-)^\tau\alpha_{d-2}\delta_\tau\delta_\omega\\
&=&
-\sum_{\tau, \omega \, \in\, \bar\Sy_n\atop \tau\omega=\id_n } 
\alpha_{d-2}\delta_\tau\delta_\omega
+
\sum_{\tau, \omega \, \in\, \bar\Sy_n} 
(-)^\omega(-)^\tau\alpha_{d-2}\delta_\tau\delta_\omega
-
\sum_{\sigma,\tau, \omega \, \in\, \bar\Sy_n \atop \tau\omega=\sigma } 
(-)^\omega(-)^\tau\alpha_{d-2}\delta_\tau\delta_\omega\\
&=&0\ ,
\end{eqnarray*}
 by the induction hypothesis and the defining relation~(\ref{eqn:Reld2}) of a higher cooperad. 
 Thus we conclude that 
 $$\dif_\P\left(
\alpha_{d-1}\dif_\calC -\sum_{k+l=d-1\atop k, l \leq d-1}\alpha_{k}\ast \alpha_{l} - \Delta(\alpha_{d-1})
\right)=0\ .$$

  \vskip1.5cm
 We now prove  that 
 $$\dif_\calQ \beta_d= f_{d-1}\big(
\alpha_{d-1}\dif_\calC -\sum_{k+l=d-1\atop k, l \leq d-1}\alpha_{k}\ast \alpha_{l} - \Delta(\alpha_{d-1})
\big)\ . $$
Since the map $\beta$ satisfies the Maurer--Cartan equation~(\ref{align:MCeqKappa}), the left-hand term is equal to 
$$\dif_\calQ \beta_d= 
\beta_{d-1}\dif_\calC -\sum_{k+l=d-1\atop k, l \leq d-1}\beta_{k}\ast \beta_{l} - \Delta(\beta_{d-1})
\ . $$
Since $f$ is a morphism of operads and by the induction hypothesis, the right-hand term is equal to 
\begin{eqnarray*}
&&f_{d-1}\left(
\alpha_{d-1}\dif_\calC -\sum_{k+l=d-1\atop k, l \leq d-1}\alpha_{k}\ast \alpha_{l} - \Delta(\alpha_{d-1})
\right)\\
&=& \beta_{d-1} \dif_\calC
-\sum_{k+l=d-1\atop k, l \leq d-1}(f_k\alpha_{k})\ast (f_l\alpha_{l}) - \Delta(f_{d-1}\alpha_{d-1})\\
&=&\beta_{d-1}\dif_\calC -\sum_{k+l=d-1\atop k, l \leq d-1}\beta_{k}\ast \beta_{l} - \Delta(\beta_{d-1})\ .
\end{eqnarray*} 
 which proves the above claim. 
 
To conclude, we use the following lemma which asserts that a morphism $f : \P \to \calQ$ is a surjective quasi-isomorphism if and only if every map $(f_d, \dif_\P) : \P_d  \epi  \calQ_d \times_{Z\calQ_{d-1}}Z\P_{d-1} $ is surjective. 
This is proved by the following diagram chasing. Suppose that $f$ is a surjective quasi-isomorphism and let 
$(q, p)\in \calQ_d \times_{Z\calQ_{d-1}}Z\P_{d-1}$, that is $\dif_\calQ(q)=f_{d-1}(p)$. Since $f_d$ is surjective, there exists $x\in \P_d$ such that $f_d(x)=q$. The element $p-\dif_\P(x)\in Z\P_{d-1}$ is a cycle, so it represents a homology class in $\P$, which is sent to the trivial homology class in $\calQ$ under $f$ since 
$$f_{d-1}(p-\dif_\P(x))=f_{d-1}(p)-\dif_\calQ(f_d(x))=f_{d-1}(p)-\dif_\calQ(q)=0\ .$$ Since $f$ is a quasi-isomorphism, there exists an element $y\in\P_d$ such that $\dif_\P(y)=p-\dif_\P(x)$. The element $f_d(y)\in Z\calQ_d$ is a cycle of $\calQ$ since $\dif_\calQ(f_d(y))=f_{d-1}(\dif_\P(y))=0$. The map $f$ being a quasi-isomorphism, there exists an element $z\in \P_d$ and an element $\xi\in \calQ_{d+1}$ such that $f(z)=f(y)+\dif_\calQ(\xi)$. Since the map $f_{d+1}$ is surjective, there exists $t\in \P_{d+1}$ such that $f_{d+1}(t)=\xi$. In the end, we have 
$$(f_d, \dif_\P)(x+y-z-\dif_\P(t))=(q, p)\ . $$

Finally, the $\KK$-modules $\calC_d(n)$ are projective for any $n\in \NN$, there exists a map of $\NN$-modules $\alpha_d : \allowbreak \calC_d \to \P_d$, which fills the above diagram. 
By construction, this map satisfies both $\beta_d=f_d \alpha_d$ and 
$\dif_\P\alpha_d = \alpha\dif_\calC -\alpha\ast \alpha - \Delta(\alpha)$, which concludes the proof. 
\end{proof}

\begin{remark}
The high cobar construction $\widetilde{\Omega}\,  \calC$ of a higher cooperad is a quasi-free $\calO$-algebra. So, forgetting the differential, it is isomorphic to  the free operad on the free $\Sy$-module generated by the $\NN$-module $\calC$. However since this latter one is not a free $\KK$-module, but only a projective one, we cannot apply (easily) the usual results, like \cite[Proposition~$12.2.3$]{Fresse09}, which give the form of cofibrant dg operads: retracts of quasi-free dg operads generated by free $\Sy$-modules endowed with a filtration lowered by the differential. 
Instead, we prove directly the lifting property with respect to acyclic fibrations using the Maurer--Cartan equation. 
\end{remark}

\begin{theorem}\label{thm:BarCobarRes}
The augmentation of the higher bar-cobar counit 
$$\I \oplus \widetilde{\Omega} \widetilde{\B} \, \oP\xrightarrow{\sim} \P$$ is a functorial resolution of dg augmented operads, which is cofibrant when the underlying chain complex of $\P$ is non-negatively graded and made up of 
 projective $\KK$-modules.
\end{theorem}

\begin{proof} 
As a graded $\N$-module, the higher bar-cobar construction is isomorphic to 
  $\widetilde{\Omega} \widetilde{\B} \, \oP \cong \opd{O} \circ q\opd{O}^\ac \circ \oP$. So its underyling space is made up   partitioned  planar rooted trees with vertices labeled by the elements of $\oP$ and by strings of (nontrivial) permutations and with the leaves globally labeled by a permutation.

 The differential $\dif$ on $\widetilde{\Omega} \widetilde{\B} \, \oP$ is a sum of several parts.  
Summing up all the components of the differential maps of the higher bar construction and the higher cobar construction from Sections~\ref{subset:HighBarConst} and  \ref{subset:HighCoBarConst}, we obtain $6$ different components for the differential $\dif$. 
First, there are four terms 
$\dif_1=\dif_1^\P+\dif_1^{sh} + \dif_1^{l} + \dif_1^{int}$, 
 coming from the higher bar construction, which apply only inside each partition. 
 The term $\dif_1^\P$ comes from the differential of the operad $\P$.
The term $\dif_1^{sh}$ contracts inner edges, composes the two elements of the operad $\P$ labeling the associated two vertices, and shuffles the two strings of permutations labeling the two associated vertices. 
The term $\dif_1^l$ removes the first element of each string of permutations and make it act on the element of of $\P$ labeling the vertex. 
The term $\dif_1^{int}$ is equal the sum of the products of two consecutive elements of a string labeling a vertex, provided that it is not the identity, in which case it gives  $0$. 
Then, there are two  terms 
$\dif_2=\dif_2^{split} + \dif_2^r$, coming from the higher cobar construction, which apply to each partition. 
The term $\dif_2^{split}$ splits each of the partition into two along an inner edge. 
The term $\dif_2^r$ extracts the right most external term of each string labeling the vertex of the tree inside the partition, pulls it out of the partition and make it act on the permutation labeling the leaves of the global tree. 
 The total differential is then the sum
  \[ \dif = \dif_1 - \dif_2 = (\dif_1^\P+\dif_1^{sh} + \dif_1^{l} + \dif_1^{int}) - (\dif_2^{split} + \dif_2^r) . \]

We consider the following bounded below and exhaustive filtration on the higher bar-cobar construction:
the component $\mathcal{F}_r \, \widetilde{\Omega} \widetilde{\B} \, \oP$ is made up of elements such that 
the sum of their degree  and the number of their partitions is at most equal to $r+1$. 
All the above terms of differential send $\mathcal{F}_r$ to $\mathcal{F}_{r-1}$, except for 
$\dif_2^{split} : \mathcal{F}_r \to \mathcal{F}_{r}$.
On the right-hand side, we consider the filtration on $\P$ defined by the homological degree. 
The morphism of dg nu operads $ \widetilde{\Omega} \widetilde{\B} \, \oP\to \oP$ is equal to the composite 
$$ \widetilde{\Omega} \widetilde{\B} \, \oP\cong \calO\circ q\calO^{\ac}\circ \oP \epi 
\calO \circ \oP
\to
 \oP\ , $$
where the first map comes from the augmentation projection $q\calO^{\ac}\epi \I$ and where the second map is given by the composition map in the operad $\P$. So this morphism preserves the respective filtrations, therefore it induces a morphism between the two associated spectral sequences. 

The underlying space of the first page of the spectral sequence associated to $\mathcal{F}_r \, \widetilde{\Omega} \widetilde{\B} \, \oP$ is isomorphic to the chain complex made up of the same partitioned labeled planar trees with differential equal only to 
$\dif_2^{split}$. This chain complex  decomposes as a direct sum of the sub-chain complexes made up of partitioned labeled planar trees with the same underlying planar tree. When the number of inner edges is greater than $1$, this chain complex is acyclic, since it is isomorphic to the $k^{\text{th}}$-tensor power of the acyclic chain complex 
$$ 0 \to \KK \to \KK \to 0\ ,$$
where $k$ is the number of internal edges. So the second page of this spectral sequence is isomorphic to 
$\oP(n)\otimes T(s\bar\Sy_n)$, with differential equal to the one of $\P$ and the differential of the normalized bar construction of the symmetric groups, see next paragraph and next proof for more details. Since $\P$ is made up of projective $\KK$-modules and since the normalized bar complex of the symmetric groups is a resolution of $\KK$, the third page of this spectral sequence is isomorphic to the underlying homology of $\oP$, that is $H_\bullet(\oP, \dif_\P)$. This is also the case for the other spectral sequence, the one associated to the operad $\oP$. In the end, the morphism $ \widetilde{\Omega} \widetilde{\B} \, \oP\to \oP$ induces an isomorphism between the third pages of the two spectral sequences. So the convergence theorem for spectral sequences proves that $ \widetilde{\Omega} \widetilde{\B} \, \oP\to \oP$ is a quasi-isomorphism. \\

Recall that the underlying graded colored $\Sy$-module of the cooperad $q\calO^{\ac}$ is non-negatively graded. So, when the underlying graded $\NN$-module of the operad $\P$ is non-negatively graded, it is also the case for the underlying graded $\NN$-module of its higher bar construction since 
$\widetilde{\B} \, \oP=q\calO^{\ac}(\oP)$.  By definition, one can see that the $\KK$-module components 
$q\calO^{\ac}(\oP)_d(n)$ of the higher bar construction are isomorphic to tensor products of $\KK$-modules of the form $\oP_e(m)$. Since these latter ones are all projective, the $\KK$-module components of the higher bar construction are projective. Finally,  Proposition~\ref{prop:CobarCofibrant} applies and proves 
the cofibrancy property.
\end{proof}

It turns out that this functorial resolution for augmented operads is actually isomorphic to the classical bar-cobar resolution of the Hadamard tensor product of $\P$ with the Barratt--Eccles operad $\calE$: 
$$(\P \otimes_{\rm H} \calE)(n):=\P(n) \otimes \calE(n)   \ .$$

Recall that the Barratt--Eccles operad $\calE$ is defined as follows, see \cite{BergerFresse04} for more details. 
The underlying graded $\Sy$-module is defined by 
$$\calE(n)_d := \KK\{  (\sigma_1, \ldots, \sigma_{d+1}); \sigma_i\in \Sy_n, \sigma_i\neq \sigma_{i+1}\} \ , $$ 
in arity $n$ and degree $d$. The  symmetric group acts diagonally on the right: 
$$(\sigma_1, \ldots, \sigma_{d+1})^\sigma:=(\sigma_1\sigma, \ldots, \sigma_{d+1}\sigma)\ . $$
(Notice that this convention differs from \cite{BergerFresse04}.)

The partial composition maps in the Barratt--Eccles operad are given as follows. 
First recall that the set-theoretical operad $Ass(n)=\Sy_n$, made up of the symmetric groups, carries partial composition maps. For $\sigma\in \Sy(n)$ and $\tau \in \Sy(k)$, the partial composite $\sigma \circ_{i} \tau$ is equal to
$$\sigma \circ_{i} \tau:=[f(\sigma(1)) \cdots f(\sigma(i-1)) \ (\tau(1)+\sigma(i)-1) \cdots (\tau(k)+\sigma(i)-1) \ f(\sigma(i+1)) \cdots f(\sigma(n))]\ , $$
where $f$ denotes the function defined by $f(x)=x$, for $x\leq \sigma(i)$,  and by $f(x)=x+k-1$, for $x> \sigma(i)$.
In plain words, this amounts to inserting accordingly the permutation $\tau$ at the $i^{\text{th}}$ place of the permutation $\sigma$.
So notice that, we have  
  \begin{align}\label{eqn:RelAss}
    \sigma \circ_i \tau &= (\sigma \circ_i \id_k)(\id_n \circ_i \,\tau) = \sigma''\tau' = (\id_n \circ_{\sigma(i)} \tau)(\sigma \circ_i \id_k) = \tau'\sigma'' ,
  \end{align}
with the notations introduced in Definition~\ref{def:Operad}.
Now the partial composition map in the Barratt--Eccles operad is equal to 
  \begin{align*}
    (\sigma_1,\dots,\sigma_{d+1}) \circ_i (\tau_1,\dots,\tau_{e+1})
    &:= \sum_{(x_\bullet,y_\bullet)}
      \pm (\sigma_{x_0+1} \circ_i \tau_{y_0+1},\dots,\sigma_{x_{d+e}+1} \circ_i \tau_{y_{d+e}+1}) \ ,
  \end{align*}
  where the sum is taken over the non-decreasing paths $(x_\bullet,y_\bullet)=\{(x_i,y_i)\}_{0\leq i\leq d+e}$ 
  from $(0,0)$ to $(d,e)$ on the grid $\N \times \N$.
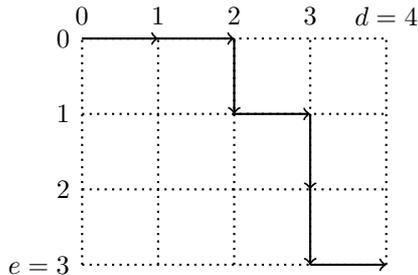
    \begin{align*}
    \begin{aligned}\begin{tikzpicture}[thick]
      \draw[dotted] (0,0) grid (4,3);
      \draw[->] (0,3) -- (1,3);
      \draw[->] (1,3) -- (2,3);
      \draw[->] (2,3) -- (2,2);
      \draw[->] (2,2) -- (3,2);
      \draw[->] (3,2) -- (3,1);
      \draw[->] (3,1) -- (3,0);
      \draw[->] (3,0) -- (4,0);
      \foreach \x in {0, 1, 2, 3}
        \node at (\x,3.3) {$\x$};
      \node at (4,3.3) {$d=4$};
      \foreach \y/\l in {0/e=3, 1/2, 2/1, 3/0}
        \node[label=left:{$\l$}] at (0.1,\y) {};
    \end{tikzpicture}\end{aligned}
  \end{align*}
Such a path is made up of $d$ horizontal segments and $e$ vertical segments. Let us number the segments  from $1$ to $d+e$, 
beginning with $(x_0, y_0)\to (x_1, y_1)$ and following the path. 
We consider  the shuffle permutation reordering the segments, such that the $d$ horizontal segments appear first 
and the $e$ vertical segments appear after. 
In the above example, we get $[1247356]$. 
The sign in the formula defining the above partial composition map is the signature of this shuffle permutation.

Since the components of the Barratt--Eccles operad are defined by the normalized bar complex of the symmetric groups $\Sy_n$, the differential map is equal to 
$$\dif_\calE(\sigma_1,\dots,\sigma_{d+1}):=\sum_{i=1}^{d+1} (-1)^{i-1}  (\sigma_1, \dots, \widehat{\sigma_i}, \ldots, \sigma_{d+1})\ ,$$
where the right-hand term is zero when $\sigma_{i-1}=\sigma_{i+1}$.

\begin{proposition}\label{prop:BarCobarBE}
There exists a natural isomorphism of augmented dg operads
$$\I \oplus \widetilde{\Omega} \widetilde{\B}\,  \overline{\P}
 \cong \Omega \B \, (\P \otimes_{\rm H} \calE) $$
between the augmentation of the higher bar-cobar construction 
of $\oP$  
and the bar-cobar construction of the Hadamard tensor product of the operad $\P$ with the Barratt--Eccles operad.  
\end{proposition}

\begin{proof}
We  consider first  the case when $\opd{P} = Com$ and we note that $Com \otimes_{\rm H} \calE = \calE$.  We define an isomorphism of dg operads 
  $\Psi\colon  \I \oplus \widetilde{\Omega} \widetilde{\B} \, \overline{Com} \to \Omega\B\,  \calE$.  Such a map is
  uniquely determined by its restriction to the space of  generators $\widetilde{\B}\,  \overline{Com}$.  As a graded $\N$-module,
  $\widetilde{\B}\,  \overline{Com} = q\opd{O}^\ac \circ \overline{Com} \cong \nsOpd^\ac \circ 
  T^c(s \bar\Sy)
   \circ \overline{Com}$, as shown in Lemma~\ref{lem:qO!}. 
  We first define a morphism of graded $\N$-modules $\psi\colon   T^c(s \bar\Sy) \circ \overline{Com} \to \calE$ by
  \[ \psi(s\sigma_1,\dotsc, s\sigma_d)
     := (\sigma_1\dotsm\sigma_d, \sigma_2\dotsm\sigma_d, \dotsc, \sigma_d, \id_n) \ , \]
     in arity $n$ and  degree
  $d$.
   Note that the image of $\psi$ is indeed non-degenerate, since $\sigma_i \neq \id$ for all
  $i$.
  This map $\psi$ extends to
  \[ \Psi|_{\widetilde{\B} \, \overline{Com}} \colon \widetilde{\B} \, \overline{Com} \to s^{-1}\B\, \calE \hookrightarrow \Omega\B\, \calE \]
  as follows.  Since $\nsOpd^\ac$ has a basis made up of planar rooted trees, there exists a basis 
  of $\widetilde{\B}\,  \overline{Com} \allowbreak \cong\allowbreak\nsOpd^\ac \circ 
  T^c(s \bar\Sy)
   \circ \overline{Com}$  made up of planar rooted trees such that each vertex has arity $n\ge 2$  and is labeled by a finite sequence of permutations of $\Sy_n$. 
The image of such an element under 
$\Psi|_{\widetilde{\B}\,  \overline{Com}}$ is the same underling tree, forgetting its planar structure, where we apply $\psi$ at each vertex. Each vertex of the resulting tree is suspended and the global tree is desuspended, resulting in an element of  $s^{-1}\B\, \calE= s^{-1}\mathcal{T} (s \overline{\calE})$.
Since the degree of a planar tree representing an element of   $\nsOpd^\ac$ is given by the number of internal edges, the map $\Psi|_{\widetilde{\B}\,  \overline{Com}}$ has degree $0$. \\

The map $\Psi\colon  \I \oplus \widetilde{\Omega} \widetilde{\B} \, \overline{Com} \to \Omega\B\,  \calE$ is by definition a morphism of operads.  The underlying space of the bar construction of an operad is made up of rooted trees (in space) with leaves labeled by $\{1,\ldots, n\}$ and with vertices labeled by the elements of the operad together with an identification of the inputs with the action of the symmetric group, see \cite[\S 5.6.1]{LodayVallette12}. Therefore, the underlying space of the bar construction $\B\, \calE$ of the Barratt--Eccles operad $\calE$ admits a basis made up of planar rooted trees where each vertex has  arity $n\ge 2$ and is labeled by a  sequence of permutations of the form 
$(\sigma_1,\dots,\sigma_{d}, \id_n)$. (One can use the last element of the strings to provide the rooted tree with a planar structure.) Since  the map $\psi$ is an isomorphism, this shows that the map $\Psi$ is an isomorphisms of graded operads. \\

  It remains to show, that the morphism of graded  operads  $\Psi$ respects the differentials. This again can be
  checked on the space of generators $\widetilde{\B}\,  \overline{Com} $.
Recall from the previous proof that the higher bar-cobar construction is made up of partioned planar rooted trees with vertices of arity $n\ge 2$ labeled by strings of nontrivial permutation of $\Sy_n$. 
 The differential $\dif$ on $\widetilde{\Omega} \widetilde{\B} \, \overline{Com}$ is a sum of several parts.  
First, there are three terms 
$\dif_1=\dif_1^{sh} + \dif_1^{l} + \dif_1^{int}$, 
 coming from the higher bar construction, which apply only inside each partition. 
The term $\dif_1^{sh}$ contracts inner edges and shuffles the two strings of permutations labeling the two associated vertices. 
The term $\dif_1^l$ drops the first element of each string of permutations. (This corresponds to the trivial action of this element on the corolla viewed as an element of $Com$). 
The term $\dif_1^{int}$ is equal to the sum of the products of two consecutive elements of a string labeling a vertex, provided that it is not the identity, which gives $0$ in this case. 
Then, there are two  terms 
$\dif_2=\dif_2^{split} + \dif_2^r$, coming from the higher cobar construction, which apply to each partition. 
The term $\dif_2^{split}$ splits each part of the partition into two along an inner edge. 
The term $\dif_2^r$ extracts the right most external term of each string labeling the vertex of the tree inside the partition, pulls it out of the partition and make it act on the permutation labeling the leaves of the global tree. 
 The total differential is then the sum
  \[ \dif = \dif_1 - \dif_2 = (\dif_1^{sh} + \dif_1^{l} + \dif_1^{int}) - (\dif_2^{split} + \dif_2^r) . \]

  On the bar-cobar construction of the Barratt--Eccles operad $\Omega\B\, \calE$,  the differential is equal to the sum of  the following three parts. 
  First, there is the internal differential $\dif_{int}$ induced by the differential $\dif_\calE$ on $\calE$.
Second, there is the differential $\dif_{bar}$ coming from the differential on the bar construction $\B\, \calE$ which is produced by the partial 
  composition maps in $\calE$. 
  And finally, there is the differential  $\dif_{cobar}$ coming the differential of the cobar construction which to split each partition into two by cutting along any internal edge inside each partition.

The sum $\dif_1^l + \dif_1^{int} -\dif_2^r$ of the three parts of the differential on $\widetilde{\Omega} \widetilde{\B} \, \overline{Com}$ involving the strings of symmetric group elements  is sent to the internal differential $\dif_{int}$ of the
  Barratt--Eccles operad under the morphism $\Psi$:
  \begin{eqnarray*}
&&    \psi\big( (\dif_1^l + \dif_1^{int} - \dif_2^r)(\sigma_1,\dots,\sigma_d) \big)\\
    &=& \psi\left( (\sigma_2, \dotsc, \sigma_d)
    + \sum_{i=2}^d (-1)^{i-1} (\sigma_1,\dots,\sigma_{i-1}\sigma_i,\dots,\sigma_d)
    - (-1)^{d-1} (\sigma_1,\dots,\sigma_{d-1}|{\sigma_d}) \right) \\
    &=& \sum_{i=1}^d (-1)^{i-1} (\sigma_1\dotsm\sigma_d, \dotsc, \widehat{\sigma_i\dotsm\sigma_d}, \dotsc, \sigma_d, \id_n) + (-1)^d (\sigma_1\dotsm\sigma_{d-1}, \dotsc, \sigma_{d-1}, \id_n)^{\sigma_d}
\end{eqnarray*}
    and
\begin{eqnarray*}
    \dif_{int}( \psi(\sigma_1, \dotsc, \sigma_d) )
    &=& \dif_{int}(\sigma_1\dotsm\sigma_d, \dotsc, \sigma_d, \id_n) \\
    &=& \sum_{i=1}^d (-1)^{i-1} (\sigma_1\dotsm\sigma_d, \dotsc, \widehat{\sigma_i\dotsm\sigma_d}, \dots, \sigma_d, \id_n) + (-1)^d (\sigma_1\dotsm\sigma_d, \dotsc, \sigma_d) .
  \end{eqnarray*}
  Note that in the first case, the last term amounts to splitting apart  $\sigma_d$ and making it act on the underlying tree. In the tree module underlying $\Omega \B \, \calE$, this action on rooted trees is identified with the symmetric group action on the strings labeling the vertices.

  The next term $-\dif_2^{split}$ of the differential on $\widetilde{\Omega} \widetilde{\B} \, \overline{Com}$
 amounts to  cutting each partition into two along any inner edge. It is sent to $-d_{cobar}$ in $\Omega \B \, \calE$, which does the same things on partitioned rooted trees and with the same minus sign, see \cite[\S 6.5.2]{LodayVallette12} where the sign is creating by the formula $\Delta_s(s^{-1})=\allowbreak-s^{-1}\otimes s^{-1}$.

  The last part $\dif_1^{sh}$ contracts inner edges inside partions and shuffles the two strings of permutations labeling the two associated vertices. Let $t((\sigma_1,\dotsc,\sigma_d), (\sigma_{d+1},\dotsc,\sigma_{d+e}))$ be a planar rooted subtree with two vertices of arity $n$ and $k$ and labeled respectively by two strings $(\sigma_1,\dotsc,\sigma_d)\in \Sy_n^d$ and $(\sigma_{d+1},\dotsc,\sigma_{d+e})\allowbreak\in \allowbreak\Sy_k^e$.  
Since each element of the symmetric group in the higher bar construction comes in degree $1$, for each shuffle permutation there is a sign, which is equal to the 
  signature of the permutation.
The image under the composite $\psi \circ \dif_2^{sh}$ of the above labeled tree is equal to 
  \begin{align*}
   &  \psi\big( \dif_2^{sh} ( t( (\sigma_1,\dotsc,\sigma_d), (\sigma_{d+1},\dotsc,\sigma_{d+e}) ) ) \big)\\
    &\quad = \psi\left( \sum_{\alpha \in \text{Sh}(d,e)}
      \mathrm{sgn}(\alpha)\, (\tilde\sigma_{\alpha^{-1}(1)},\dotsc,\tilde\sigma_{\alpha^{-1}(d+e)}) \right) \\
    &\quad = \sum_{\alpha \in \text{Sh}(d,e)} \mathrm{sgn}(\alpha)\,
      (\tilde\sigma_{\alpha^{-1}(1)}\dotsm\tilde\sigma_{\alpha^{-1}(d+e)},\dotsc,\tilde\sigma_{\alpha^{-1}(d+e)},
      \id_{n+m-1})
      \ ,
  \end{align*}
where  $\text{Sh}(d,e)$ denotes the set of $(d,e)$-shuffles, that is permutations $\alpha \in \Sy_{d+e}$
  such that \ $\alpha(1) < \dotsc < \alpha(d)$ and $\alpha(d+1) < \dotsc < \alpha(d+e)$, and where we use the overall notation
  $\tilde\sigma_j = \sigma_j''$ for $j \leq d$ and $\tilde\sigma_j = \sigma_j'$ for $j > d$.
Notice that the $(d,e)$-shuffles are in one-to-one correspondence with non-decreasing paths $(x_\bullet,y_\bullet)=\allowbreak\{(x_i,y_i)\}_{0\leq i\leq d+e}$ 
  from $(0,0)$ to $(d,e)$ on the grid $\N \times \N$.
On the other hand, we have 
  \begin{align*}
& \dif_{bar}\big(\Psi(t((\sigma_1,\dotsc,\sigma_d), (\sigma_{d+1},\dotsc,\sigma_{d+e})))\big) \\
    & \quad = \dif_{bar}\big(t(\psi(\sigma_1,\dots,\sigma_d), \psi(\sigma_{d+1},\dots,\sigma_{d+e}))\big) \\
        & \quad =\psi(\sigma_1,\dots,\sigma_d) \circ_i \psi(\sigma_{d+1},\dots,\sigma_{d+e}) \\
    &\quad = (\sigma_1\dotsm\sigma_d,\dots,\sigma_d,\id_n) \circ_i (\sigma_{d+1}\dotsm\sigma_{d+e},\dots,\sigma_{d+e}, \id_k) \\
    &\quad = \sum_{(x_\bullet,y_\bullet)}
      \pm ((\sigma_{x_0+1}\dotsm\sigma_d) \circ_i (\sigma_{y_0+d+1}\dotsm\sigma_{d+e}),\dots,
        (\sigma_{x_{d+e}+1}\dotsm\sigma_d) \circ_i (\sigma_{y_{d+e}+d+1}\dotsm\sigma_{d+e}), \\
        & \qquad \qquad\qquad\ \id_n \circ_i \id_k) \ ,
  \end{align*}
where $i$ is the leaf of the root corolla of the tree $t$ where the second corolla is attached to.  
It remains to show that
  \begin{align*}
    (\sigma_{x_j+1}\dotsm\sigma_d) \circ_i (\sigma_{y_j+d+1}\dotsm\sigma_{d+e})
    = \tilde\sigma_{\alpha^{-1}(j+1)}\dotsm\tilde\sigma_{\alpha^{-1}(d+e)} \ .
  \end{align*}
First, it is straightforward to prove  
$\alpha^{-1}([j,d+e]) = [x_j+1,d] \cup [y_j+d+1,d+e]$, for instance by induction on $j$. 
Then, using Relation~(\ref{eqn:RelAss}) with  $(\sigma\tau)'=\sigma'\tau'$ and $(\sigma\tau)''=\sigma''\tau''$, we get 
  \begin{eqnarray*}
     \tilde\sigma_{\alpha^{-1}(j+1)}\dotsm\tilde\sigma_{\alpha^{-1}(d+e)} 
     &=& \sigma''_{x_j+1}\dotsm\sigma''_d \sigma'_{y_j+d+1}\dotsm\sigma'_{d+e}\\
          &=& (\sigma_{x_j+1}\dotsm\sigma_d)'' (\sigma_{y_j+d+1}\dotsm\sigma_{d+e})'\\
          &=& (\sigma_{x_j+1}\dotsm\sigma_d) \circ_i (\sigma_{y_j+d+1}\dotsm\sigma_{d+e})
     \ .
  \end{eqnarray*}
  This completes the proof for $\P = Com$.  
  
  The general case of any augmented dg operad $\P$ is treated in the same way. 
We consider now the following morphism of graded $\N$-modules 
$\psi\colon   T^c(s \bar\Sy) \circ \oP \to \calE$ by
  \[ \psi(\mu; s\sigma_1,\dotsc, s\sigma_d)
     := (\mu^{\sigma_1\dotsm\sigma_d};\sigma_1\dotsm\sigma_d, \sigma_2\dotsm\sigma_d, \dotsc, \sigma_d, \id_n) \ . \]
and the induced morphism of operads 
$\Psi\colon  \I \oplus \widetilde{\Omega} \widetilde{\B} \, \oP \to \Omega\B\,  (\P\otimes_H \calE)$.

The vertices of the partitioned planar rooted trees of $\widetilde{\Omega} \widetilde{\B}\,  \overline{\P}$ are now labeled by elements of $\P$ and strings of permutation and the vertices of the partitioned rooted trees of $\Omega \B \, (\P \otimes_{\rm H} \calE)$ are labelled by elements of $\oP\otimes \overline{\calE}$. Therefore, the same arguments apply to prove that $\Psi$ is an isomorphism. 

  Then, one has to deal with the internal differential $\dif_\P$ of the  dg operad $\P$. It is easy to see that the extra   term induced by $\dif_\P$ on the differentials of $ \widetilde{\Omega} \widetilde{\B}\,  \overline{\P}$ and $\Omega \B \, (\P \otimes_{\rm H} \calE)$ respectively are sent to one another under the map $\Psi$. 
The term $\dif_1^l$ now incorporates the action of the symmetric group on $\P$: 
$$\dif_1^l(\mu; \sigma_1, \ldots, \sigma_d)=\allowbreak(\mu^{\sigma_1}; \sigma_2,\allowbreak \ldots, \sigma_d)\ .$$
And the terms $\dif_1^l + \dif_1^{int} - \dif_2^r$ are still sent to $\dif_{int}$ under the map $\Psi$ since
 \begin{eqnarray*}
&&    \psi\big( (\dif_1^l + \dif_1^{int} - \dif_2^r)(\mu; \sigma_1,\dots,\sigma_d) \big)=
    \dif_{int}( \psi(\mu; \sigma_1, \dotsc, \sigma_d) )\\
    &=& \sum_{i=1}^d (-1)^{i-1} (\mu^{\sigma_1\dotsm\sigma_d}; \sigma_1\dotsm\sigma_d, \dotsc, \widehat{\sigma_i\dotsm\sigma_d}, \dots, \sigma_d, \id_n) + (-1)^d (\mu^{\sigma_1\dotsm\sigma_d}; \sigma_1\dotsm\sigma_d, \dotsc, \sigma_d) .
  \end{eqnarray*}
Since they do not involve the operad $\P$, the two terms $-\dif_2^{split}$ and $-\dif_{bar}$ still correspond under $\Psi$. 
 
 The component $\dif_1^{sh}$ now also includes the partial composite $\circ_i$ of the two elements of $\oP$ labeling the two vertices located on both sides of internal edges. This is also the case for the component  $\dif_{bar}$. They actually satisfy 
  \begin{align*}
   &  \psi\big( \dif_2^{sh} ( t( (\mu; \sigma_1,\dotsc,\sigma_d), (\nu; \sigma_{d+1},\dotsc,\sigma_{d+e}) ) ) \big)\\
    &\quad = \sum_{\alpha \in \text{Sh}(d,e)} \mathrm{sgn}(\alpha)\,
      ((\mu\circ_i \nu)^{\tilde\sigma_{\alpha^{-1}(1)}\dotsc\tilde\sigma_{\alpha^{-1}(d+e)}}; \tilde\sigma_{\alpha^{-1}(1)}\dotsm\tilde\sigma_{\alpha^{-1}(d+e)},\dotsc,\tilde\sigma_{\alpha^{-1}(d+e)},
      \id_{n+m-1})
  \end{align*}
and 
  \begin{align*}
& \dif_{bar}\big(\Psi(t((\mu; \sigma_1,\dotsc,\sigma_d), (\nu; \sigma_{d+1},\dotsc,\sigma_{d+e})))\big) \\
    &\quad = \sum_{(x_\bullet,y_\bullet)}
      \pm (\mu^{\sigma_1\dotsm\sigma_d}\circ_i \nu^{\sigma_{d+1}\dotsm\sigma_{d+e}}; (\sigma_{x_0+1}\dotsm\sigma_d) \circ_i (\sigma_{y_0+d+1}\dotsm\sigma_{d+e}),\dots,\\
        & \qquad \qquad\qquad\         (\sigma_{x_{d+e}+1}\dotsm\sigma_d) \circ_i (\sigma_{y_{d+e}+d+1}\dotsm\sigma_{d+e}), \id_n \circ_i \id_k) \ .
  \end{align*}
It remains to notice that 
\begin{eqnarray*}
(\mu\circ_i \nu)^{\tilde\sigma_{\alpha^{-1}(1)}\dotsc\tilde\sigma_{\alpha^{-1}(d+e)}}
=
(\mu\circ_i \nu)^{(\sigma_{1}\dotsm\sigma_d)'' (\sigma_{d+1}\dotsm\sigma_{d+e})'}=
\mu^{\sigma_1\dotsm\sigma_d}\circ_i \nu^{\sigma_{d+1}\dotsm\sigma_{d+e}}
\end{eqnarray*}
to conclude the proof. 
\end{proof}

\subsection{$E_\infty$-operad}\label{subsec:EinftyOp} 
Recall that an \textit{$E_\infty$-operad} is
an operad quasi-isomorphic to the operad $Com$ and whose components are free (or projective) $\KK[\Sy_n]$-modules. By \cite[Proposition~$4.3$]{BergerMoerdijk03}, any cofibrant resolution of the operad $Com$ is an $E_\infty$-operad. 
However, in application, for instance in algebraic topology like in \cite{Fresse08, Fresse10bis}, it is sometimes necessary to work with a cofibrant $E_\infty$-operad and not only an $E_\infty$-operad. 

We made explicit the higher bar-cobar construction of the operad $\P=Com$ of commutative algebras in the core of the proof of Proposition~\ref{prop:BarCobarBE}. 
This construction provides us with a cofibrant $E_\infty$-operad over any ring. 

\begin{theorem}
The  augmentation of the higher bar-cobar counit 
$$\calE_\infty:=\I \oplus \widetilde{\Omega} \widetilde{\B} \, \overline{Com}\xrightarrow{\sim} Com$$
 is a cofibrant $E_\infty$-operad over any ring, which is moreover a Hopf operad resolution.
\end{theorem}

\begin{proof}
The operad $Com$ is non-negatively graded. Its components are made of up the ground ring $Com(n)\allowbreak=\KK$, for any $n \ge 1$, so they are projective $\KK$-modules. Therefore Theorem~\ref{thm:BarCobarRes} applies and gives the result. 

Proposition~\ref{prop:BarCobarBE} shows that the operad $\calE_\infty\cong \Omega \B \, \calE$ is isomorphic to the bar-cobar construction of the Barratt--Eccles operad. Since this later one is a Hopf operad, it is also the case of its bar-cobar construction, by \cite[Theorem~$\S$2.A (c)]{Fresse07}. And the composition 
$\calE_\infty\cong \Omega \B \, \calE \qi \calE \qi Com$ is a quasi-isomorphism of Hopf operads. 
\end{proof}

\begin{remark}
Let us emphasize that, by Proposition~\ref{prop:BarCobarBE}, this theorem was already known under the following form: the present cofibrant $E_\infty$-operad is isomorphic to the classical bar-cobar construction of the Barratt--Eccles operad, which is a cofibrant $E_\infty$-operad over any ring by \cite[Theorem~$5.1$]{BergerMoerdijk07}. 
\end{remark}

The operad $\calE_\infty$ is by  given  partitioned planar rooted trees with vertices of valence at least $2$, labeled by  strings of nontrivial  permutations and with the leaves globally labeled by a permutation 
$$   
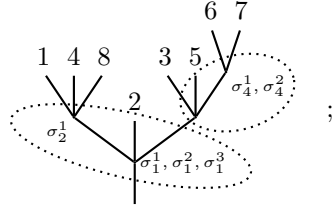
\begin{aligned}\begin{tikzpicture}[optree,
     level distance=12mm,
     level 2/.style={sibling distance=16mm},
     level 3/.style={sibling distance=8mm}]
   \node{}
     child { node[comp,label=right:{\tiny $\sigma_1^1, \sigma_1^2, \sigma_1^3$}] (a) {}
       child { node[comp,label=183:{\tiny $\sigma_2^1$}] (b) {}
         child { node[label=above:$1$]{} }
         child { node[label=above:$4$]{} }
         child { node[label=above:$8$]{} } }
         child { node[label=above:$2$]{} }
         child { node[comp] (c) {}
         child { node[label=above:$3$]{} }
         child { node[label=above:$5$]{} }
         child { node[comp,label=-3:{\tiny $\sigma_4^1, \sigma_4^2$}] (d) {}
           child { node[label=above:$6$]{} }
           child { node[label=above:$7$]{} } } } };
   \node[draw,dotted,ellipse,rotate=-13,inner sep=-2mm,yscale=0.7,xscale=1.8,xshift=2mm,fit=(a) (b)] {};
   \node[draw,dotted,ellipse,rotate=15,inner sep=-2mm,yscale=0.8,xscale=2.3,yshift=-0.8mm,xshift=1.5mm,fit=(c) (d)] {};
 \end{tikzpicture}\end{aligned} \quad ;$$
 it is equipped with a differential given by the sum of three types of terms:
\begin{enumerate}
\item at each vertex, one deletes the left-most permutation $\sigma_j^1$,  one composes pairs of permutations when they are not inverse to each other, and one removes the last permutation $\sigma_j^{i_j}$ of the string, pulls it out of the underlying tree (action which may modify this tree) and rearranges the labels of the leaves accordingly, 

\item inside each subtree, one contracts every edge and shuffles the two strings of permutations in all possible ways, 

\item for each subtree, one splits it into two subtrees in all possible ways, the number of which is equal to the number of internal edges. 
\end{enumerate}

By definition, an \textit{$E_\infty$-algebra} is an algebra over an $E_\infty$-operad. This definition offers an interesting degree of freedom. But it is sometimes useful  to have at hand an explicit definition by means of generators and relations. The form of the present $E_\infty$-operad $\calE_\infty$ allows us to give such a definition. Moreover, this definition carries canonical properties since the resolution $\calE_\infty$ that we use is a cofibrant and Hopf operad. 

\begin{proposition}\label{prop:EinftyAlg}
Let $(A, \dif_A)$ be a chain complex equipped with operations 
\begin{align*}
\left\{  \mu_{t(\sigma)} \ : \ A^{\otimes n} \to A \right\}
\end{align*}
indexed by planar rooted trees with vertices of valence $m$ at least $2$   labeled by string of nontrivial permutations of $\Sy_m$. They are supposed to be of degree 
$$|\mu_{t(\sigma)}|= \text{number of internal edges}+ \text{number of permutations}$$
 and to satisfy the relations
\begin{eqnarray*}
\partial_A \mu_{t(\sigma)} &-&\sum_{\text{internal edges}\atop   j-e-l} 
\varepsilon_1\, \mu_{t/e(\bar \sigma_1, \ldots, \shuffle(\bar \sigma_j, \bar \sigma_l), \ldots, \bar \sigma_k)}
+\sum_{\text{internal edges}, \atop t=t_1 \circ_i t_2} 
\varepsilon_2 \, \mu_{t_1(\bar \sigma_1, \dots)} \circ_i \mu_{t_2(\bar \sigma_l, \ldots)}
\\
&-&\sum_{j=1}^k 
\varepsilon_3\, \mu_{t(\bar \sigma_1, \ldots, (\sigma_j^2, \ldots, \sigma_j^{i_j}) , \ldots, \bar \sigma_k)}
-\sum_{j=1}^k \sum_{l=1}^{i_j-1} 
\varepsilon_4\,  \mu_{t(\bar \sigma_1, \ldots, (\sigma_j^1, \ldots,  \underbrace{{\scriptstyle \sigma_j^l\sigma_j^{l+1}}}_{\ne \id}   , \ldots, \sigma_j^{i_j}) , \ldots, \bar \sigma_k)}\\
&+&\sum_{j=1}^k \varepsilon_5\, \left(\mu_{\tilde t(\bar \sigma_1, \ldots, (\sigma_j^1, \ldots, \sigma_j^{i_j-1}) , \ldots, \bar \sigma_k)}\right)^{\tilde \sigma_j^{i_j}}
=0\ ,
\end{eqnarray*}
where the sign $\varepsilon_1$ comes from the permutations of the $\sigma$'s, as described in Section~\ref{subsec:KDCoopQ}, the sign $\varepsilon_2$ comes from the reordering of the tree into the  two parts, the sign $\varepsilon_3$ is equal to $(-1)^{|t|+|\bar \sigma_1|+\cdots+|\bar \sigma_{j-1}|}$, the sign $\varepsilon_4$ is equal to $(-1)^{|t|+|\bar \sigma_1|+\cdots+|\bar \sigma_{j-1}|+l}$, and the sign 
$\varepsilon_5$ is equal to $(-1)^{|t|+|\bar \sigma_1|+\cdots+|\bar \sigma_{j-1}|+|\sigma_j^1|+ \cdots + |\sigma_j^{i_j-1}|}$.

Such a data $(A, \dif_A, \{ \mu_{t(\sigma)} \})$ is an $E_\infty$-algebra. 
\end{proposition}

\begin{proof}
An  $\calE_\infty$-algebra structure on the chain complex $(A, \dif_A)$ amounts to a morphism of dg operads $\calE_\infty \to \End_A$, data  equivalent to a morphism of $\NN$-modules $\mu : \widetilde{\B} \overline{Com} \to \End_A$ satisfying the Maurer--Cartan equation~(\ref{align:MCeqKappa}): 
$\partial \mu + \mu \ast\mu +\Delta \mu =0$, by Proposition~\ref{prop:BarCobarTwKappa}. Since $\mu$ is concentrated in degree $0$, all the signs are directly given by the ones of the higher bar construction $\widetilde{\B} \overline{Com}$ described in Section~\ref{subset:HighBarConst}. 
\end{proof}

\begin{remark} Again, an equivalent explicit definition can be obtained using the bar-cobar construction of the Barratt--Eccles operad. The existence of a cofibrant Hopf $E_\infty$-operad plays a crucial role, for instance in B. Fresse's works \cite{Fresse08, Fresse10bis}. But the actual description of the category of algebras over it  does not seem to be present in the existing literature. We believe that it could be useful to have such an explicit description at hand. Moreover, the operad which emerges from the higher bar-cobar construction seems to provide us with a ``simple'' model since it is made of planar trees and easy combinatorics of permutations. 
\end{remark}

In characteristic $0$, one can use the Koszul model of the operad $Com$ to get a suitable notion of an $E_\infty$-algebra. This only amounts to resolving the associative relation of commutative algebras. In the above definition, there are strings of permutations whose role is to resolve, in a coherent way, the trivial symmetric group action on the components of the operad $Com$ over any ring. Let us unravel a little bit the above definition. 

An $E_\infty$-algebra, in the above sense, is equipped with a degree $0$ operation 
$$\mu_{\Y} : A^{\otimes 2} \to A\ . $$
The above relation applied to the tree $\Y$ is equivalent to the fact that $\dif_A$ is a derivation with respect to 
$\mu_{\Y}$. It is also equipped with 
a degree $1$ operation 
$$\mu_{\Yb} : A^{\otimes 2} \to A\ , $$ 
which satisfies 
$$\partial_A\,  \mu_{\Yb}  =\mu_{\Y}-  (\mu_{\Y})^{(12)}\ .$$
So $\mu_{\Yb}$ is a homotopy for the commutativity relation of $\mu_{\Y}$. 
It also comes equipped with a degree $0$ operation 
$$\mu_{\W} : A^{\otimes 3} \to A\ . $$
and two degree $1$ operations 
$$\mu_{\YYl} : A^{\otimes 3} \to A \quad \text{and} \quad 
\mu_{\YYr} : A^{\otimes 3} \to A\ , $$
satisfying respectively
$$\partial_A\,  \mu_{\YYl}  =\mu_{\W}-  \mu_{\Y} \circ_1 \mu_{\Y}
\quad \text{and} \quad
\partial_A\,  \mu_{\YYr}  =\mu_{\W}-  \mu_{\Y} \circ_2 \mu_{\Y}
\ . $$
So $-\mu_{\YYl}+\mu_{\YYr}$ is a homotopy for the associativity relation of $\mu_{\Y}$. 
The higher operations coming from more elaborate labeled planar trees provide us with coherent higher homotopies for these relations.

\begin{further}
Recall that, in the case of the set-theoretical nonsymmetric operad $As$, there exists a topological operad, made up of the Stasheff polytopes, see \cite{Stasheff63}, whose chain complex is the minimal model $A_\infty$ of the algebraic operad spanned by $As$. In this way, one can get a definition for a notion of an $A_\infty$-space as an algebra over this topological operad. 

Since the  operad $Com$ which encodes commutative algebras is also set-theoretical, one can look for a topological  operad whose chain complex would give the resolution $\calE_\infty$ of the operad $Com$. In this case, one would be able to define a suitable notion of an $E_\infty$-space. Such a question is made sensible by the fact that the operad $\calE_\infty$ is a Hopf operad, like the chain operad of a topological operad. 
\end{further}

\subsection{Higher Koszul duality theory} 
The higher bar-cobar construction of Theorem~\ref{thm:BarCobarRes} shares  nice properties: functorial, cofibrant over any ring. However, the underlying $\NN$-module  is quite huge. Given a particular operad $\P$ with some description, like a presentation by means of generators and relations, one might want to describe a smaller cofibrant resolution of $\P$. This is precisely what the 
 Koszul duality theory for operads performs, see \cite[Chapter~$7$]{LodayVallette12}. However, the original Koszul duality theory for operads of \cite{GinzburgKapranov94, GetzlerJones94} works well over a field of characterstic $0$. 

\begin{further}
The general pattern of the Koszul duality theory of algebras over an (colored) operad $\calO$ was given by J. Mill\`es in \cite{Milles12}. One should be able to apply it to the present case and thus obtain minimal models for operads generated by free $\Sy$-modules. 
In this way, the present paper already settles the first step of this tentative higher Koszul duality theory with the definitions of the higher bar and the higher cobar constructions. It remains to define the Koszul dual higher cooperad of an operad. This issue is more subtle in the present case since the underlying space of the higher bar construction is not a differential graded module but a curved one. So one cannot define the Koszul dual directly by the homology groups of the bar construction.

The case of the operad $Com$ is particularly interesting since this would simplify the above $E_\infty$-operad. This should give an operad related the cofibrant $E_\infty$-operad given in \cite{Fresse09bis} or to the chain complex of the Fox--Neuwirth cells of the Fulton--Mac\-Pher\-son operad and solve the issues raised by  the presence of "bad cells" in \cite{GetzlerJones94}. 
All of this  will be studied in a sequel to this paper. 
\end{further}


\section{Higher homotopy operads} 

The main result of Section~\ref{sec:one} allows us to define a new notion of an operad up to homotopy. Since we take into account the action of the symmetric groups on the same footing as the partial composition maps, the upshot is a higher notion of a homotopy operad where all the defining relations of an nu operad are relaxed. 
Such a higher notion includes all the previously known cases like $A_\infty$-algebras, $A_\infty$-modules, and homotopy nonsymmetric operads, for instance. 

Since this new notion is conceptually produced by the curved Koszul duality theory, we can endow it with a suitable notion of $\infty$-morphisms, which are shown to satisfy the required interesting homotopy properties. For instance, this allows us to describe the homotopy theory of dg augmented operads, like the homotopy transfer theorem,  over any ring. 

\subsection{Definition}

The cofibrant resolution $\calO_\infty:=\Omega \calO^{\ac}$ of the colored operad $\calO$ allows us to define 
a notion of operads up to homotopy with the required homotopy properties over any ring. 

\begin{definition}[Higher homotopy operad] A \emph{higher homotopy operad}, also called  a $\textit{higher operad}_\infty$, is an
$\calO_\infty$-algebra. 
\end{definition}

Using the above computation of the Koszul dual curved cooperad 
$\calO^\ac=(q\calO^{\ac}, \dif_{\calO^\ac},\theta)$, we can make the 
definition of higher homotopy operads explicit as follows. 

Let $\{(\P(n), \dif_{\P(n)}\}_{n\in \NN}$ be a differential graded $\NN$-module. We consider the space of maps from the Koszul dual curved cooperad $\calO^{\ac}$ to the endomorphism operad $\End_\P$: 
$$\Hom_\Sy(\calO^{\ac}, \End_\P):=\prod_{k\in \NN} 
\left( \prod_{n_1, \ldots, n_k\in \NN} 
\Hom\left(
\overline{\calO}^{\ac}(n; n_1, \ldots, n_k), \Hom\big(\P(n_1)\otimes \cdots\otimes \P(n_k), \P(n)\big)
\right)\right)^{\Sy_{k}}\ , $$
where $n=n_1+\cdots +n_k-k+1$.
This mapping space is endowed with a convolution pre-Lie product $f\star g$ defined by the following composition 
$$f\star g\ : \ \calO^{\ac}  \xrightarrow{\Delta_{(1)}} \calO^{\ac} \otimes \calO^{\ac} 
\xrightarrow{f\otimes g} \End_\P \otimes \End_\P \xrightarrow{\gamma_{\End_\P} } \End_\P\ .$$
As usual, the antisymmetrized bracket gives rise to a convolution Lie algebra structure
$$[f,g]:=f\star g - (-1)^{|f||g|} g \star f\ .$$
Let us denote by $\partial_\P$ the differential of the endomorphism operad $\End_\P$. We consider the classical derivation 
$$\partial(f):=\dif_\P \circ f - (-1)^{|f|}f \circ \dif_{\calO^{\ac}} $$
on the mapping space $\Hom_\Sy(\calO^{\ac}, \End_\P)$. Finally, the curvature $\theta$ of the Koszul dual curved cooperad $\KDO$ induces the following  curvature 
$$\Theta\ :\ \KDO \xrightarrow{-\theta} \I \to \End_\P\ .  $$

\begin{lemma}[\cite{HirshMilles12} \S$3.2.3$]
The quadruple $\big(\Hom_\Sy(\calO^{\ac}, \End_\P), \star, \partial, \Theta\big)$ 
forms a curved pre-Lie algebra, that is 
$$ \partial(\Theta)=0 \quad \text{and} \quad \partial^2=(\text{-} \star \Theta) - (\Theta \star \text{-})\ . $$
The quadruple $\big(\Hom_\Sy(\calO^{\ac}, \End_\P), [\, , ], \partial, \Theta\big)$
forms a curved Lie algebra, that is 
$$ \partial(\Theta)=0 \quad \text{and} \quad \partial^2=[\text{-}, \Theta]\ . $$
\end{lemma}

\begin{proposition}
A higher homotopy operad structure on a dg $\NN$-module $\P$ is equivalent to a 
degree $-1$ solution to the 
curved Maurer--Cartan equation:
\begin{eqnarray}\label{eqn:MCHiHoOp}
\partial \gamma+ \gamma \star \gamma  = \Theta
 \end{eqnarray}
in the curved convolution algebra $\Hom_\Sy(\calO^\ac,\End_\P)$, which vanishes on $\I\subset \calO^{\ac}$.
\end{proposition}

\begin{proof}
This is a direct application of \cite[Theorem~$3.4.1$]{HirshMilles12}.
\end{proof}

\begin{theorem}
A higher homotopy operad structure on a dg $\NN$-module $\{\P(n), \dif_{\P(n)}\}_{n\in \NN}$ amounts to  a collection of maps 
\begin{align*}
\left\{ \gamma_{t(\sigma)} : \P(n_1)\otimes \cdots \otimes \P(n_k) \to \P(n) \ ; 
\begin{aligned}\begin{tikzpicture}[optree,
      level distance=12mm,
      level 2/.style={sibling distance=12mm},
      level 3/.style={sibling distance=12mm}]
    \node{}
      child { node[comp,label={-3:{\tiny$\sigma^1_1$}}]{}
      node[comp,label={[label distance=1mm]183:$\mu$}]{}
        child 
        child { edge from parent[draw=none] }
        child { node[comp,label={[label distance=2mm]1:{\tiny $\sigma^1_2,\sigma^2_2$}}]{}
        node[comp,label={[label distance=1mm]183:$\nu$}]{}
          child { node[comp,label={[label distance=1mm]3:{\tiny $\sigma_3^1$}}]{}
          node[comp,label={[label distance=1mm]left:$\chi$}]{}
            child
            child }
        child { edge from parent[draw=none] }
          child
          child
          child }
        child { edge from parent[draw=none] }
        child { edge from parent[draw=none] }
        child } ;
  \end{tikzpicture}\end{aligned}
  \mapsto 
  \gamma_{t(\sigma)}(\mu, \nu, \chi)\right\}
\end{align*}
labeled by planar rooted  trees with vertices labeled by strings of nontrivial permutations $t(\sigma)=t(\bar\sigma_1, \ldots, \bar\sigma_k)$ and with at least one internal edge or one permutation, of degree 
$$|\gamma_{t(\sigma)}|=|t|+|\bar\sigma_1|+\cdots+|\bar\sigma_k|-1$$ and 
satisfying 
\begin{align}\label{RelationHoOp}
\begin{split}
\partial_\P(\gamma_{t(\sigma)})
+ 
\sum_{j=1}^k \sum_{l=1}^{i_j-1}
\pm\  \gamma_{t\big(\bar \sigma_1, \ldots, \bar \sigma_{j-1},
(\sigma_j^1, \ldots, {\SMALL \underbrace{{\tiny \sigma_j^{l}\sigma_j^{l+1}}}_{\neq \id}}, \ldots, \sigma_j^{i_j})
\bar \sigma_{j+1}, \ldots, \bar \sigma_{k}\big)}+\sum_{s\subset t, \shuffle} \pm\, 
\gamma_{t/s(\vec{\sigma})}\circ_{s_1} \gamma_{\tilde{s}({\SMALL \cev{\sigma}})}
\\
=
\left\{
\begin{array}{lll}
-\id_{\P(n)} & \text{when} & t(\sigma)= 
\begin{aligned}\begin{tikzpicture}[optree]
    \node{}
      child { node[dot,label=right:$\sigma^{-1}$]{}
        child { node[dot,label=right:$\sigma$]{}
          child { edge from parent node[right,near end]{\tiny$n$} } } } ;
  \end{tikzpicture}\end{aligned}\ ,\\
 0&  \text{otherwise}\ ,&
 \end{array}
 \right. 
 \end{split}
\end{align}
where the first sign is equal to    
$(-1)^{|t|+|\bar \sigma_1|+\cdots+|\bar \sigma_{j-1}|+l}$
 and 
 the second sign is equal to $(-1)^{|t/s(\vec{\sigma})|} \varepsilon_{s,t,\shuffle}$.
\end{theorem}

\begin{proof}
The map $\gamma$ is completely characterized by its image on the operadic tree basis: 
$\gamma_{t(\sigma)}:=\gamma({t(\sigma)})$.
 These structure maps have to
satisfy  the curved Maurer--Cartan Equation~(\ref{eqn:MCHiHoOp}), which evaluated on the operadic tree $t(\sigma)$ gives Equation~(\ref{RelationHoOp}).
Its second term, equal to  $(\gamma\star \gamma) (t(\sigma))$ in $\End_\P$, is given by the infinitesimal decomposition map of the Koszul dual cooperad described in Proposition~\ref{prop:InfDecomp} as follows. One first considers the image of $t(\sigma)$ under $\Delta_{(1)}$, which produces pairs 
$(t/s(\vec{\sigma}),\tilde{s}({\cev{\sigma}}))$
of operadic trees. Then, we compose the associated  structure map 
$\gamma_{t/s(\vec{\sigma})}$
at the entry $s_1$ of the structure map 
$\gamma_{\tilde{s}({\SMALL \cev{\sigma}})}$.
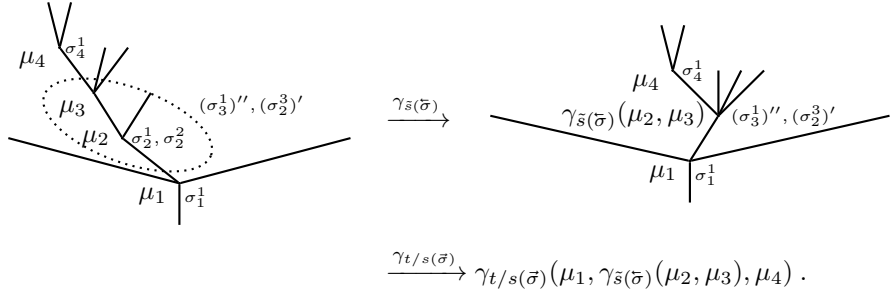
\begin{align*}
  \begin{aligned}\begin{tikzpicture}[optree,
      level distance=12mm,
      level 2/.style={sibling distance=15mm},
      level 3/.style={sibling distance=15mm},
      level 4/.style={sibling distance=6mm}]
    \tiny
    \node{}
      child { node[comp,label={-3:{\tiny$\sigma_1^1$}}]{}     
            node[comp,label={[label distance=1mm]183:{\normalsize$\mu_1$}}]{}
        child
        child { edge from parent[draw=none] }
        child { node[comp] (a) {}
                  node[comp,label={[label distance=1mm]182:{\normalsize$\mu_2$}}]{}
          child { node[comp] (b) {}
                     node[comp,label={[label distance=0.5mm]230:{\normalsize$\mu_3$}}]{}
            child { node[comp,label={right:{\tiny$\sigma_4^1$}}]{}
                        node[comp,label={[label distance=1mm]183:{\normalsize$\mu_4$}}]{}
              child
              child }
            child { edge from parent[draw=none] }
            child
            child }
          child }
        child { edge from parent[draw=none] }
        child { edge from parent[draw=none] }
        child { edge from parent[draw=none] }
        child } ;
    \node[rectangle,right=.5mm of a] (c) {\tiny $\sigma_2^1, \sigma_2^2$};
    \node[draw,dotted,ellipse,rotate=-20,xshift=-1.5mm,yshift=-1mm,inner sep=-4.5mm,fit=(a) (b) (c),label={25:{\tiny $(\sigma_3^1)'', (\sigma_2^3)'$}}] {};
  \end{tikzpicture}\end{aligned}
&\quad  \xrightarrow{\gamma_{\tilde{s}({\tiny \cev{\sigma}})}} \quad 
    \begin{aligned}\begin{tikzpicture}[optree,
      level distance=12mm,
      level 2/.style={sibling distance=15mm},
      level 3/.style={sibling distance=6mm}]
    \tiny
    \node{}
      child { node[comp,label={-3:{\tiny$\sigma_1^1$}}]{}
        node[comp,label={[label distance=1mm]183:{\normalsize$\mu_1$}}]{}
        child
        child { edge from parent[draw=none] }  
        child { edge from parent[draw=none] }        
                child { edge from parent[draw=none] }        
        child { node[comp,label={[label distance=1mm]right:{\tiny $(\sigma_3^1)'', (\sigma_2^3)'$}}]{}
        node[comp,label={[label distance=1mm]left:{\normalsize$\gamma_{\tilde{s}({\SMALL \cev{\sigma}})}(\mu_2, \mu_3)$}}]{}        
          child { node[comp,label={[label distance=0.5mm]2:{\tiny $\sigma_4^1$}}]{}
                                  node[comp,label={[label distance=1mm]183:{\normalsize$\mu_4$}}]{}
            child
            child }
          child { edge from parent[draw=none] }
          child
          child
          child }
        child { edge from parent[draw=none] }
        child { edge from parent[draw=none] }
        child } ;
  \end{tikzpicture}\end{aligned}\\
  & \quad  \xrightarrow{\gamma_{t/s(\vec{\sigma})}}
\gamma_{t/s(\vec{\si})}(\mu_1,  \gamma_{\tilde{s}({\SMALL \cev{\sigma}})}(\mu_2, \mu_3), \mu_4)\ .
  \end{align*}
The second term of Equation~(\ref{RelationHoOp}), equal to $\gamma \circ d_{\calO^{\ac}}(t(\sigma))$,  is given by the formula for the coderivation of the Koszul dual cooperad described in Proposition~\ref{pro:Coder}. Finally, the right-hand side of Equation~(\ref{RelationHoOp}) is the curvature described in Proposition~\ref{pro:Curvature}.
\end{proof}

A morphism of higher homotopy operads is a morphism of dg $\NN$-modules which respects the structure maps. The associated category is denoted by $\mathsf{Op}_\infty$. \\

Since the Koszul dual colored cooperad $q\calO^{\ac}$ is homogenous quadratic, then it is weight graded. So 
the structure maps $\gamma_{t(\sigma)}$ are stratified by the weight of the operadic trees $t(\sigma)$, which is equal to the number of internal edges plus the total number of permutations. 

\begin{proposition}
There is a canonical embedding of categories 
$$\mathsf{nu\ Op} \mono \mathsf{Op}_\infty$$
given by the set of Maurer--Cartan elements which vanish on weight greater or equal than 2.
\end{proposition}

\begin{proof}
In the present case, the weight one structure operations $\gamma_t$ labeled by operadic trees with one internal edge and no permutation give partial compositions $\circ_i$ and the weight one operations $\gamma_{c(\sigma)}$ labeled by corollas with only one permutation give the symmetric groups action. Since $\gamma$ vanishes on operadic trees with 2 internal edges and no permutation, Equation~(\ref{RelationHoOp}) for these types of trees corresponds to the defining relations (iii) and (iv). In the same way, Equation~(\ref{RelationHoOp}) for corollas with 2 permutations corresponds to the defining relation (i). And for trees with 1 internal edge and 1 permutation, it corresponds to the defining relations (v) and (vi). 
\end{proof}

The structure maps $\gamma_{t(\sigma)}$ of weight 2 are the first homotopies for these relations.  And then, each higher stratum correspond a higher level of coherent homotopies for the defining relations of the notion of an nu operad. 

\begin{further}
In the very same way as explained above, 
 since the colored operad $\calO$ which encodes nu operads is set-theoretical, one can look for a topological colored operad whose chain complex would give the resolution $\calO_\infty$. In this case, one would be able to define a suitable notion of a  higher homotopy topological operad. 
\end{further}

\subsection{Comparison with other similar notions}

The Koszul dual colored cooperad $\calO^{\ac}$ actually admits several gradings: arity, number of internal edges, total number of permutations, for instance. Using them, one can get  many full subcategories of the category of higher homotopy operads by requiring that the Maurer--Cartan elements vanish outside some gradings. 
This shows that the above definition of homotopy operads  generalizes the following well-known notions, see Figure~\ref{Fig:VariousCat}. One can go from one category to another by forgetting some structure. In the other way round, one can consider some trivial structure, or, in the case of symmetric groups action, the regular representations. 
\newpage

\begin{figure}[!h]
\rotatebox{90}
{$$\xymatrix@R=10pt@C=10pt@M=4pt{
&&& & & & & & & & &&  \mathsf{Op}_\infty \\
& &&& & & & & & & & {\mathsf{Symmetric} \atop \mathsf{ns\ Op}_\infty }\ar@{^{(}->}[ur]& \\
\infty &&& \mathsf{Mod}_\infty &\ar@{_{(}->}[l] A_\infty\text{-}\mathsf{alg} \ar@{^{(}->}[rrrrrr]& & && & &\mathsf{ns\ Op}_\infty  \ar@{^{(}->}[ur]&  \\
& &&&&&&&&&&& {\mathsf{nu \ Operads} \ \text{in}\atop  \Sy\text{-}\mathsf{Mod}_\infty}\ar@{^{(}->}[uuu]\\
& & & & &&& & && &\mathsf{nu \ Op} \ar@{^{(}->}[ur]\ar@{^{(}->}[uuu]& \\
1\ar@{..}[uuu]&& & \mathsf{Mod}\ar@{_{(}->}[uuu] & \ar@{_{(}->}[l]Ass\text{-}\mathsf{alg}\ar@{_{(}->}[uuu]\ar@{^{(}->}[rrrrrr]& & & & & &
\mathsf{ns\ Op} \ar@{^{(}->}[ur]\ar@{^{(}->}[uuu]& & \\
& &\infty &&& & &\Sy_2\text{-}\mathsf{Mod}_\infty & \Sy_3\text{-}\mathsf{Mod}_\infty& \cdots&& &\Sy\text{-}\mathsf{Mod}_\infty\ar@{^{(}->}[uuu] \\
& 1 \ar@{..}[ur]& & &&& \Sy_2\text{-}\mathsf{Mod} \ar@{^{(}->}[ur]&\Sy_3\text{-}\mathsf{Mod} \ar@{^{(}->}[ur] &\cdots && & \Sy\text{-}\mathsf{Mod}\ar@{^{(}->}[uuu]\ar@{^{(}->}[ur] & \\
 0\ar@{..}[uuu]^{\text{i}}\ar@{..}[ur]^{\text{p}}&& & \mathsf{Vect}^2\ar@{_{(}->}[uuu]&\ar@{_{(}->}[l]\mathsf{Vect} \ar@{_{(}->}[uuu]&\mathsf{Vect} 
\ar@{^{(}->}[ur]&\mathsf{Vect} 
\ar@{^{(}->}[ur] & \cdots& && \NN\text{-}\mathsf{Mod} \ar@{^{(}->}[ur] \ar@{^{(}->}[uuu]& &\\
&&& 0+1 \ar@{..}[r]^{\text{a}}& 1\ar@{..}[r] &2\ar@{..}[r] & 3\ar@{..}[rrrr]  &&&&\infty && } $$}

\noindent
a: arity, i: number of internal edges, p: total number of permutations\caption{The various full subcategories of the category $\mathsf{Op}_\infty$}\label{Fig:VariousCat}
\end{figure}
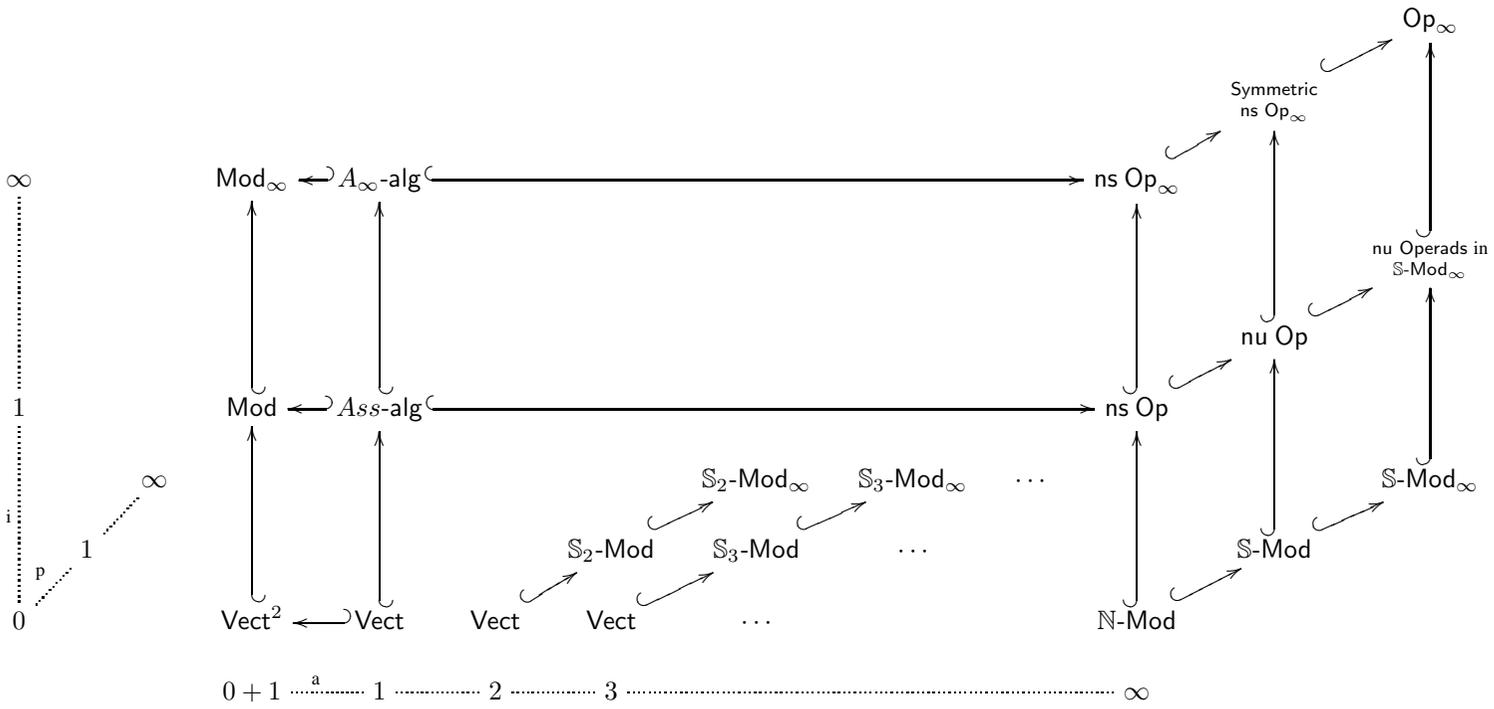

\newpage

\begin{proposition}\label{prop:FullsubCat}
The following categories are canonical full subcategories of the category $\mathsf{Op}_\infty$ of higher homotopy operads. 

  \begin{enumerate}
  \item When $\gamma$ vanishes on operadic trees of arity not equal to $1$, we get the notions of an \emph{$A_\infty$-algebra} of  \cite{Stasheff63} and, respectively, an \emph{associative algebra}, when moreover, it vanishes on operadic trees with more than 2 vertices. 

  \item When $\gamma$ vanishes on operadic trees of arity not equal to $0$ or $1$, we get the notion of a homotopy associative algebra together with a \emph{homotopy left module} over it. When, moreover, it vanishes on operadic trees with more than 2 vertices, we get the strict notion of an associative algebra together with a \emph{left module} over it. 

\item When $\gamma$ vanishes outside corollas of arity $n$, we get a new notion of a \emph{homotopy $\Sy_n$-module}, see Appendix~\ref{subsec:HoGMod} for more details. 
 
    \item When $\gamma$ vanishes 
   everywhere except on the $2$-vertex operadic trees with no permutation, we recover the notion of a \emph{nonsymmetric operad}. When $\gamma$ vanishes outside  operadic trees labeled with no permutation, we recover the notion of a \emph{homotopy nonsymmetric operad} of Van der Laan \cite[\S $4.4$]{VanderLaan03}. 
  
    \item 
   When $\gamma$ vanishes on  operadic trees labeled by at least 2 permutations, we recover the notion of a \emph{symmetric homotopy operad}  of Van der Laan \cite[\S $4.5$]{VanderLaan03}, where the partial operadic composition is laxed up to homotopy by not the action of the symmetric groups.  
 
    \item 
   When $\gamma$ vanishes on  operadic trees with at least 2 internal edges, we get a new notion of an \emph{operad in the category of homotopy $\Sy$-modules}. 
  \end{enumerate}
\end{proposition}

\begin{proof}
This is directly checked. 
\end{proof}

Brinkmeier defined in \cite{Brinkmeier00} a notion called \emph{lax operad} by applying the topological W-con\-stru\-ction of Boardman--Vogt \cite{BoardmanVogt73} to the set-theoretical colored operad which encodes operads (with unit). In this way, he gets a homotopy operad notion where everything, the partial composition, the symmetric group actions and the unit, are relaxed up to homotopy. Let us consider the same method applied to the set-theoretical colored operad $O$ encoding non-unital operads and satisfying $\KK O =\calO$. 
Applying the chain complex functor to $W O$, one gets the bar-cobar construction $\Omega \mathrm{B} \calO$ of the associated linear colored operad $\calO$, according to \cite{BergerMoerdijk06}. So such ``lax operads'' are made up of more structure maps, which are labeled by trees with vertices labeled by operadic trees with one permutation (the ones of Proposition~\ref{prop:OperadicTreesBasis}). 
The natural colored operad map $\Omega \calO^{\ac} \to \Omega \mathrm{B} \calO$ 
induces a functor from  lax operads to the present notion of homotopy operads, which amounts to defining  maps $\gamma_{t(\sigma)}$ by sums of the above structure maps over some set of trees of trees.  
So, finally, the notion of lax operad of \cite{Brinkmeier00} is more complex than the present one in two directions: the unit is encoded and relaxed up to homotopy and there are much more homotopies for the partial compositions and the symmetric group actions. 

\smallskip

There is also a notion of \emph{$\infty$-operad} \cite{MoerdijkWeiss09, Lurie12} due to  Moerdijk--Weiss and Lurie, which is a notion of colored set-theoretical operads in $\infty$-categories. This notion is different from the notion of homotopy operad developed here. However, Le Grignou build recently a homotopy coherent nerve 
between the category of homotopy operads \`a la Van der Laan, which relaxes only the partial composition products up to homotopy, to the category of $\infty$-operads, see \cite{LeGrignou14}. Le Grignou's construction is very likely to extend to the present full notion of higher homotopy operads. 

\subsection{Infinity-morphisms} As usual, the strict notion of morphisms of higher homotopy operads does not  share the suitable homotopy properties. Instead, we consider a more general version called $\infty$-morphism and defined as follows.\\ 

The curved Koszul duality theory provides us with a third equivalent definition of a higher homotopy operad structure on 
a dg $\NN$-module $\P$, see \cite[\S $10.1.8$]{LodayVallette12}. It amounts to  the data of a curved codifferential on the cofree $q\calO^{\ac}$-coalgebra $q\calO^{\ac}(\P)$:
\begin{align*}
\left(
\P, \dif_\P, \gamma
\right)
\stackrel{1\text{-}1}{\longleftrightarrow} 
\left( q\calO^{\ac}(\P), \dif_\gamma:=\dif_{\calO^{\ac}(\P)} + \dif^r_\gamma\right) \ .
\end{align*}
Recall that a morphism of $\calO^{\ac}$-coalgebras is a morphism of $q\calO^{\ac}$-coalgebras which preserves the curved codifferentials. 

\begin{definition}[$\infty$-morphism]
An \emph{$\infty$-morphism} $\P \rightsquigarrow \calQ$ between two higher homotopy operads $(\P, \dif_\P, \gamma)$ and 
$(\calQ, \dif_\calQ, \lambda)$ is a morphism of $\calO^{\ac}$-coalgebras between 
$\left( q\calO^{\ac}(\P),  \dif_\gamma\right)$ and 
$\left( q\calO^{\ac}(\calQ), \dif_\lambda\right)$.
\end{definition}

By definition, $\infty$-morphisms are composable; they form a new category denoted 
$\infty\text{-}\mathsf{Op}_\infty$ in which sits  the category $\mathsf{Op}_\infty$, but which is better behaved with respect to homotopy properties of higher homotopy operads. Notice that all the subcategories of 
$\mathsf{Op}_\infty$,
given in Proposition~\ref{prop:FullsubCat}, are also full subcategories of $\infty\text{-}\mathsf{Op}_\infty$ when they are considered with their respective notions of $\infty$-morphisms. 

\begin{definition}[$\infty$-isomorphism, $\infty$-quasi-isomorphism]
An {$\infty$-morphism} $\P \rightsquigarrow \calQ$ is called an \emph{$\infty$-iso\-mor\-phi\-sm} (resp. an \emph{$\infty$-quasi-isomorphism}) when its first component $\P \to \calQ$ is an isomorphism (resp. a quasi-isomorphism). 
\end{definition}

\begin{proposition}
The class of $\infty$-isomorphisms is the class of the isomorphisms of the category $\infty\text{-}\mathsf{Op}_\infty$. 
\end{proposition}

\begin{proof}
The arguments of \cite[\S $10.4$]{LodayVallette12} apply mutatis mutandis.
\end{proof}

\begin{theorem}[Homotopy Transfer Theorem]
Let $\{\calH(n), \dif_{\calH(n)}\}_{n\in \NN}$
be a homotopy retract of 
$\{\P(n),\allowbreak \dif_{\P(n)}\}_{n\in \NN}$ in the category of dg $\NN$-modules
\begin{align*}
&\xymatrix{     *{ \quad \ \  \quad (\P, \dif_\P)\ } \ar@(dl,ul)[]^{h}\ \ar@<0.5ex>[r]^{p} & *{\
(\calH, \dif_{\calH})\quad \ \  \ \quad }  \ar@<0.5ex>[l]^{i}\ , }\\
&i p- \id_\P =\dif_\P  h+ h  \dif_\P, \ i\  \text{quasi-isomorphism}\ .
\end{align*}
Any higher homotopy operad structure on $\P$ can be transferred into a higher homotopy operad structure on $\calH$ such that the quasi-isomorphism $i$ extends to an $\infty$-quasi-isomorphism.
\end{theorem}

\begin{proof}
The arguments of \cite[\S $10.3$]{LodayVallette12} apply mutatis mutandis. 
\end{proof}

\begin{corollary}
For any $\infty$-quasi-isomorphism $\P \stackrel{\sim}{\rightsquigarrow} \calQ$ of higher homotopy operads, there exists 
$\infty$-quasi-isomorphism $\calQ \stackrel{\sim}{\rightsquigarrow} \P$ which is the inverse of 
$H(\P) \cong H(\calQ)$ on the level of homology. 
\end{corollary}

\begin{proof}
The arguments of \cite[\S $10.4$]{LodayVallette12} apply mutatis mutandis. 
\end{proof}

\subsection{Symmetric homotopy theory of operads}

The general theory of \cite[\S $5.2$]{HirshMilles12} and  \cite[\S $11.4$]{LodayVallette12} allows us to use the Koszul duality for the colored operad $\calO$ to settle new tools to study the homotopy properties of operads over any ring. \\

We first consider the bar-cobar adjunction 
$$\xymatrix@C=30pt{{\B_\iota \ : \ {\mathsf{Op}_\infty}  \ }  \ar@_{->}@<-1ex>[r]  \ar@^{}[r]|(0.38){\perp}   & \ar@_{->}@<-1ex>[l]  {\ \mathsf{conil}\ \mathsf{higher}\ \mathsf{Coop}  \ : \ \Omega_\iota}}   $$
associated to the universal twisting morphism $\iota : \KDO \to \Omega \KDO$. The image $\B_\iota \P$ of a higher homotopy operad is the equivalent definition  mentioned at the end of Section~\ref{subsec:KDOcoalg} in terms of quasi-free $\KDO$-coalgebras. By definition of the $\infty$-morphisms, the functor $\B_\iota$ extends to an isomorphism of categories 
$$\widetilde{\B}_\iota \ : \ \infty\text{-}\mathsf{Op}_\infty \cong 
 \mathsf{quasi}\textsf{-}\mathsf{free}\  \mathsf{higher}\ \mathsf{Coop}\ .$$
All these constructions can be summed up in the following diagram. 
$$\xymatrix@R=30pt@C=50pt@M=7pt{
{\mathsf{nu}\  \mathsf{Op} \cong\mathsf{aug}\  \mathsf{Op}}\  \ar@<+0.5ex>@^{^{(}->}[d]  \ar@_{->}@<-0.5ex>[r]_{\widetilde{\B}}
&
 \ar@<-0.5ex>@_{->}[l]_{\widetilde{\Omega}} \ \mathsf{conil}\ \mathsf{higher}\ \mathsf{Coop} \ar@<-0.5ex>@_{->}[dl]_(0.49){\Omega_\iota} 
 \\
\mathsf{Op}_\infty \ar@_{->}@<-0.5ex>[ur]_(0.45){\B_\iota}  \ar@<+0.5ex>@^{^{(}->}[d]   
&
 \\
\infty\mathsf{-Op}_\infty \ar[r]^(0.43)\cong_(0.43){\widetilde{\B}_\iota} \ar@/^2pc/@{-^{>}}[uu]^{\widetilde{\Omega} \widetilde{\B}_\iota} 
&
  \mathsf{quasi}\textsf{-}\mathsf{free}\ \mathsf{higher}\ \mathsf{Coop} \ar@^{^{(}->}[uu] }$$
The left column is made up of a pair of adjoint functors 
$$\xymatrix@C=30pt{{i \ :\ \mathsf{nu \ Op}  \ }  \ar@_{->}@<-1ex>[r]  \ar@^{}[r]|(0.42){\perp}   & \ar@_{->}@<-1ex>[l]  {\ \infty\text{-}{\mathsf{Op}_\infty\ : \ \widetilde{\Omega} \widetilde{\B}_\iota}}} \ ,  $$
which satisfy the following property. 

\begin{theorem}[Rectification]
Any higher homotopy operad $(\P, \dif_\P, \gamma)$ is naturally $\infty$-quasi-isomorphic to the dg nu operad  $\widetilde{\Omega} {\B}_\iota\, \P$ 
$$\P \stackrel{\sim}{\rightsquigarrow}  \widetilde{\Omega} {\B}_\iota\,  \P\ .$$
\end{theorem}

\begin{proof}
The arguments of \cite{LodayVallette12, HirshMilles12} apply mutatis mutandis. 
\end{proof}

Such rectification is unique up to unique $\infty$-isomorphism, see \cite[Proposition~$11.4.6$]{LodayVallette12}.

\begin{proposition}[Homotopy properties]\leavevmode
\begin{itemize}
\item[$\diamond$] For any pair of dg augmented operads $\P$ and $\calQ$, there exists a zig-zag of quasi-isomorphisms 
$$\P  \stackrel{\sim}{\leftarrow} \bullet \stackrel{\sim}{\to} \bullet 
\stackrel{\sim}{\leftarrow} \bullet  \cdots
 \bullet \stackrel{\sim}{\to}
\calQ $$
if and only if there exists an $\infty$-quasi-isomorphism 
$$\P   \stackrel{\sim}{\rightsquigarrow} \calQ \ .$$
\item[$\diamond$]
The homotopy category of dg augmented operads and the homotopy category of homotopy operads are equivalent
$$ \mathsf{Ho}(\mathsf{dg \ aug \ Op}) \cong 
 \mathsf{Ho}(\infty\text{-}\mathsf{Op}_\infty)\ .$$
\end{itemize}
\end{proposition}

\begin{proof}
The arguments of \cite{LodayVallette12, HirshMilles12} apply mutatis mutandis. 
\end{proof}

\begin{further}
At this stage, one would like to go one step further, that is to write these two localized categories as a ``honest'' category. Following \cite{LefevreHasegawa03, Vallette14}, one should be able to prove an equivalence like
$$\mathsf{Ho}(\mathsf{dg \ aug \ Op}) \cong \mathsf{\infty}\textsf{-}\mathsf{Op}_\infty/\sim_h\ ,$$
with a good notion of homotopy equivalence of $\infty$-morphisms. There are two  main conceptual difficulties here since one would have to endow the category of conilpotent $\KDO$-coalgebras with a suitable model category structure: the underlying objects are not chain complexes, the ``differential'' map does not square to $0$, and one will have to deal with the general homotopy properties of modules over a ring, not a field.  
\end{further}


\appendix 

\section{Homotopy group representation}\label{subsec:HoGMod}
The me\-th\-od used in this paper can actually treat an interesting question: how can one define a relaxed notion up to homotopy of a module over a group with good properties. Let $(\mathbb{G}, ., e)$ be a group; a representation of $\mathbb{G}$ is a $\KK[\GG]$-module. So one can use the Koszul duality theory to find resolutions of the group algebra $\KK[\GG]$. 
Any presentation of the group $\GG$ gives a presentation of the group algebra $\KK[\GG]$ that one can use to get ``small'' resolutions. Without any particular presentation, one can only consider the trivial presentation where the set of generators is the full group itself. At that point, there are two ways  to apply the Koszul duality theory; we show that they are equivalent. \\

First, one can consider the presentation where the set of generators $\oGG$ is the group without the unit, that is 
$$\KK[\GG]\cong T(\oGG)/(g \otimes g^{-1} -1, g\otimes h - gh)\ , \quad \text{when}\quad  gh\neq e\ .$$
This presentation is semi-augmented according to the terminology of \cite[\S $3.3.1$]{HirshMilles12}, this means that the augmentation map is not required to be an algebra morphism. Then, the curved Koszul duality theory of loc. cit. applies; that is what we did in the case of the symmetric groups $\GG=\Sy_n$ in the arity 1 part of the colored operad $\calO$. 
\begin{lemma}
The group algebra $\KK[\GG]$ with this quadratic-linear-constant presentation is curved Koszul. 
\end{lemma}
\begin{proof}
Since the above quadratic-linear-constant presentation forms a Gr\"obner basis of the group algebra, Conditions   (I)-(II) hold and the algebra is Koszul. 
\end{proof}
In this case, the Koszul dual curved coalgebra is given by the cofree coalgebra 
$$\left(T^c(s\oGG), d, \theta
\right)\ ,$$
where 
$$d(g_1\otimes \cdots \otimes g_n)=\sum_{i=1}^{n-1} (-1)^{i-1}g_1\otimes \cdots  \otimes \underbrace{g_ig_{i+1}}_{\neq e}\otimes  \cdots  \otimes g_n$$ 
and $\theta(g \otimes g^{-1})=1$. 
\begin{proposition}
The cobar construction of the Koszul dual curved coalgebra forms a cofibrant resolution of the group algebra 
$$\Omega_1:=\left( T(s^{-1}\overline T^c (s\oGG), d_2+d_1-d_0
\right) \xrightarrow{\sim} \KK[\GG]\ .$$
\end{proposition}

On the other hand, one can consider the augmentation 
\begin{align*}
\left\{
\begin{array}{lcll}
\varepsilon \ : &\KK[\GG] &\to &\KK\\
& g & \mapsto & 1
\end{array}
\right.
\end{align*}
which induces the splitting  $\KK[\GG]\cong \KK \oplus \Ker\, \varepsilon$. The augmentation ideal  admits $\{Ê\tilde{g}:=1-g, \ g \in \oGG\}$ for basis;  we can use it to provide the group algebra $\KK[\GG]$ with a quadratic-linear presentation:
$$\KK[\GG]\cong T(\oGG)/\left(\tig \otimes \widetilde{g^{-1}} -\tig- \widetilde{g^{-1}}, 
\tig\otimes \tih - \tig - \tih + \widetilde{gh}\right)\ , \quad \text{when}\quad  gh\neq e\ .$$
Such formulas using $\tig$ instead of group elements $g$ can be interpreted as some kind of ``linearization'' of the group $\GG$. The first presentation of the group algebra is set-theoretical, this second one is not.
\begin{lemma}
The group algebra $\KK[\GG]$ with this quadratic-linear presentation is inhomogeneous Koszul. 
\end{lemma}
\begin{proof}
Again this presentation provides the group algebra with a Gr\"obner basis. Hence it  proves Conditions   (I)-(II) and that this presentation is inhomogeneous Koszul.
\end{proof}
In this case, the Koszul dual the a dg coalgebra
$$\left(T^c(s\oGG), d_\varphi \right)\ ,$$
where 
$$d_\varphi(\tig_1\otimes \cdots \otimes \tig_n)=\sum_{i=1}^{n-1} (-1)^{i-1}\tig_1\otimes \cdots  \otimes 
(\tig_i + \tig_{i+1}-\widetilde{g_ig_{i+1}})
\otimes  \cdots  \otimes \tig_n\ .$$ 
\begin{proposition}
The cobar construction of the Koszul dual dg coalgebra forms a cofibrant resolution of the group algebra 
$$\Omega_2:=\left( T(s^{-1}\overline T^c (s\oGG), d'_2+d'_1
\right) \xrightarrow{\sim} \KK[\GG]\ .$$
\end{proposition}

So both $\Omega_1$-algebras and $\Omega_2$-algebras define a notion of homotopy $\GG$-module with the required homotopy properties: $\infty$-morphisms, Homotopy Transfer Theorem, description of the homotopy category of $\GG$-modules, for instance. By abstract nonsense, the two cofibrant resolutions $\Omega_1$ and $\Omega_2$ are quasi-isomorphic so these two notions of homotopy $\GG$-modules are homotopy equivalent. But, in this case, one can prove the following stronger result. 

\begin{theorem}
The two  resolutions $\Omega_1$ and $\Omega_2$ of the group algebra $\KK[\GG]$ are isomorphic, which makes  the two associated notions of homotopy $\GG$-modules isomorphic. 
\end{theorem}

\begin{proof}
One can consider the isomorphism $\Omega_2 \to \Omega_1$ defined on the generators by 
$$ 
\left\{
\begin{array}{llll}
\tig &\mapsto& 1-g \ , & \\
s^{-1}(s\tig_1\otimes \cdots \otimes s\tig_n)& \mapsto & (-1)^n\ s^{-1}(s g_1\otimes \cdots \otimes s g_n) 
& \text{for} \quad n\ge 2\ ,
\end{array}
\right.
$$
with the reverse isomorphism  $\Omega_1 \to \Omega_2$ given by 
$$ 
\left\{
\begin{array}{llll}
g &\mapsto& 1-\tig \ , & \\
s^{-1}(s g_1\otimes \cdots \otimes s g_n)& \mapsto & (-1)^n\ s^{-1}(s \tig_1\otimes \cdots \otimes s \tig_n) 
& \text{for} \quad n\ge 2\ .
\end{array}
\right.
$$
\end{proof}

We choose the first resolution to define the homotopy notion of $\GG$-module 
\begin{definition}[$\GG_\infty$-module]
A \emph{homotopy $\GG$-module}, or \emph{$\GG_\infty$-module}, is chain complex $(V,d)$ equipped with 
degree $n-1$ maps 
$$\gamma(g_1, \ldots, g_n) : V \to V\ ,$$ 
for any $n$-tuples $(g_1, \ldots, g_n)\in \oGG^n$, satisfying 
\begin{align*}
&\sum_{i=1}^{n-1}(-1)^{i-1}\gamma(g_1, \ldots, g_i) \circ \gamma(g_{i+1}, \ldots, g_n)
- \sum_{i=1}^{n-1} (-1)^{i-1} \gamma(g_1, \ldots, \underbrace{g_ig_{i+1}}_{\neq e}, g_n)
-\partial(\gamma(g_1, \ldots, g_n))=\\
&
\left\{
\begin{array}{lll}
\id_V & \text{when} & (g_1, \ldots, g_n)= (g, g^{-1})\ ,\\
 0&  \text{otherwise}\ .&
 \end{array}
 \right. 
\end{align*}
\end{definition}

A notion of ``representation up to homotopy'' was introduced in \cite{AriasAbadCrainicDherin11} as follows. Instead of considering algebra morphisms $\KK[\GG] \to \Hom(V,V)$, one can consider the wider class of $A_\infty$-morphisms. Such a data is actually equivalent to an algebra morphism from the bar-cobar construction of $\KK[\GG]$, as explained in \cite[\S $10.5.5$]{LodayVallette12}. Notice that the Koszul resolution obtained from the trivial presentation with all the elements for generators is the bar-cobar construction. 
Hence the formulas in loc. cit. are similar to the one given here, except that the specific role of the group unit is not taken into account there. One obtains the definition of \cite{AriasAbadCrainicDherin11} by applying the bar-cobar construction to the group algebra  after forgetting its unit. The interpretation in terms of Maurer--Cartan elements of \cite[\S 3]{AriasAbadCrainicDherin11} follows then directly from the general theory \cite[Chapter~$2$]{LodayVallette12}.

\begin{remark}
In the present paper, we could have chosen to work with the linearized presentation of the symmetric group algebras. Only Relations (i)-(ii) would change, not Relations (v)-(vi), which would remain exactly the same with $\tilde \sigma$ instead of $\sigma$. So all the results would still hold with nearly the same formulas and we would get a cofibrant resolution isomorphic to the one given in Theorem~\ref{thm:OKoszul}.
With that approach, we would get a Koszul dual dg colored cooperad instead of a curved colored cooperad but the underlying colored cooperad would be the same, the Maurer--Cartan equation encoding homotopy operads would have no curvature terms and only the third term on the left-hand side of Relation~(\ref{RelationHoOp}) would be modified using the above formula for $d_\varphi$. However, the (homotopy) action of the symmetric groups would be less transparent, less set-theoretical: it is the linearized action that would be relaxed up to homotopy.
\end{remark}

\bibliographystyle{alpha}
\bibliography{bib}

\end{document}